\documentclass{article}
\usepackage{fancyhdr}

\usepackage[british]{babel}
\usepackage[letterpaper,top=2cm,bottom=2cm,left=3cm,right=3cm,marginparwidth=1.75cm]{geometry}
\usepackage{amsmath, amsbsy, amssymb, amsthm}
\usepackage{graphicx}
\usepackage{xspace}
\usepackage[T1]{fontenc}
\usepackage[utf8]{inputenc}
\usepackage{float}
\usepackage{caption}
\usepackage{quiver}
\usepackage{yhmath}
\usepackage[nottoc,numbib]{tocbibind}
\usepackage{scalerel}
\usepackage{footmisc}
\usepackage[useregional]{datetime2}
\usepackage[colorlinks=true, allcolors=blue]{hyperref}
\usepackage[maxbibnames=99, style=alphabetic, backend=biber, url=false, doi=false, isbn=false, date=year, hyperref, backref, backrefstyle=none]{biblatex}
\DefineBibliographyStrings{english}{
  backrefpage={$\uparrow$\!\!},
  backrefpages={$\uparrow$\!\!}
}
\usepackage{filecontents}
\DTMlangsetup[en-GB]{showdayofmonth=false}
\usepackage{tikz}
\newcommand*\circled[1]{\tikz[baseline=(char.base)]{
    \node[shape=circle,draw,inner sep=0.5pt] (char) {#1};}}
\usepackage{mathtools}

\newtheorem{proposition}{Proposition}[section]
\newtheorem{theorem}[proposition]{Theorem}
\newtheorem{lemma}[proposition]{Lemma}
\newtheorem{example}[proposition]{Example}
\newtheorem{definition}[proposition]{Definition}
\newtheorem{definition-lemma}[proposition]{Definition-Lemma}
\newtheorem{remark}[proposition]{Remark}
\newtheorem{statement}[proposition]{Statement}
\newtheorem{corollary}[proposition]{Corollary}
\newtheorem{question}[proposition]{Question}
\newtheorem*{question*}{Questions}
\newtheorem*{lemma*}{Lemma}

\newtheorem{conjecture}[proposition]{Conjecture}

\numberwithin{equation}{proposition}

\makeatletter
\newcommand{\etale}{\'etal\@ifstar{\'e}{e\xspace}}
\makeatother
\newcommand\blfootnote[1]{%
  \begingroup
  \renewcommand\thefootnote{}\footnote{#1}%
  \addtocounter{footnote}{-1}%
  \endgroup
}
\newcommand{\Addresses}{{
  \bigskip
\noindent\textsc{Faculty of Mathematics, University of Regensburg, Germany,}\par\nopagebreak
\noindent Email: \texttt{tong.g.h.zhou@gmail.com}
  }}

\newcommand{\isoto}{\xrightarrow{\raisebox{-0.5ex}[0ex][0ex]{$\sim$}}}
\newcommand{\isoot}{\xleftarrow{\raisebox{-0.5ex}[0ex][0ex]{$\sim$}}}

\newcommand{\RR}{\mathbf{R}}

\newcommand{\ZZ}{\mathbf{Z}}

\newcommand{\NN}{\mathbf{N}}

\newcommand{\FF}{\mathbf{F}}
\newcommand{\PP}{\mathbf{P}}
\newcommand{\CF}{\mathcal{F}}
\newcommand{\CG}{\mathcal{G}}

\newcommand{\CB}{\mathcal{B}}
\newcommand{\CalC}{\mathcal{C}}
\newcommand{\Gm}{\mathbf{G}_m}
\newcommand{\X}{\times}

\newcommand{\CCC}{\mathcal{C}}
\newcommand{\CH}{\mathcal{H}}
\newcommand{\CL}{\mathcal{L}}

\newcommand{\AAA}{\mathbf{A}}
\newcommand{\DD}{\mathbb{D}}

\newcommand{\CA}{\mathcal{A}}
\newcommand{\CO}{\mathcal{O}}

\newcommand{\RHom}{R\underline{Hom}}

\newcommand{\Spa}{\mathrm{Spa}}
\newcommand{\DX}{D^{(b)}_{zc}(X)}
\newcommand{\Dbzc}{D^{(b)}_{zc}}

\newcommand{\CX}{\mathcal{X}}
\newcommand{\dotF}{\{\dddot\CF\}}
\newcommand{\tE}{\widetilde{E}}
\newcommand{\F}{\mathrm{F}}
\newcommand{\K}{\mathrm{K}}
\newcommand{\Bl}{\mathrm{Bl}}

\newcommand{\cone}{\mathrm{cone}}

\newcommand{\Sh}{\mathrm{Sh}}

\newcommand{\miwa}{(-1)^\tau}
\newcommand{\supp}{\mathrm{supp}}

\newcommand{\pr}{\mathrm{pr}}
\newcommand{\id}{\mathrm{id}}
\newcommand{\adj}{\mathrm{adj}}
\newcommand{\pt}{\mathrm{pt}}
\newcommand{\cDbzc}{\mathcal{D}^{(b)}_{zc}}
\newcommand{\cDb}{\mathcal{D}^{(b)}}
\newcommand{\prbar}{\overline{\pr}}
\newcommand{\et}{{\mathrm{\acute{e}t}}}

\title{A Microlocal Theory for Zariski-Constructible Sheaves on Rigid Analytic Varieties}
\author{Tong Zhou}
\date{}
\begin{document}
\maketitle
\begin{abstract}
\blfootnote{August 2026}
We develop a microlocal theory, in the sense of Kashiwara-Schapira, for Zariski-constructible sheaves on rigid analytic varieties. We define and study monodromic sheaves, the monodromic Fourier transform, specialisation, microlocalisation, micro-hom, and singular support in this context. Some questions and conjectures are formulated in the end. The appendix contains infinity-categorical characterisations of monodromic sheaves.
\end{abstract}
\renewcommand{\baselinestretch}{1.0}\normalsize
\setcounter{tocdepth}{1}
\tableofcontents
\renewcommand{\baselinestretch}{1.0}\normalsize
\section{Introduction}\label{sec_intro}
In the 1960s, Mikio Sato, when studying partial differential equations, introduced the microlocal point of view, which systematically studies objects on a manifold via constructions involving the cotangent bundle. This point of view spread to other fields. In the 1970s–90s, Masaki Kashiwara and Pierre Schapira developed a microlocal theory for sheaves on real and complex manifolds, and studied their relation to D-modules (\cite{kashiwara_sheaves_1990}). It has since had applications to partial differential equations, symplectic geometry, geometric representation theory, geometric Langlands, and other fields.\\

On real manifolds, the basic constructions of microlocal sheaf theory are as follows. Let $Z\hookrightarrow X$ be a closed immersion of manifolds, and consider $\RR$-valued sheaves for definiteness. Recall that a sheaf on a vector bundle is called \underline{conic} if it is constant on each $\RR_{>0}$-orbit. (1) The specialisation $\nu_Z(-)$ is a functor from sheaves on $X$ to conic sheaves on the normal bundle $T_ZX$. It is a version of nearby cycles for higher codimensions and without the choice of a function. (2) The microlocalisation $\mu_Z(-)$ is a functor from sheaves on $X$ to conic sheaves on the conormal bundle $T^*_ZX$, defined as the Fourier-Sato transform of $\nu_Z$. It is a version of vanishing cycles for higher codimensions and without the choice of a function. (3) The micro-hom $\mu hom(-,-)$ is a functor from a pair of sheaves on $X$ to conic sheaves on the cotangent bundle $T^*X$. It is a sheaf-theoretic analogue of the micro-hom of D-modules. (4) For every sheaf $\CF$ on $X$, the singular support $SS(\CF)$ is a closed conical subset in $T^*X$, measuring the directions in which $\CF$ is not locally constant. (5) $\mu sh$ is a stack on $T^*X$, which is a sheaf-theoretic analogue of microdifferential-modules on $T^*X$. (6) For every constructible sheaf $\CF$ on $X$, the characteristic cycle $CC(\CF)$ is a $\ZZ$-cycle supported on $SS(\CF)$, with coefficients equal to the “local Euler characteristics” of $\CF$. There are various relations among these constructions, most notably, we have $\mu hom(\underline{\RR}_Z,-)\simeq\mu_Z(-)$, $SS(\CF)=\supp\,\mu hom(\CF,\CF)$, and $H^0(\mu hom(\CG,\CF)_v)\simeq Hom_{\mu sh_v}(\CG,\CF)$ for $v\in T^*X$. The last relation is crucial in Kashiwara and Schapira's proof of the theorem of coisotropicity of the singular support\textemdash a theorem of central importance. These together make precise that a sheaf $\CF$ on $X$ gives a “microlocal-sheaf $\mu\CF$” on $T^*X$: $SS(\CF)$ is the support of $\mu\CF$ and $\mu hom$’s are homomorphisms between them. This is the sheaf-theoretic analogue of obtaining a microdifferential-module on $T^*X$ from a D-module on $X$, and has precise relations to the constructions on the D-module side through Riemann-Hilbert.\\

A natural question is if one can develop microlocal theories in other sheaf-theoretic contexts. For $\ell$-adic \etale sheaves on algebraic varieties over a field, a complete theory of $SS$ and $CC$ has been developed by Alexander Beilinson and Takeshi Saito for finite coefficients (\cite{beilinson_constructible_2016}, \cite{saito_characteristic_2017}), and extended to $\ell$-adic coefficients in \cite{barrett_singular_2023} and \cite{umezaki_characteristic_2020}. No satisfactory theory of the specialisation, microlocalisation, micro-hom and $\mu sh$ exists so far, due to the difficulty of wild ramifications in positive characteristic (the specialisation defined in \cite{verdier_specialisation_1983} does not see wild ramifications, \textit{cf.} \cite{abbes_microlocal_2006}).\\

In this work, we develop a microlocal theory for $\Lambda$-valued Zariski-constructible sheaves on rigid analytic varieties over $K$, where $K$ is a non-archimedean field which is non-trivially valued, complete, algebraically closed, of characteristic $0$ and residue characteristic $p\geq0$, and $\Lambda$ is either $\FF_{\ell^r}$ or $\ZZ/\ell^r$, $p\neq\ell$. This is based on the six-functor formalism of \cite{bhatt_six_2022} and a theory of nearby and vanishing cycles developed in \cite{zhou_vanishing_2025} (reviewed in §\ref{sec_defofvan}).\\

In §\ref{sec_monsheav} and §\ref{sec_fourier}, as preparation for later sections, we study the analogues of conic sheaves and the Fourier-Sato transform: monodromic sheaves ($D_{mon}$)(\cite{verdier_specialisation_1983}) and the monodromic Fourier transform ($\F$) (\cite{wang_new_2015}). We give several equivalent definitions of $D_{mon}$, show it is compatible with the usual and perverse t-structures, study its functoriality, show the $\Gm$-contraction lemma, study in detail the case of a rank-$1$ trivial vector bundle, and show that $\F$ is perverse t-exact and commutes with duality (after shifts and twists). In the appendix, we give some $\infty$-categorical characterisations of monodromic sheaves suggested by Bhargav Bhatt.\\

In §\ref{sec_special}, §\ref{sec_microlocal} and §\ref{sec_muhom}, we define and study the basic properties of the specialisation ($\nu_Z$), microlocalisation ($\mu_Z$) and micro-hom ($\mu hom$). We show that $\nu_Z$ (hence $\mu_Z$ and $\mu hom$, by definition) lands in monodromic sheaves, study their functoriality, show the perverse t-exactness and commutativity with duality (after shifts and twists), and study in detail the case when $Z$ is a hypersurface.\\

We highlight the following results (see \cite[\nopp 8.6.3, 4.4.3]{kashiwara_sheaves_1990} for the classical analogues), which make precise that the specialisation and microlocalisation generalise nearby and vanishing cycles, and $\mu hom$ generalises the microlocalisation. They also make precise that $\mu hom$ can be viewed as a replacement in the analytic world of the vanishing cycle with general bases, with two notably differences: (1) it is most useful when one maps testing sheaves into the sheaf to be understood (Proposition \ref{prop_1.1}.3, Lemma \ref{lem_muhom_functoriality}), (2) it only records information of “\emph{linear} specialisations”. Nevertheless, it is expected that, in our context, which does not have local wild ramifications, linear information suffices for microlocal purposes, and there should be a fairly complete microlocal theory similar to that in the real or complex setting.

\begin{proposition}[Proposition \ref{lem_nu&psi}, Proposition \ref{lem_microlocalisation_is_vanishingcycle}, Lemma \ref{lem_muhom_is_microlocalisation}]\label{prop_1.1}
    Let $Z=V(f)\hookrightarrow X$ be a smooth hypersurface in a quasi-separated smooth rigid analytic variety, and $\CF\in\DX$. $f: X\rightarrow\AAA^1$ induces a map $\widetilde{f}: T_ZX\rightarrow T_0\AAA^1\simeq\AAA^1$, which induces an isomorphism $T_ZX\isoto Z\times\AAA^1$. Denote the section $Z\hookrightarrow T_ZX\simeq Z\times\AAA^1$, $z\mapsto(z,1)$ by $s_f$, and the section $Z\hookrightarrow T^*_ZX$, $z\mapsto (z,df_z)$ by $s'_f$. Then:\\
    (1) there is a canonical isomorphism $s_f^*\nu_Z(\CF)\isoto\psi_f(\CF)$, under which the equivariant monodromy of $s_f^*\nu_Z(\CF)$ coincides with the opposite of the monodromy of $\psi_f(\CF)$;\\
    (2) there is an isomorphism\footnote{We only assert the existence of an isomorphism here, as we work with triangulated categories except in the appendix, and we took cones in the proof of (2) (Proposition \ref{lem_rank1monbasic}.2).} $s'^*_f(\mu_Z(\CF))\cong\Phi_f(\CF)(-1)$, under which the equivariant monodromy of $s'^*_f(\mu_Z(\CF))$ coincides with the usual monodromy of $\Phi_f(\CF)(-1)$;\\
    (3) there is a canonical isomorphism $\mu hom(i_*\underline{\Lambda}_Z,\CF)\simeq\widetilde{i}_*\mu_Z(\CF)$, where $\widetilde{i}$ denotes the closed immersion $T^*_ZX\hookrightarrow T^*X$.
\end{proposition}

In §\ref{sec_SS}, we define and study the singular support ($SS$). Here is a summary of our main results:

\begin{theorem}
    Let $X$ be a smooth rigid analytic variety, and $\CF\in\Dbzc(X)$. Assume $\Lambda=\FF_{\ell^r}$. Then: \\
    (1) $SS(\CF)$ is a closed conical subset of $T^*X$, whose base equals the support of $\CF$.\\
    (2) Let $p: Y\rightarrow X$ (resp. $g: X\rightarrow Z$) be a smooth (resp. proper) map of smooth rigid analytic varieties, and $\CF\in \DX$. Then $p^\circ SS(\CF)\subseteq SS(p^*\CF)$ (resp. $SS(g_*\CF)\subseteq g_\circ SS(\CF)$). If $p$ is \etale (resp. $g$ is a closed immersion), then equality holds.\\
    (3) If ($X$, $\CF$) is the analytification of ($\mathcal{X}$, $\CG$), where $\CX$ is a smooth finite type scheme over $K$ and $\CG\in D^b_c(\CX)$, then $(SS(\CG))^{an}\subseteq SS(\CG^{an})$.\\
    (4) $SS(\CF)=SS(\DD\CF)$.\\
    (5) If $\CF$ has finitely many irreducible perverse constituents $\{\CF_\alpha\}$, then $SS(\CF)=\cup SS(\CF_\alpha)$.\\
    (6) $SS(\CF)$ equals the $0$-section $T_X^*X$ if and only if $\CF$ is a non-zero local system.\\
    (7) Let $H$ be a simple normal crossings divisor, \textit{i.e.}, $H=H_1\cup H_2\cup...\cup H_r$ where $H_i$ are smooth divisors such that for each subset $I\subseteq \{1,2,...,r\}$, we have $H_I:=\cap_{i\in I}H_i$ is smooth. Let $\CL$ be a non-zero local system on $X-H$, and $j: X-H\hookrightarrow X$ be the open immersion. Then $SS(j_{!*}\CL)\subseteq SS(j_!\CL)=SS(j_*\CL)=\cup T^*_{H_I}X$, where $I$ ranges through subsets of $\{1,2,...,r\}$ (for $I=\varnothing$, $T^*_{H_I}X:=T_X^*X$).
\end{theorem}

\begin{remark}\label{rmk_to_thm_1.2}
    (1) One expects that $SS(\CF)$ is an analytic closed subset, $p^\circ SS(\CF)\subseteq SS(p^*\CF)$ in (2), and $(SS(\CG))^{an}= SS(\CF)$ in (3). These are discussed in §\ref{sec_open}.\\
    (2) Note that, in (3), for a general $\CG\in D^{+}(\CX)$, if $\CG^{an}$ is Zariski-constructible, then $SS(\CG^{an})$ usually contains more information than $SS(\CG)$.\footnote{We remark that $SS(\CG)$ can be defined for all $\CG\in D^{+}(\CX)$: Beilinson worked with constructible sheaves in \cite{beilinson_constructible_2016} but his definition and method apply to all sheaves $
    \CG\in D^{+}(\CX)$. In fact, one can write $
    \CG\in D^{+}(\CX)$ as a filtered colimit of $\CG_\alpha\in D^b_c(\CX)$, then show $\CG$ is micro-supported on $\cup_\alpha SS(\CG_\alpha)$ and the analogues of \cite[\nopp 1.6.1.(i), 3.3]{beilinson_constructible_2016}. One may also show $\supp(\CG)$ is contained in the base of every closed conical $C\in T^*\CX$ on which $\CG$ is micro-supported. Then, one can show 2.1.(ii), 2.1.(iv), 2.2, 2.3, 2.5, 3.9.(i) as in \cite{beilinson_constructible_2016}. The existence of $SS(\CG)$ then follows from these and the argument in \cite[\nopp 3.5]{beilinson_constructible_2016}. Note that the main result of \cite{abbes_holonomic_2024}, which is the algebraic analogue of \cite[\nopp 8.4.2, 8.5.5]{kashiwara_sheaves_1990}, can then be phrased as: for $\CX$ a connected smooth algebraic variety and $\CG\in D^{+}(\CX)$, we have $\dim SS(\CG)\leq\dim \CX$ if and only if $\dim SS(\CG)=\dim \CX$ if and only if $\CG$ lies in $D^b_c(\CX)$. Here $\dim$ denotes the maximum of the dimensions of irreducible components.} For example, consider $\CX=\AAA_K^1$, and $\CG=\oplus\underline{\Lambda}_{a_i}$, where $\{a_i\}$ is a subset of $K$ which is discrete in $X=\CX^{an}$. Then $SS(\CG)=T^*\CX$ while $SS(\CG^{an})=\cup T^*_{a_i}X$.
\end{remark}

We note a distinction between the real or complex analytic and the rigid analytic settings: the Microlocal Morse Lemma \cite[\nopp 5.4.19]{kashiwara_sheaves_1990} fails in the rigid world. See Remark \ref{rmk_failure_Morse}.\\

In §\ref{sec_open}, we discuss some questions and conjecture some deeper properties of the microlocal constructions, which, together with $CC$, are subjects for future work. 
 
\section*{Conventions}
We use the six-functor formalism of \etale sheaves on adic spaces constructed in \cite{huber_etale_1996}. The derived categories $D(X,\Lambda)$ are in the triangulated sense. $\Sh(X,\Lambda)$ denotes the abelian category of sheaves. The coefficient $\Lambda$ is often omitted. All sheaf-theoretic functors are derived. By a “local system” $\CL$ on $X$ we mean an object of $\DX$ such that each $\CH^i(\CL)$ is locally constant with finite type stalks. For $\CF\in D(X)$ and $d\in\ZZ$, we use $\CF\langle d\rangle$ to denote $\CF(d)[2d]$, where $(-)$ (resp. $[-]$) denotes the Tate twist (resp. shift). \\\\
By a “rigid analytic variety” we mean an adic space locally topologically of finite type over $\Spa(K,K^\circ)$. We use the terms “analytic closed subset” and “Zariski-closed subset” interchangeably. For a Huber ring $A$ topologically of finite type over $K$, we abbreviate $\Spa(A,A^\circ)$ as $\Spa(A)$. For $r\in|K^\X|$, $D(r):=\Spa(K\langle \frac{T}{r}\rangle)$ denotes the closed unit disc of radius $r$, $D(1)$ is abbreviated as $D$, punctured discs $D(r)-\{0\}$ are denoted by $D^\X(r)$. By a “classical point” of $X$ we mean a $K$-point, \textit{i.e.}, an element of $X(K)$. For $\lambda\in K^\X$ and $n\in\ZZ_{\geq1}$, $\mu_n(\lambda)$ denotes the $n$-th roots of $\lambda$ in $K$. We abbreviate $\mu_n(1)$ by $\mu_n$, and denote the profinite group $\varprojlim\mu_n(1)$ by $\mu$. The spaces $\Gm, \AAA^1$ and fibre products are over $K$ unless otherwise specified by subscripts.\\\\
We use “$\simeq$” to denote a canonical isomorphism, with the morphism clear from the context, and “$\cong$” to denote an isomorphism without naturality or canonicity assertions.


\section*{Acknowledgements}
I am very grateful to Bhargav Bhatt for numerous valuable discussions, especially for showing me Lemma \ref{lem_glue_morph}, Lemma \ref{lem_max_liss_open_is_zariski_open}, and suggesting Definition \ref{def_infinity_monodromic}. I sincerely thank Grigory Andreychev, Denis-Charles Cisinski, Yash Deshmukh, Bradley Dirks, David Hansen, David Nadler and Bogdan Zavyalov for their interest and suggestions. I also thank Takeshi Saito very much for discussion related to the footnote to Remark \ref{rmk_to_thm_1.2}.2, and Tomoyuki Abe for encouraging me to remove smoothness assumptions in §\ref{sec_monsheav} and §\ref{sec_special} in an original version of this article. This work is done during my stay at the Institute for Advanced Study, and the manuscript is completed at the University of Regensburg where I am a Humboldt Research Fellow. I want to thank the Institute and Regensburg for providing beautiful environments and the Humboldt Foundation for its support.

\section{Nearby and vanishing cycles}\label{sec_defofvan}
Nearby and vanishing cycles serve as the foundation for defining the specialisation, microlocalisation and other microlocal functors. In this section, we recall the theory from \cite{zhou_vanishing_2025} of nearby and vanishing cycles in our context. The set-up is as in §\ref{sec_intro}.

\begin{definition}\label{def_van}
    Let $f: X\rightarrow \AAA^1$ be a map of rigid analytic varieties, $X_0:=X\times_{\AAA^1} 0$, $X^\times:=X\times_{\AAA^1}\Gm$, and $\CF\in D^{}(X^\times)$. The \underline{nearby cycle} $\psi_f(\CF)$ of $\CF$ with respect to $f$ is an object of $D^{}(X_0)$ carrying a continuous action (the \underline{monodromy action}) of $\pi^{\mathrm{fin}}_1(\Gm,1)=\hat{\ZZ}(1)=\mu:=\varprojlim\mu_n$ defined as follows\footnote{Here, $\pi^{\mathrm{fin}}_1$ denotes the fundamental group classifying finite \etale surjective maps, and $\mu_n$ denotes the $n$-th roots of unity in $K$. In the following, we use $\pi^{\mathrm{fin}}_1(\Gm,1)$, $\hat{\ZZ}(1)$ and $\mu$ interchangeably.}: consider the diagrams, where the second is the base change to $X$ of the first:
\[\begin{tikzcd}
	&& {\mathbf{G}_m(n)=\mathbf{G}_m} &&&& {X^{\times}(n)} \\
	0 & {\mathbf{A}^1} & {\mathbf{G}_m} && {X_0} & X & {X^{\times}}
	\arrow[from=1-3, to=2-2]
	\arrow["{e_n:z\mapsto z^n}", from=1-3, to=2-3]
	\arrow["{j_n}"', from=1-7, to=2-6]
	\arrow["{p_n}", from=1-7, to=2-7]
	\arrow[hook, from=2-1, to=2-2]
	\arrow[hook', from=2-3, to=2-2]
	\arrow["i"', hook, from=2-5, to=2-6]
	\arrow["j", hook', from=2-7, to=2-6]
\end{tikzcd}\]
    Then $\psi_f(\CF):=\varinjlim_{n\rightarrow\infty} i^*j_*p_{n*}p_n^*\CF$. The $\pi^{\mathrm{fin}}_1(\Gm,1)$-action is induced by the natural actions of $\mathrm{Aut}(e_n)=\mu_n$ on $p_{n*}p_n^*\CF$. We will often view $\psi_f$ as a functor from $D(X)$ to $D(X_0)$ via $\psi_f j^*$.\\\\
    For $\CG\in D^{}(X)$, the \underline{vanishing cycle} $\phi_f(\CG)$ is an object of $D^{}(X_0)$ carrying a continuous action of $\pi^{\mathrm{fin}}_1(\Gm,1)$ defined as $\cone(i^*\CG\xrightarrow{\mathrm{sp}}\psi_f(\CG))$,\footnote{Strictly speaking, one takes the cone on the $\infty$-categorical level, which induces a functor on the triangulated level.} where $\mathrm{sp}: i^*\CG\rightarrow\psi_f(\CG)$ is induced by the adjunctions $\CG\rightarrow j_{n*}j_n^*\CG$ and called the \underline{specialisation map}. The $\pi^{\mathrm{fin}}_1(\Gm,1)$-action is induced from that on $\psi_f(\CG)$.
\end{definition}

\begin{remark}
    Throughout this work, by an isomorphism of nearby or vanishing cycles, we mean an isomorphism commuting with the monodromy actions.
\end{remark}

\begin{lemma}[scaling invariance]\label{lem_psietalebasechange}
    Let $f: X\rightarrow\AAA^1$ be a map of rigid analytic varieties, and $\CF\in D^{}(X)$. Then, for every $\lambda\in K^\times$, we have a natural isomorphism $\psi_f(\CF)\isoto\psi_{\lambda f}(\CF)$. This isomorphism is canonical up to a choice of $\bar{\lambda}=\{\lambda_n\}_{n\in\ZZ_{\geq1}}\in\varprojlim\mu_n(\lambda)$ with $\lambda_1=\lambda$, where $\mu_n(\lambda)$ denotes the $n$-th roots of $\lambda$ in $K$. Similarly for $\phi$.
\end{lemma}

\begin{lemma}[functoriality]\label{lem_smoothqcqs}
    (1) Let $g: Y\rightarrow X$, $f: X\rightarrow\AAA^1$ be maps of rigid analytic varieties with $g$ smooth, and $\CF\in D^{}(X)$. Then there is a canonical isomorphism $g^*\psi_f(\CF)\isoto\psi_{fg}(g^*\CF)$. Similarly for $\phi$.\\
    (2) Let $g: Y\rightarrow X$, $f: X\rightarrow\AAA^1$ be maps of rigid analytic varieties with $g$ quasi-compact quasi-separated, and $\CG\in D^{}(Y)$. Then there is a canonical isomorphism $g_*\psi_{fg}(\CG)\isoot\psi_{f}(g_*\CG)$. Similarly for $\phi$.
\end{lemma}

\begin{lemma}[commute with analytification]\label{lem_comparisonphi}
    Let $f: \CX\rightarrow\AAA^1$ be a map of finite type schemes over $K$, and $\CF\in D^b_c(\CX^\X)$. Let $f^{an}: X\rightarrow\AAA^1$ and $\CF^{an}$ be their analytifications. Then there is a canonical isomorphism $(\psi_f(\CF))^{an}\isoto\psi_{f^{an}}(\CF^{an})$. Similarly for $\phi$.
\end{lemma}

\begin{proposition}[preserve Zariski-constructibility]\label{prop_phipreserveszc}
    Let $f: X\rightarrow\AAA^1$ be a map of rigid analytic varieties, and $\CF\in \DX$. Then $\psi_f(\CF)$ and $\phi_f(\CF)$ lie in $D^{(b)}_{zc}(X_0)$.
\end{proposition}

\begin{proposition}[Milnor fibre interpretation of stalks]\label{prop_nearbyfibinterp}
    Let $f: X=\Spa(A)\rightarrow\AAA^1$ be a map of rigid analytic varieties, $\CF\in D^b_{zc}(X)$, $x\in X_0$ be a classical point, and $f_1,...,f_r\in A$ such that $x=V(f_1,...,f_r)$ set-theoretically. Then there exists $\epsilon_0\in |K^\times|$ such that for every $0<\epsilon<\epsilon_0$ with $\epsilon\in|K^\times|$, there exists $\eta_0\in |K^\times|$ such that for every $0<\eta<\eta_0$ with $\eta\in|K^\times|$, and every classical point $a\in D^\times(\eta)$, there exists an isomorphism $\psi_f(\CF)_x\cong R\Gamma(U_{\epsilon,a},\CF)$, where $U_{\epsilon,a}:=\{f=a,|f_i|\leq\epsilon,\forall i\}\subseteq X$ is the “Milnor fibre”. Similarly for $\phi_f(\CF)_x$: there exist (possibly different) constants $(\epsilon_0,\epsilon,\eta_0,\eta)$ as above such that there exists an isomorphism $\phi_f(\CF)_x\cong \cone(\CF_x\rightarrow R\Gamma(U_{\epsilon,a},\CF))$ for every classical point $a\in D^\times(\eta)$.
\end{proposition}

\begin{theorem}[perverse t-exactness, commute with duality]\label{thm_main_psi_phi}
Let $f: X\rightarrow\AAA^1$ be a map of rigid analytic varieties. Denote $\psi_f[-1]$ (resp. $\phi_f[-1]$) by $\Psi_f$ (resp. $\Phi_f$). Then:\\
    (1) $\Psi_f$ is perverse t-exact in the following sense: 
    for every $\CF\in$ $^p\!D^{\leq 0}_{zctf}(X^\times)$ (resp. $^p\!D^{\geq 0}_{zctf}(X^\times)$) which extends to an object of $D^{(b)}_{zc}(X)$, we have $\Psi_f(\CF)\in$ $^p\!D^{\leq 0}_{zctf}(X_0)$ (resp. $^p\!D^{\geq 0}_{zctf}(X_0)$);\\
    (2) $\Phi_f$ is perverse t-exact in the following sense: for every $\CF\in$ $^p\!D^{\leq 0}_{zctf}(X)$ (resp. $^p\!D^{\geq 0}_{zctf}(X)$), we have $\Phi_f(\CF)\in$ $^p\!D^{\leq 0}_{zctf}(X_0)$ (resp. $^p\!D^{\geq 0}_{zctf}(X_0)$);\\
    (3) for $\CF\in$ $D^{(b)}_{zc}(X^\X)$ which extends to an object of $\DX$, there is a canonical isomorphism $g: \Psi_f\mathbb{D}\CF\isoto(\mathbb{D}\Psi_f\CF)(1)$ in $D(X_0\bar\X B\mu)$;\\
    (4) for $\CF\in$ $D^{(b)}_{zc}(X)$, there is a natural isomorphism $h: \Phi_f\mathbb{D}\CF\isoto(\mathbb{D}\Phi_f\CF)\miwa(1)$ in $D(X_0\bar\X B\mu)$.\\\\
    Furthermore, the $\mathrm{can}$ and $\mathrm{var}$ triangles are dual to each other in the sense that we have a commutative diagram:
\[\begin{tikzcd}
	{(i^*\mathbb{D}\mathcal{F})[-1]} && {\Psi\mathbb{D}\mathcal{F}} && {\Phi\mathbb{D}\mathcal{F}} \\
	{(\mathbb{D}i^!\mathcal{F})[-1]} && {(\mathbb{D}\Psi\mathcal{F})(1)} && {(\mathbb{D}\Phi\mathcal{F})(-1)^\tau(1)}
	\arrow["\mathrm{sp}", from=1-1, to=1-3]
	\arrow["\alpha", from=1-1, to=2-1]
	\arrow["\mathrm{can}", from=1-3, to=1-5]
	\arrow["g", from=1-3, to=2-3]
	\arrow["\simeq"', from=1-3, to=2-3]
	\arrow["h", from=1-5, to=2-5]
	\arrow["{\mathbb{D}(\mathrm{cosp}(1)^\tau(-1))}", from=2-1, to=2-3]
	\arrow["{\mathbb{D}(\mathrm{var}(1)^\tau(-1))}", from=2-3, to=2-5]
\end{tikzcd}\]
\end{theorem}

Here $X_0\bar\X B\mu$ is the topos of sheaves on $X_0$ equipped with a $\mu$-action, $\miwa$ is the Iwasawa twist, and $^p\!D^{\leq 0}_{zctf}(X^\times)$ \textit{etc.} denote $^p\!D^{\leq 0}_{zc}(X^\X)\cap D^{(b)}_{zctf}(X^\X)$ \textit{etc.} We refer to \cite[\nopp §2 and §5]{zhou_vanishing_2025} for the details of the constructions involved.

\begin{example}\label{ex_phicomp}
    (1) Let $f: X\rightarrow\AAA^1$ be a smooth map of rigid analytic varieties, and $\CL$ a local system on $X$. Then $\phi_f(\CL)=0$.\\
    (2) Let $\CL$ be a local system in degree $0$ on the punctured disc $\Gm$, corresponding to a representation $\pi^{\mathrm{fin}}_1(\Gm,1)\rightarrow \mathrm{Aut}_K(\CL_1)$. Then there exists a canonical isomorphism $\psi_{\id}(\CL)\simeq\CL_1$, with the monodromy action identified with the representation.
\end{example}

For later use, we also recall a general form of Artin-Grothendieck Vanishing. We refer to \cite[\nopp §4]{zhou_vanishing_2025} for details.

\begin{definition}\label{def_w_stein}
    A rigid analytic variety $X$ is called \underline{weakly Stein} if there exist affinoid opens $U_0\subseteq U_1\subseteq ...\subseteq X$, $i\in\NN$, such that $X=\cup_i U_i$. A map $f: X\rightarrow Y$ of rigid analytic varieties is called \underline{weakly Stein} if every $y\in Y$ has a weakly Stein open neighbourhood $V$ such that $f^{-1}(V)=X\X_YV$ is weakly Stein.
\end{definition}

\begin{proposition}[General Artin-Grothendieck Vanishing]\label{prop_artinvanishing}
    Let $f: X\rightarrow Y$ be a weakly Stein map of rigid analytic varieties. Then, $f_*$ (resp. $f_!$) is perverse right (resp. left) t-exact, in the following sense: for every $\CF\in$ $^p\!D^{\leq 0}_{zc}(X)$ (resp. $^p\!D^{\geq 0}_{zc}(Y)$) such that $f_*\CF$ (resp. $f_!\CF$) lies in $D^{(b)}_{zc}(Y)$, we have $f_*\CF\in$ $^p\!D^{\leq 0}_{zc}(Y)$ (resp. $f_!\CF\in$ $^p\!D^{\geq 0}_{zc}(Y)$).
\end{proposition}

\section{Monodromic sheaves}\label{sec_monsheav}

The target of the microlocal functors will be monodromic sheaves on certain cones. In this section, we define and study monodromic sheaves in our context. The set-up is as in §\ref{sec_intro}.\\

Monodromic sheaves were defined by Jean-Louis Verdier (\cite{verdier_specialisation_1983}) in the algebraic context. Our context is similar, but one needs to (1) deal with the non-quasi-compactness of a general rigid analytic variety, and (2) check Zariski-constructibility is preserved in various operations. For (1), working with triangulated categories, we deal with it by a gluing lemma (Lemma \ref{lem_glue_morph}), which requires the underlying space to be quasi-separated. In the appendix, we give an $\infty$-categorical characterisation of monodromic sheaves, which removes the quasi-separatedness assumption but requires $p>0$.\\

We first recall the notion of cones over a rigid analytic variety (see \cite[\nopp §8.3]{EGAII}, \cite[\nopp §2]{verdier_specialisation_1983} for the scheme setting). Let $X$ be a rigid analytic variety, and $\CA_\bullet=\oplus_{i\ge0}\CA_i$ be an $\NN$-graded $\CO_X$-algebra satisfying the following conditions: (i) $\CA_0=\CO_X$, (ii) each $\CA_i$ is coherent as an $\CO_X$-module, (iii) $\CA_1$ generates $\CA_\bullet$ as an $\CO_X$-algebra. The rigid analytic variety $\underline{\Spa}_X(\CA_\bullet)$ is called the \emph{cone} associated to $\CA_\bullet$. Here $\underline{\Spa}$ denotes the relative $\Spa$: over an affinoid $X=\Spa(A)$, $\underline{\Spa}_X(-):=(\underline{\mathrm{Spec}}_{\mathrm{Spec}(A)}(-))^{an/X}$; over a general $X$, it is glued from the local ones. A map between cones $\underline{\Spa}_X(\CA_\bullet)\rightarrow\underline{\Spa}_X(\CB_\bullet)$ is a map induced by an $\CO_X$-module map $\CA_1\leftarrow\CB_1$. By \cite[\nopp 7.3]{zavyalov_some_2024}, over an affinoid $X=\Spa(A)$, the category of cones over $X$ is equivalent to the category of cones over $\mathrm{Spec}(A)$ via the analytification functor.\\

Let $E$ be a cone over $X$. There is a canonical $\Gm$-action on $E$ associated to the grading, which we refer to as the scaling action. For each $\lambda\in K$, we denote by $\theta_\lambda$ the $\lambda$-scaling map: $E\rightarrow E$, $v\mapsto \lambda v$. For each $n\in\ZZ_{\geq1}$, we denote by $\theta(n)$ the map $\Gm\X E\rightarrow E:$ $(\lambda,v)\mapsto \lambda^nv$, and by $\pr$ the projection $\Gm\X E\rightarrow E$. For $n, m\in\ZZ_{\geq1}$, $m|n$, we denote by $e_{n,m}$ the map $\Gm\X E\rightarrow\Gm\X E$, $(\lambda,v)\mapsto (\lambda^{n/m},v)$.\\

Consider $\PP(E):=\underline{\mathrm{Proj}}_X(\CA_\bullet)$ (where $\underline{\mathrm{Proj}}_X$ is the relative Proj, defined similarly as $\underline{\Spa}_X$ via analytification) and $\PP(E\oplus 1):=\underline{\mathrm{Proj}}_X(\CA_\bullet[z])$ ($z$ is a free variable in degree 1). We have the following maps among the various spaces (the claims all follow from their algebraic analogues via analytification):\\
(i) The zero-section $X\hookrightarrow E$, which is a closed immersion. We will identify $X$ with its image in $E$.\\
(ii) The projection $(E-X)\rightarrow X$. This factors as $(E-X)\rightarrow \PP(E)\rightarrow X$, and the first map exhibits $E-X$ as a $\Gm$-torsor over $\PP(E)$. In particular, the image of $(E-X)\rightarrow X$ is Zariski-closed.\\
(iii) The Zariski-open immersion $E\hookrightarrow \PP(E\oplus 1)$. The complement is a Zariski-closed subset identified with $\PP(E)$.

\begin{proposition}[{\cite[\nopp 5.1]{verdier_specialisation_1983}}]\label{prop_fundamental_of_mon_qc}
    Let $E$ be a cone over a quasi-compact rigid analytic variety $X$. Let $\CF\in\Dbzc(E)$. Then, the following are equivalent:\\
    (1) for every $i\in\ZZ$ and every $\lambda\in K^\X$, there exists an isomorphism $\theta_\lambda^*\CH^i(\CF)\isoto \CH^i(\CF)$;\\
    (2) for every $\lambda\in K^\X$, there exists an isomorphism $\theta_\lambda^*\CF\isoto \CF$;\\
    (3) there exists an $n\in\ZZ_{\geq1}$ and an isomorphism $\iota(n): \theta(n)^*\CF\isoto \pr^*\CF$ which restricts to $\id$ on $1\X E$.
\end{proposition}

\begin{proof}
    It is clear that (2)$\implies$(1). For (3)$\implies$(2), just observe that the restriction of $\iota(n)$ to $\lambda_{n}\X E$ gives a desired isomorphism, where $\lambda_n$ is any $n$-th root of $\lambda$ in $K$. We now show (1)$\implies$(3). A preliminary observation: if $\CF$ satisfies (1), then $\CF$ is constructible with respect to a $\Gm$-stable Zariski stratification. Indeed, condition (1) implies that the maximal Zariski open on which $\CF$ is lisse is $\Gm$-stable, iterate this on its complement, we get a desired Zariski stratification. As $X$ is quasi-compact, it follows that for each $\CF\in\Dbzc(E)$ satisfying (1), $\CF$ in fact lies in $D^b_{zc}(E)$.\\\\
    We do induction on the amplitude of $\CF$. If $\CF$ is concentrated in degree $0$, then it follows from the above observation that there exists an $n\in\ZZ_{\geq1}$ such that $\theta(n)^*\CF$ is constant on $\Gm\X v$ for every $v\in E$. Indeed: on each stratum in $E-X$, $\CF$ is trivialised by some finite \etale covering, which restricts to some covering of Kummer type on each $\Gm$-orbit, their degrees are uniformly bounded as there are only finitely many strata. It follows that $\theta(n)^*\CF\isoto \pr^*\CH$ for some $\CH\in D^b_{zc}(E)$: in fact, as one can verify on each geometric $\pr$-fibre, the composition $\theta(n)^*\CF\rightarrow \pr^!\pr_!(\theta(n)^*\CF)\rightarrow \CH^0(\pr^!\pr_!(\theta(n)^*\CF))$ is an isomorphism. Restricting to $1\X E$, we see $\CH\cong\CF$. We have shown that there exists an isomorphism $\alpha: \theta(n)^*\CF\isoto \pr^*\CF$. Compose with $\pr^*$ of the inverse of $\alpha|_{1\X E}$, we get an isomorphism $\iota(n): \theta(n)^*\CF\isoto \pr^*\CF$ which restricts to $\id$ on $1\X E$.\\\\
    Now consider a general $\CF$.\\\\
    \underline{Claim}.$\,\,\forall \CF, \CG\in D^b_{zc}(E)$, $\pr^*$ induces an isomorphism $RHom(\CF,\CG)\isoto (RHom(\pr^*\CF, \pr^*\CG),\rho_n)$ as ind-objects of $D^b_c(\Lambda)$ indexed by $\ZZ_{\geq1}$ partially ordered by divisibility. Here $\rho_n$ are transition maps induced by $e_{n_1,n_2}^*\pr^*=\pr^*$.\\\\
    \underline{Proof}. First note $RHom(\pr^*\CF, \pr^*\CG)\simeq R\Gamma(\Gm\X E, \RHom(\pr^*\CF, \pr^*\CG))\simeq R\Gamma(\Gm\X E, \pr^*\RHom(\CF, \CG))$, where in the last step we used the general formula $\pr^!\RHom(\CF, \CG)\simeq\RHom(\pr^*\CF, \pr^!\CG)$ and $\pr^!=\pr^*\langle1\rangle$. Denote the projection to $\Gm$ by $q$, we then get $R\Gamma(\Gm\X E, \pr^*\RHom(\CF, \CG))\simeq R\Gamma(\Gm$, $q_*\pr^*\RHom(\CF, \CG)))\simeq R\Gamma(\Gm, \underline{RHom(\CF, \CG)})\simeq RHom(\CF, \CG)\otimes R\Gamma(\Gm,\Lambda)$, where in the second step we used smooth base change, and in the last step we used Lemma \ref{lem_coh_purity_consequence}. Combine the above, we get $RHom(\pr^*\CF, \pr^*\CG)\simeq RHom(\CF, \CG)\otimes R\Gamma(\Gm,\Lambda)$, and $\rho_n$ is induced by the maps among $R\Gamma(\Gm,\Lambda)$ induced by $e_n$. Finally, as $\varinjlim_{n\rightarrow\infty}(R\Gamma(\Gm,\Lambda), e_n)=\Lambda$ (in degree $0$), we get the Claim.\\\\
    Return to the proof. Let $\CA\xrightarrow{u}\CB\rightarrow\CF\rightarrow$ be a distinguished triangle with $\CA, \CB\in D^b_{zc}(E)$ each having strictly smaller amplitude than $\CF$. By induction hypothesis, there exists an $n\in\ZZ_{\geq1}$ and isomorphisms $\iota_{\CA}(n)$ and $\iota_{\CB}(n)$ restricting to id on $1\X E$. Consider the following diagram, where $v$ is a composition of an inverse of $\iota_{\CA}(n)$ with $\theta(n)^*u$ and $\iota_{\CB}(n)$:
\[\begin{tikzcd}
	{\theta(n)^*\mathcal{A}} & {\theta(n)^*\mathcal{B}} \\
	{\pr^*\mathcal{A}} & {\pr^*\mathcal{B}}
	\arrow["{\theta(n)^*u}", from=1-1, to=1-2]
	\arrow["{\iota_\mathcal{A}(n)}", from=1-1, to=2-1]
	\arrow["\cong"', from=1-1, to=2-1]
	\arrow["\cong"', from=1-2, to=2-2]
	\arrow["{\iota_\mathcal{B}(n)}", from=1-2, to=2-2]
	\arrow["v"', from=2-1, to=2-2]
\end{tikzcd}\]
By the Claim, by increasing $n$ we may assume that $v=\pr^*v'$, for some $v'\in Hom(\CA,\CB)$. Taking cones, we get an isomorphism $\theta(n)^*\CF\isoto \pr^*\CH$, where $\CH\in D^b_{zc}(E)$ is a cone of $v'$. Argue as in the end of the second paragraph above, we get an $\iota(n): \theta(n)^*\CF\isoto \pr^*\CF$ which restricts to $\id$ on $1\X E$.
\end{proof}

We used the following lemma in the proof, which is a consequence of the projection formula, the recollement distinguished triangles and Poincaré duality (\cite[\nopp 1.3.2]{zavyalov_poincare_2023}).

\begin{lemma}\label{lem_coh_purity_consequence}
    Let $M\in D(\Lambda)$, $p_1: \Gm\rightarrow \mathrm{pt}(=\Spa(K))$, $p_2: \AAA^1\rightarrow \mathrm{pt}$, and $p_3: \PP^1\rightarrow \mathrm{pt}$. Then $p_{1*}p_1^*M\simeq R\Gamma(\Gm,\Lambda)\otimes M\simeq M\oplus M(-1)[-1]$, $p_{2*}p_2^*M\simeq M$, and $p_{3*}p_3^*M\simeq R\Gamma(\PP^1,\Lambda)\otimes M\simeq M\oplus M\langle -1\rangle$.
\end{lemma}

\begin{lemma}[{\cite[\nopp 5.1]{verdier_specialisation_1983}}]\label{lem_mon_ess_uni_fun_of_iota}
    Same set-up as in Proposition \ref{prop_fundamental_of_mon_qc}. Then, the $\iota(n)$ is essentially unique and functorial in the following sense:\\
    (1) for two such $\iota(n_1)$, $\iota(n_2)$, there exists an $n$ divisible by $n_1$ and $n_2$ such that $e_{n,n_1}^*\iota(n_1)=e_{n,n_2}^*\iota(n_2)$ as elements of $Hom(\theta(n)^*\CF, \pr^*\CF)$;\\
    (2) let $\CG\in \Dbzc(E)$ also satisfy the conditions in Proposition \ref{prop_fundamental_of_mon_qc}, and $u: \CF\rightarrow\CG$ be a map. Then there exist $n\in\ZZ_{\geq1}$ and $\iota_\CF(n)$ and $\iota_\CG(n)$ as in Proposition \ref{prop_fundamental_of_mon_qc}.3 such that the following diagram commutes:
\[\begin{tikzcd}
	{\theta(n)^*\mathcal{F}} & {\pr^*\mathcal{F}} \\
	{\theta(n)^*\mathcal{G}} & {\pr^*\mathcal{G}}
	\arrow["{\iota_{\mathcal{F}}(n)}", from=1-1, to=1-2]
	\arrow["\sim"', from=1-1, to=1-2]
	\arrow["{\theta(n)^*u}"', from=1-1, to=2-1]
	\arrow["{\pr^*u}", from=1-2, to=2-2]
	\arrow["{\iota_{\mathcal{G}}(n)}", from=2-1, to=2-2]
	\arrow["\sim"', from=2-1, to=2-2]
\end{tikzcd}\]
\end{lemma}

\begin{proof}
(1) Take any $n\in\ZZ_{\geq1}$ dividing both $n_1$ and $n_2$. Then $u:=e_{n,n_1}^*\iota(n_1)$ and $v:=e_{n,n_2}^*\iota(n_2)$ are isomorphisms $\theta(n)^*\CF\isoto \pr^*\CF$. Consider the map $u\circ v^{-1}: \pr^*\CF\rightarrow \pr^*\CF$. By the Claim in the proof of Proposition \ref{prop_fundamental_of_mon_qc}, by increasing $n$, we may assume $u\circ v^{-1}$ is of the form $\pr^*w$ for some $w: \CF\rightarrow\CF$. Look at the restriction to $1\X E$, we see $w=\id$. So $u=v$.\\\\
(2) Take any $n\in\ZZ_{\geq1}$, and any $\iota_\CF(n), \iota_\CG(n)$. Apply the same argument as in (1) to the map $\iota_\CG(n)\circ\theta(n)^*u\circ\iota_\CF(n)^{-1}: \pr^*\CF\rightarrow \pr^*\CG$, we see it equals $\pr^*u$ after possibly increasing $n$. 
\end{proof}

\begin{corollary}\label{cor_fundamental_of_mon}
    Let $E$ be a cone over a quasi-separated rigid analytic variety $X$. Let $\CF\in\Dbzc(E)$. Then, the following are equivalent:\\
    (1) for every $i\in\ZZ$ and every $\lambda\in K^\X$, there exists an isomorphism $\theta_\lambda^*\CH^i(\CF)\isoto \CH^i(\CF)$;\\
    (2) for every $\lambda\in K^\X$, there exists an isomorphism $\theta_\lambda^*\CF\isoto \CF$;\\
    (3) over each quasi-compact open $U$ in $X$, there exists an $n_U\in\ZZ_{\geq1}$ and an isomorphism $\iota_U(n_U): \theta_U(n_U)^*\CF_{E_U}\isoto \pr_U^*\CF_{E_U}$ which restricts to $\id$ on $1\X E_U$.
\end{corollary}

\begin{proof}
    (2)$\implies$(1) is clear, (1)$\implies$(3) is a direct consequence of Proposition \ref{prop_fundamental_of_mon_qc}. We show (3)$\implies$(2). Fix a $\bar{\lambda}=\{\lambda_n\}_{n\in\ZZ_{\geq1}}\in\varprojlim\mu_n(\lambda)$ with $\lambda_1=\lambda$. Write $X=\cup_{\alpha\in I}U_\alpha$ as the filtered union of all its quasi-compact open subsets. On each $U_\alpha$, we have an isomorphism $\rho_\alpha: \theta_\lambda^*\CF\isoto \CF$, which is the restriction of $\iota(n_\alpha)$ to $\lambda_{n_\alpha}\X E$ (we omit the subscripts $\alpha$ on $\theta$, $\CF$, $\iota$ and $E$). By Lemma \ref{lem_mon_ess_uni_fun_of_iota}, this is independent of the choice of $\iota(n_\alpha)$, and $\rho_\alpha$ is compatible with restrictions along $U_\alpha\subseteq U_{\alpha'}$ for $\alpha\leq \alpha'$. Apply Lemma \ref{lem_glue_morph} to $\{U_\alpha\}$ and $\{\rho_\alpha\}$, we get a unique isomorphism $\rho(\bar{\lambda}): \theta_\lambda^*\CF\isoto\CF$ over $X$ which restricts to $\rho_\alpha$ on $U_\alpha$ for each $\alpha$.
\end{proof}

We learned the following lemma from Bhargav Bhatt.

\begin{lemma}\label{lem_glue_morph}
    Let $X$ be a rigid analytic variety, and $X=\cup_{\alpha\in I}U_\alpha$ be a cover by quasi-compact quasi-separated open subsets $U_\alpha$, where $I$ is a filtered index set and $U_\alpha\subseteq U_{\alpha'}$ for $\alpha\leq \alpha'$. Suppose given $\CF$, $\CG\in \Dbzc(X)$, and $\{\varphi_{\alpha}\in Hom_{U_\alpha}(\CF_{U_\alpha},\CG_{U_\alpha})\}_{\alpha\in I}$ compatible with the transition maps, then there is a unique $\varphi\in Hom_X(\CF,\CG)$ whose restriction to $U_\alpha$ coincides with $\varphi_{\alpha}$ for each $\alpha$.
\end{lemma}

\begin{proof}
    Denote $\RHom_{X}(\CF,\CG)$ by $\CA$. Note $\CA\in D^{(+)}_{zc}(X)$ (\cite[item (8) after 3.29]{bhatt_six_2022}). We have $Hom_{X}(\CF,\CG)=H^0(X,\CA)$ and $Hom_{U_\alpha}(\CF,\CG)=H^0(U_\alpha,\CA)$. There is a spectral sequence $E^{p,q}_2=R^p\varprojlim_\alpha H^q(U_\alpha,\CA)\Rightarrow H^{p+q}(X,\CA)$. As $\CA$ is in $D^{(+)}_{zc}(X)$ and $U_\alpha$ is quasi-compact and quasi-separated, we get that each $H^q(U_\alpha,\CA)$ is a finitely generated $\Lambda$-module (\cite[\nopp 2.1]{huber_finiteness_1998}). By Matlis Duality, $\varprojlim_\alpha$ is exact on finitely generated $\Lambda$-modules (\textit{cf.} \cite[\nopp 4.10]{bhatt_globally_2023}), so $R^p\varprojlim_\alpha H^q(U_\alpha,\CA)=0$ for all $p\neq 0$ and all $q$. In particular, $H^0(X,\CA)\simeq \varprojlim_\alpha H^0(U_\alpha,\CA)$. So $\{\varphi_{\alpha}\}\in\varprojlim_\alpha H^0(U_\alpha,\CA)$ corresponds to some unique element $\varphi\in H^0(X,\CA)$.
\end{proof}

\begin{definition}\label{def_monsheav}
    Let $E$ be a cone over a quasi-separated rigid analytic variety $X$. We say $\CF\in\Dbzc(E)$ is \underline{monodromic} if it satisfies any of the equivalent conditions in Corollary \ref{cor_fundamental_of_mon}. The full sub-category of monodromic sheaves is denoted by $D_{mon}(E)$.
\end{definition}

As we saw in the proof of (3)$\implies$(2) in Corollary \ref{cor_fundamental_of_mon}, for $\lambda\in K^\X$, the isomorphism $\rho(\bar{\lambda}): \theta_\lambda^*\CF\isoto\CF$ is canonical up to a choice of $\bar{\lambda}=\{\lambda_n\}\in\varprojlim\mu_n(\lambda)$ with $\lambda_1=\lambda$. One may think of choosing $\bar\lambda$ as choosing a path from $1$ to $\lambda$ on $\Gm$, and $\rho(\bar{\lambda})$ is induced by “transporting” the stalk $\CF_{\lambda x}$ to the stalk $\CF_x$ along this path, for $x\in E$. In particular, for $\lambda=1$, we get

\begin{definition}\label{def_equimonodromy}
    Let $E$ be a cone over a quasi-separated rigid analytic variety $X$. The \underline{monodromy} of $\CF\in D_{mon}(E)$ is the following action of $\mu=\varprojlim\mu_n$ on $\CF$: for $X$ quasi-compact and $\bar\lambda=\{\lambda_n\}\in\varprojlim\mu_n$, $\rho(\bar\lambda): \CF\isoto\CF$ is the restriction of any $\iota(n)$ as in Proposition \ref{prop_fundamental_of_mon_qc}.3 to $\lambda_n\X E$; for general $X$, $\rho(\bar\lambda)$ is the isomorphism constructed in the proof of Corollary \ref{cor_fundamental_of_mon}. To distinguish from other notions of monodromy, we will also refer to this monodromy as the \underline{equivariant monodromy}.
\end{definition}

\begin{remark}\label{rmk_defofmon}
    (1) Monodromic sheaves can be defined in much greater generality on spaces with a torus action. We will not pursue this further here.\\
    (2) Let $E$ be a cone over a quasi-separated rigid analytic variety $X$, and $\CF\in D_{mon}(E)$. Recall the observation in the beginning of the proof of Proposition \ref{prop_fundamental_of_mon_qc} that $\CF$ is constructible with respect to a $\Gm$-stable Zariski-stratification. Further consequences of this: \\
    (a) $j_!\CF$ (hence also $j_*\CF$) is Zariski-constructible, where $j$ is the Zariski-open immersion $E\hookrightarrow\PP(E\oplus 1)$; \\
    (b) if $E$ is a cone over $\mathrm{Spec}(K)$, then the analytification induces an equivalence $D_{mon}(E)\rightarrow D_{mon}(E^{an})$ between monodromic sheaves in the algebraic and analytic worlds.\\
    (3) The perverse t-structure on $\Dbzc(E)$ restricts to a perverse t-structure on $D_{mon}(E)$ (\textit{i.e.}, $^p\!D^{\leq 0}_{mon}(E):=$ $^p\!D^{\leq 0}_{zc}(E)\cap D_{mon}(E)$ and similarly for $^p\!D^{\geq 0}_{mon}(E)$ form a t-structure on $D_{mon}(E)$). This follows directly from Corollary \ref{cor_fundamental_of_mon}.3 and the shifted perverse t-exactness of $\theta(n)^*$ and $\pr^*$.
\end{remark}

\begin{corollary}\label{cor_dualmono}
    Let $E$ be a cone over a quasi-separated rigid analytic variety $X$. Then,\\
    (1) $D_{mon}(E)$ is a triangulated subcategory of $\Dbzc(E)$;\\ 
    (2) $\DD$ is a contravariant self-equivalence on $D_{mon}(E)$.
\end{corollary}

\begin{proof}
    (1). That $D_{mon}(E)$ is closed under taking cones follows from Corollary \ref{cor_fundamental_of_mon}.3 and Lemma \ref{lem_mon_ess_uni_fun_of_iota}.2. (2). This follows from Corollary \ref{cor_fundamental_of_mon}.2, $\theta_\lambda^*\DD\simeq\DD\theta_\lambda^*$, and the fact that $\DD$ is a contravariant self-equivalence on $\Dbzc(E)$.
\end{proof}

\begin{example}\label{ex_mono_reverse}
    Let $\CL$ be a local system in degree $0$ on $\Gm$, denote $j: \Gm\hookrightarrow\AAA^1$. Then $j_!\CL$ is monodromic, with equivariant monodromy the opposite of the natural monodromy of $\CL$. More precisely: if $\CL$ corresponds to a representation $\rho: \pi^{\mathrm{fin}}_1(\Gm,1)\rightarrow \mathrm{Aut}_\Lambda(\CL_1)$ under the $\{\mathrm{Rep}(\pi_1)\}-\{\mathrm{Local\,\, system}\}$ correspondence, then, under the identification $\pi^{\mathrm{fin}}_1(\Gm,1)=\mu$, the equivariant monodromy $j_!\CL$ restricts at $1$ to $\rho^{-1}: \mu\rightarrow\mathrm{Aut}_\Lambda(\CL_1)$, $\bar\lambda\mapsto \rho(\bar\lambda)^{-1}$. This is because the equivariant monodromy is induced from $\rho(\bar{\lambda'}): \theta_{\lambda'}^*\CF\isoto\CF$ for $\bar{\lambda'}\in(\varprojlim\Gm)(K)$ and this map is in the opposite direction of the natural monodromy (\textit{cf.} the discussion above Definition \ref{def_equimonodromy}). We leave the details to the reader.
\end{example}

\begin{lemma}\label{lem_monfunctoriality}
    Let $\alpha: \tE\rightarrow E$ be a map of cones over a quasi-separated rigid analytic variety $X$, $\CF\in D_{mon}(\tE)$, and $\CG\in D_{mon}(E)$. Then,\\
    (1) $\alpha_*\CF$ and $\alpha_!\CF$ lie in $D_{mon}(E)$;\\
    (2) $\alpha^*\CG$ and $\alpha^!\CG$ lie in $D_{mon}(\tE)$.
\end{lemma}

\begin{proof}
    (1) It suffices to show $\alpha_!\CF\in D_{mon}(E)$, the other one follows by duality (Corollary \ref{cor_dualmono}.2 and $\DD\alpha_!\simeq\alpha_*\DD$). It is clear that the condition in Corollary \ref{cor_fundamental_of_mon}.2 is satisfied, by the $\Gm$-equivariance of $\alpha$. So we only need to show $\alpha_!\CF$ is Zariski-constructible. It suffices to treat the affinoid case $X=\Spa(A)$. Let $\Gamma\hookrightarrow \tE\X_X E$ be the graph of $\alpha$ (over $X$), and $\overline{E}=\PP(\widetilde{E}\oplus 1)$. Note that they are analytifications of the corresponding objects $\Gamma^{alg}$ and $\overline{E}^{alg}$ over $X^{alg}=\mathrm{Spec}(A)$. Consider the closure $\overline\Gamma^{alg}$ of $\Gamma^{alg}$ in $(\overline{E}\X_X E)^{alg}$. Denote its analytification by $\Gamma'$, it is a Zariski-closed subset of $\overline{E}\X_X E$, whose intersection with $\widetilde{E}\times_XE$ is $\Gamma$ (whether $\Gamma'$ is the closure of $\Gamma$ in $\overline{E}\X_X E$ is irrelevant for this proof). We thus have the following diagram:
\[\begin{tikzcd}
	& {\Gamma'} & \Gamma \\
	E & {\overline{E}\times_XE} & {\widetilde{E}\times_XE} \\
	& {\overline{E}} & {\widetilde{E}}
	\arrow["{\overline{i}}"', hook, from=1-2, to=2-2]
	\arrow["j"', hook', from=1-3, to=1-2]
	\arrow["\lrcorner"{anchor=center, pos=0.125, rotate=-90}, draw=none, from=1-3, to=2-2]
	\arrow["i"', hook, from=1-3, to=2-3]
	\arrow["q"', from=2-2, to=2-1]
	\arrow["{\bar{p}}"', from=2-2, to=3-2]
	\arrow["j"', hook', from=2-3, to=2-2]
	\arrow["\lrcorner"{anchor=center, pos=0.125, rotate=-90}, draw=none, from=2-3, to=3-2]
	\arrow["{p}"', from=2-3, to=3-3]
	\arrow["{\simeq, \sigma}"', curve={height=30pt}, from=3-3, to=1-3]
	\arrow["j", hook', from=3-3, to=3-2]
\end{tikzcd}\]
Here, $p$, $\bar{p}$, $q$ are projections, $\sigma$ is the isomorphism of $\tE$ to its graph, and $\overline{i}$, $i$ are closed immersions. Then, $\alpha_!\CF\simeq q_*\overline{i}_*\overline{i}^*\bar{p}^*j_!\CF$. As $j_!\CF$ is Zariski-constructible (Remark \ref{rmk_defofmon}.2), and proper pushforward preserves Zariski-constructible sheaves, we get $q_*\overline{i}_*\overline{i}^*\bar{p}^*j_!\CF$ is Zariski-constructible.\\\\
(2) By duality, it suffices to show $\alpha^*\CF\in D_{mon}(E)$. The condition in Corollary \ref{cor_fundamental_of_mon}.2 is satisfied, by the $\Gm$-equivariance of $\alpha$. Since Zariski-constructible sheaves are preserved under pullbacks, the conclusion follows.
\end{proof}

\begin{lemma}\label{lem_hyp_localisation}
    Let $E$ be a cone over a quasi-separated rigid analytic variety $X$, and $\CF\in D_{mon}(E)$. Denote by $i: X\hookrightarrow E$ the zero section, and $p: E\rightarrow X$ the projection. Then, there are canonical isomorphisms $p_*\CF\isoto i^*\CF$ and $p_!\CF\isoot i^!\CF$.
\end{lemma}

Note that all four terms here are Zariski-constructible (use Remark \ref{rmk_defofmon}.2.(a) for the left hand sides). A more general statement is true in a slightly different setting: \cite[\nopp IV.6.5]{fargues_geometrization_2021}. The following proof is an adaptation of \cite[Proposition 1]{springer_purity_1984}.

\begin{proof}
      The map $p_*\CF\rightarrow i^*\CF$ (resp. $p_!\CF\leftarrow i^!\CF$) is $p_*$ (resp. $p_!$) applied to $\CF\rightarrow i_*i^*\CF$ (resp. $\CF\leftarrow i_!i^!\CF$). It suffices to show $p_*\CF\rightarrow i^*\CF$ is an isomorphism, the other follows from duality. We may assume $X$ is quasi-compact, and, by dévissage, assume $i^*\CF=0$. We will show $p_*\CF=0$ by constructing a homotopy between $\id$ and $0$ in $End_{X}(p_*\CF)$, \textit{i.e.}, a $\gamma$ in $End_{\AAA^1_X}(\prbar^*p_*\CF)$ which restricts to $0$ on the $0$-section $0_X$ and $\id$ on the $1$-section $1_X$, where $\prbar$ is the projection $\AAA^1_X\rightarrow X$. This is sufficient, because $\gamma$ is necessarily of the form $\prbar^*u$ for some $u\in End_{X}(p_*\CF)$,\footnote{This follows from a similar argument as for the Claim in the proof of Proposition \ref{prop_fundamental_of_mon_qc}. Namely, for $A$, $B\in D^{(b)}(X)$, we have $RHom(\overline{\pr}^*A,\overline{\pr}^*B)\simeq R\Gamma(\AAA^1_X, \RHom(\overline{\pr}^*A,\overline{\pr}^*B))\simeq R\Gamma(\AAA^1_X, \overline{\pr}^*\RHom(A,B))\simeq R\Gamma(\AAA^1_K,\underline{RHom(A,B)})\simeq RHom(A,B)$, where in the second and third step we used the smoothness of $\overline{\pr}$, and in the last step we used Lemma \ref{lem_coh_purity_consequence}.} so $\cone(\gamma)\cong\prbar^*\cone(u)$, consequently $\cone(u)$ is $0$ (by looking at $1_X$) and isomorphic to $p_*\CF$ (by looking at $0_X$), hence $p_*\CF=0$.\\\\
      To construct $\gamma$, let $n\in\ZZ_{\geq1}$, $\theta(n)$ and $\iota(n)$ be as in Corollary \ref{cor_fundamental_of_mon}.3. We will omit $n$ from the notations. Consider the following diagram:
\[\begin{tikzcd}
	{\mathbf{G}^1_{m,X}} & {\mathbf{G}^1_{m}\times E} \\
	{\mathbf{A}^1_X} & {\mathbf{A}^1\times E} \\
	X & E
	\arrow[hook, from=1-1, to=2-1]
	\arrow[from=1-2, to=1-1]
	\arrow["\lrcorner"{anchor=center, pos=0.125, rotate=-90}, draw=none, from=1-2, to=2-1]
	\arrow["j", hook, from=1-2, to=2-2]
	\arrow["\prbar"', from=2-1, to=3-1]
	\arrow["p"', from=2-2, to=2-1]
	\arrow["\lrcorner"{anchor=center, pos=0.125, rotate=-90}, draw=none, from=2-2, to=3-1]
	\arrow["\prbar"', from=2-2, to=3-2]
	\arrow["{\bar{\theta}}", shift left=3, curve={height=-6pt}, from=2-2, to=3-2]
	\arrow["p", from=3-2, to=3-1]
\end{tikzcd}\]
where $j$ is the usual inclusion, $\prbar$ and $p$ are projections, and $\bar{\theta}:=\prbar\circ\tau$, where $\tau: \AAA^1\X E\rightarrow\AAA^1\X E, (t,v)\mapsto(t,t^nv)$. As $i^*\CF$ is zero and $\bar{\theta}$ restricts to $p$ on $0\X E$, $\iota$ induces an isomorphism $\bar{\theta}^*\CF\isoto j_!j^*\prbar^*\CF$. Compose with the adjunction $j_!j^*\prbar^*\CF\rightarrow\prbar^*\CF$, and apply $p_*$, we get a map $\alpha: p_*\bar{\theta}^*\CF\rightarrow p_*\prbar^*\CF$. On the other hand, we have $\bar{\theta}^*\CF\simeq\tau^*\prbar^*\CF$, apply $p_*$ and compose with the adjunction  $p_*\rightarrow p_*\tau^*(\simeq p_*\tau_*\tau^*)$, we get a map $\beta: p_*\prbar^*\CF\rightarrow p_*\bar{\theta}^*\CF$. Compose $\alpha$ and $\beta$, we get a map $p_*\overline{\pr}^*\CF\rightarrow p_*\overline{\pr}^*\CF$. By smooth base change, this gives a map $\gamma: \prbar^*p_*\CF\rightarrow \prbar^*p_*\CF$. It follows from the construction that $\gamma$ is the identity on the $1$-section $1_X$, and zero on the $0$-section $0_X$. 
\end{proof}

\begin{proposition}\label{lem_rank1monbasic}
    Let $E=\AAA_X^1$ be the trivial bundle over a quasi-separated rigid analytic variety $X$, and $\CF\in D_{mon}(E)$. Denote by $i_0$ (resp. $i_1$) the closed immersion $0_X\hookrightarrow \AAA^1_X$ (resp. $1_X\hookrightarrow \AAA^1_X$) of the 0-section (resp. 1-section), by $s: 0_X\isoto 1_X$ the canonical isomorphism, by $t$ the projection $\AAA_X^1\rightarrow \AAA^1_K$, and by $p$ the projection $\AAA^1_X\rightarrow X$. Then:\\
    (1) there is a canonical isomorphism $\psi_t(\CF)\isoto s^*i^*_1\CF$, under which the monodromy of $\psi_t(\CF)$ coincides with the opposite of the equivariant monodromy of $s^*i^*_1\CF$. Furthermore, the following diagram commutes: 
\begin{equation}\label{eqn_rank1monbasic}
\begin{tikzcd}
	{i_0^*\mathcal{F}} & {\psi_t\mathcal{F}} & {s^*i_1^*\mathcal{F}\simeq p_*i_{1*}i_1^*\mathcal{F}} \\
	& {p_*\mathcal{F}}
	\arrow["{\mathrm{sp}}", from=1-1, to=1-2]
	\arrow["\sim", from=1-2, to=1-3]
	\arrow["{p_*(\mathrm{adj}_0^*)}", curve={height=-12pt}, from=2-2, to=1-1]
	\arrow["\sim"', curve={height=-12pt}, from=2-2, to=1-1]
	\arrow["{p_*(\mathrm{adj}_1^*)}"'{pos=0.6}, curve={height=12pt}, from=2-2, to=1-3]
\end{tikzcd}
\end{equation}
Here $\mathrm{adj}_0^*$ (resp. $\mathrm{adj}_1^*$) denotes the adjunction $\mathrm{id}\rightarrow i_{0*}i_0^*$ (resp. $\mathrm{id}\rightarrow i_{1*}i_1^*$), and the bottom left arrow is an isomorphism by Lemma \ref{lem_hyp_localisation}.\\
    (2) There is a distinguished triangle $s^*i_1^!\CF\rightarrow i_0^!\CF\rightarrow(\Phi_t\CF)(-1)\rightarrow$, under which the monodromy of $(\Phi_t\CF)(-1)$ coincides with the opposite of that induced by the equivariant monodromy of $s^*i^!_1\CF$. Here the first arrow is defined by the maps $s^*i_1^!\CF\simeq p_!i_{1*}i_1^!\CF\xrightarrow{p_!(\mathrm{adj}_1^!)}p_!\CF\xleftarrow[\sim]{p_!(\mathrm{adj}_0^!)} i^!_0\CF$, where $\mathrm{adj}_1^!$ (resp. $\mathrm{adj}_0^!$) denotes the adjunction $i_{1*}i_1^!\rightarrow\mathrm{id}$ (resp. $i_{0*}i_0^!\rightarrow\mathrm{id}$), and we have used Lemma \ref{lem_hyp_localisation}.
\end{proposition}

\begin{remark}
    In (2), we have ignored the Iwasawa twist as we are working with triangulated categories and not concerned with the canonicity of the triangle.
\end{remark}

\begin{proof}
    (1) For convenience, we denote $i_1^*\CF$ by $\CF_1$ and $p^*s^*i^*_1\CF$ by $\underline{\CF_1}$, and view $\CF_1$ as a sheaf on both $0_X$ and $1_X$. First assume $X$ is quasi-compact. Recall Definition \ref{def_van} (same notations as there): $\psi_t(\CF):=\varinjlim_{n\rightarrow\infty} i^*j_*p_{n*}p_n^*\CF$. Fix an $m\in\ZZ_{\geq1}$ and $\iota(m): \theta(m)^*\CF\isoto \mathrm{pr}^*\CF$ as in Corollary \ref{prop_fundamental_of_mon_qc}.3. $\iota(m)$ restricts on $\Gm\X 1_X$ to an isomorphism $\iota_m: p_m^*\CF\isoto \mathrm{pr}^*\CF_1\simeq p_m^*\underline{\CF_1}$, where $\mathrm{pr}$ here is the projection $\Gm\X 1_X\rightarrow 1_X$. So $\psi_t(\CF)\simeq\varinjlim_{m|n} i^*j_*p_{n*}p_n^*\CF\isoto\varinjlim_{m|n} i^*j_*p_{n*}p_n^*\underline{\CF_1}\simeq \psi_t(\underline{\CF_1})\xleftarrow[\mathrm{sp}]{\sim} \CF_1$. It follows from the essential uniqueness and functoriality of $\iota(m)$ (Lemma \ref{lem_mon_ess_uni_fun_of_iota}) that this isomorphism is independent of $m$ and $\iota(m)$, and is functorial. That the monodromy is reversed can be seen similarly as in Example \ref{ex_mono_reverse}.\\\\
    For a general $X$, write $X=\cup_\alpha U_\alpha$ as the filtered union of all its quasi-compact opens subsets. Over each $U_\alpha$ we have the isomorphism constructed above, and they are compatible with restrictions by the essential uniqueness of $\iota(m)$. Apply Lemma \ref{lem_glue_morph}, we get a canonical isomorphism $\psi_t(\CF)\isoto \CF_1$ over $X$.\\\\
    To see the commutativity of Diagram \ref{eqn_rank1monbasic}, we may work locally and hence assume $X$ is quasi-compact. Let $m$ and $\iota_m$ be as above. It suffices to show the following diagram is commutative:
\[\begin{tikzcd}
	{i_0^*\mathcal{F}} & {i^*j_*p_{m*}p_m^*\CF\isoto i^*j_*p_{m*}p_m^*\underline{\CF_1}} & {\CF_1} \\
	& {p_*\mathcal{F}}
	\arrow["{\mathrm{sp}}", from=1-1, to=1-2]
	\arrow["{\mathrm{sp}}"', from=1-3, to=1-2]
	\arrow["{p_*(\mathrm{adj}_0^*)}", curve={height=-12pt}, from=2-2, to=1-1]
	\arrow["\sim"', curve={height=-12pt}, from=2-2, to=1-1]
	\arrow["{p_*(\mathrm{adj}_1^*)}"'{pos=0.6}, curve={height=12pt}, from=2-2, to=1-3]
\end{tikzcd}\]
where the last term on the first line is identified with $i^*_0\underline{\CF_1}$ and $p_*\underline{\CF_1}$. Note that $j_*p_{m*}p_m^*\CF$ (resp. $j_*p_{m*}p_m^*\underline{\CF_1}$) is monodromic, so $i^*j_*p_{m*}p_m^*\CF\simeq p_*p_{m*}p_m^*\CF$ (resp. $i^*j_*p_{m*}p_m^*\underline{\CF_1}\simeq p_*p_{m*}p_m^*\underline{\CF_1}$). Then, for the left half of the diagram, the composition $p_*\CF\isoto i_0^*\CF\rightarrow p_*p_{m*}p_m^*\CF$ coincides with $p_*(\adj)$, where $\adj$ denotes $\id\rightarrow p_{m*}p_m^*$. Note that (i) the map $\iota_m$ restricts to $\id$ on $1\X 1_X$, and (ii) $p_*p_{m*}p_m^*\underline{\CF_1}\simeq \CF_1\oplus \CF_1(-1)[-1]$ (for example, by cohomological purity). It follows that the further composition $p_*\CF\rightarrow p_*p_{m*}p_m^*\CF\simeq p_*p_{m*}p_m^*\underline{\CF_1}\simeq \CF_1\oplus \CF_1(-1)[-1]$ coincides with $p_*\CF\xrightarrow{p_*(\adj_1^*)}\CF_1\hookrightarrow\CF_1\oplus \CF_1(-1)[-1]$, where the second arrow is inclusion to the first factor. This is just the right half of the diagram, we are done.\\\\
    (2) Diagram \ref{eqn_rank1monbasic} gives a distinguished triangle $i_0^*\CF\rightarrow s^*i_1^*\CF\rightarrow\phi_t(\CF)\rightarrow$, where the first arrow is define by $i_0^*\CF\xleftarrow[\sim]{p_*(\mathrm{adj}_0^*)}p_*\CF\xrightarrow{p_*(\mathrm{adj}_1^*)}s^*i^*_1\CF$. Apply $\DD$, we get $s^*i_1^!\DD\CF\rightarrow i_0^!\DD\CF \rightarrow\DD\Phi_t(\CF)\rightarrow$, where the first arrow coincides with the one defined by $s^*i_1^!\DD\CF\simeq p_!i_{1*}i_1^!\DD\CF\xrightarrow{p_!(\mathrm{adj}_1^!)}p_!\DD\CF\xleftarrow[\sim]{p_!(\mathrm{adj}_0^!)} i^!_0\DD\CF$. Replace $\CF$ by $\DD\CF$ and apply Theorem \ref{thm_main_psi_phi}.4, we get $s^*i_1^!\CF\rightarrow i_0^!\CF\rightarrow(\Phi_t\CF)(-1)\rightarrow$.
\end{proof} 

\section{Monodromic Fourier transform}\label{sec_fourier}

The microlocalisation will be defined as the monodromic Fourier transform of the specialisation. In this section, we study this Fourier transform. The set-up is as in §\ref{sec_intro}.\\

In the complex D-module and positive characteristic $\ell$-adic contexts, where wild ramifications exist, one has the “full” Fourier transform, using as kernels the exponential D-module and an Artin–Schreier sheaf, respectively. In the context of real manifolds, where wild ramifications do not exist, ones does not have a full Fourier transform\footnote{However, see \cite{scholze_wild_2025}.}, but has the Fourier-Sato transform for conic sheaves, which is sufficient for microlocal purposes. Our situation is similar: as we are working with finite coefficients, there is no local wild ramification\footnote{See \cite{ramero_hasse-arf_2012} for the ind-finite coefficient case, where there are local wild ramifications and a full Fourier transform is defined and studied.}, but restricted to monodromic sheaves, we do have the following monodromic Fourier transform, see \cite{wang_new_2015} and the references therein.

\begin{definition-lemma}[\cite{wang_new_2015}]\label{def_fourier}
    Let $E$ be a vector bundle over a quasi-separated rigid analytic variety $X$. The \underline{monodromic Fourier transform} is the functor $\F: D_{mon}(E)\rightarrow D_{mon}(E')$, $\CF\mapsto p'_!(p^*\CF\otimes m^*B)$. The maps are as in the following diagram, $E'$ is the dual bundle of $E$, $B$ is the sheaf $u_*u^*\underline{\Lambda}$ on $\AAA^1$ where $u$ is the open immersion $\AAA^1-1\hookrightarrow\AAA^1$.
\[\begin{tikzcd}
	& {E\times_XE'} && {\AAA^1} \\
	E && {E'}
	\arrow["{m\,\,\mathrm{(pairing)}}", from=1-2, to=1-4]
	\arrow["p"', from=1-2, to=2-1]
	\arrow["{p'}", from=1-2, to=2-3]
\end{tikzcd}\]
\end{definition-lemma}

\begin{proof}
    The statement to be proved is that $\F\CF$ is monodromic for monodromic $\CF$. First we show $\F\CF$ is Zariski-constructible. Let $\overline{E}=\PP(E\oplus 1)$. Consider the following diagram:
\[\begin{tikzcd}
	{E'} & {\overline{E}\times_X E'} & {E\times_X E'} \\
	& {\overline{E}} & E
	\arrow["{\bar{p}'}"', from=1-2, to=1-1]
	\arrow["{\bar{p}}"', from=1-2, to=2-2]
	\arrow["{\overline{j}}"', hook', from=1-3, to=1-2]
	\arrow["\lrcorner"{anchor=center, pos=0.125, rotate=-90}, draw=none, from=1-3, to=2-2]
	\arrow["p", from=1-3, to=2-3]
	\arrow["j", hook', from=2-3, to=2-2]
\end{tikzcd}\]
Since $j_!\CF$ is Zariski-constructible (Remark \ref{rmk_defofmon}.2.(a)), $\overline{j}_!p^*\CF\simeq\bar{p}^*j_!\CF$ also. Since $\overline{j}_!m^*B$ is Zariski-constructible (this is clear when $X=\mathrm{pt}$ because then the situation algebraisable, the general case is locally on $X$ a pullback from the point case), we get $\overline{j}_!(p^*\CF\otimes m^*B)\simeq\overline{j}_!p^*\CF\otimes\overline{j}_!m^*B$ is Zariski-constructible. So $p'_!(p^*\CF\otimes m^*B)\simeq\bar{p}'_!\overline{j}_!(p^*\CF\otimes m^*B)$ is Zariski-constructible.\\\\
Next we show $\F\CF$ is monodromic by verifying Corollary \ref{cor_fundamental_of_mon}.2. For $\lambda\in K^\X$, let $\theta_\lambda$ (resp. $\theta'_\lambda$) denote the $\lambda$-scaling on $E$ (resp. $E'$). We will also use $\theta_\lambda$ and $\theta'_\lambda$ to denote the scaling on $E\X_X E'$ in the $E$ and $E'$ directions, respectively. Then $\theta'^*_\lambda p'_!(p^*\CF\otimes m^*B)\simeq p'_!\theta'^*_\lambda(p^*\CF\otimes m^*B)\simeq p'_!(p^*\CF\otimes m^*\lambda^*B)\simeq p'_!\theta^*_\lambda(p^*\theta^*_{\lambda^{-1}}\CF\otimes m^*B)\simeq p'_!(p^*\theta^*_{\lambda^{-1}}\CF\otimes m^*B)\isoto p'_!(p^*\CF\otimes m^*B)$. Here $\lambda^*$ denotes the pullback on $\AAA^1$ under $\lambda$-scaling, and we used the monodromicity of $\CF$ in the last step.
\end{proof}

\begin{theorem}[{\cite[\nopp 3.8]{wang_new_2015}}]
    Let $E$ be a vector bundle with constant rank $d$ over a quasi-separated rigid analytic variety $X$. Let $I^0\in \mathrm{Pro}(\Sh(\mathbf{G}_{m}))$ be the pro-object \textup{“}$\varprojlim_n\!\!\!$\textup{”}$e_{n_!}e_n^!\underline{\Lambda}$, where $e_n: \Gm\rightarrow\Gm$ is the $n$-th power map and $j:\mathbf{G}_{m}\hookrightarrow\AAA^1$ is the inclusion. Define the functor $\K: D_{mon}(E')\rightarrow \mathrm{Pro}(D_{mon}(E))$, $\CG\mapsto p_!(p'^*\CG\otimes m^*j_* I^0)$. Then, $\K$ lands in $D_{mon}(E)$ (i.e., $\K\CG$ is essentially constant for $\CG\in D_{mon}(E')$), and $\K(d+1)[2d+1]$ is an inverse to $\F$.
\end{theorem}


\begin{proof}
    The proof in \cite[\nopp 3.8]{wang_new_2015} is applicable in our context, taking into account the following comments: (a) Remark \ref{rmk_defofmon}.2.(b) implies results in Appendix A and 3.9 in \textit{loc. cit.} hold in our context; (b) the sheaf $\rho_!\mathrm{pr}_1^*j^*B$ in the proof of Theorem 2.1 in \textit{loc. cit.} lies in $\Dbzc(X)$. This can be seen as follows: it suffices to check locally on $X$, so we may assume $E$ is a trivial bundle. One can then further reduce to the case $X=\pt$ using proper base change. The situation is now algebraisable, and the statement is clear.
\end{proof}

\begin{lemma}\label{lem_F_base_change}
    Let $f: Y\rightarrow X$ be a map of quasi-separated rigid analytic varieties, $E\rightarrow X$ be a vector bundle, $E_Y$ be the base change to $Y$, and $\widetilde{f}: E_Y\rightarrow E$ and $\widetilde{f}': E'_Y\rightarrow E'$ be the induced maps. Then, for $\CF\in D_{mon}(E)$, there is a canonical isomorphism $\widetilde{f}'^*\F\CF\simeq\F\widetilde{f}^*\CF$; for $\CG\in D_{mon}(E_Y)$ such that $\widetilde{f}_!\CG\in\Dbzc(E)$, there is a canonical isomorphism $\widetilde{f}'_!\F\CG\simeq \F\widetilde{f}_!\CG$.
\end{lemma}

\begin{proof}
    These are formal consequences of proper base change and the projection formula. We omit the details.
\end{proof}

\begin{lemma}\label{lem_fourierfunctoriality}
    Let $\alpha: \tE\rightarrow E$ be a map of vector bundles over a quasi-separated rigid analytic variety $X$. Assume $\tE$ (resp. $E$) has constant rank $\widetilde{d}$ (resp. $d$). Denote $\delta=\widetilde{d}-d$. Let $\CF\in D_{mon}(\tE)$ and $\CG\in D_{mon}(E)$. Then there are canonical isomorphisms: 
    (1) $\F\alpha_!\CF\simeq\alpha'^*\widetilde{\F}\CF$,
    (2) $\F\alpha_*\CF\simeq\alpha'^!\widetilde{\F}\CF\langle\delta\rangle$,
    (3) $\widetilde{\F}\alpha^!\CG\simeq\alpha'_*\F\CG$,
    (4) $\widetilde{\F}\alpha^*\CG\langle\delta\rangle\simeq\alpha'_!\F\CG$.
\end{lemma}

In this proof and the following, we will implicitly use Lemma \ref{lem_monfunctoriality} and Definition-Lemma \ref{def_fourier}, which imply all relevant sheaves are monodromic.

\begin{proof}
    (1). This is standard, using the following diagram (the top left and bottom right “squares” are Cartesian), proper base change, and the projection formula. 
\[\begin{tikzcd}
	{\widetilde{E}'} && {E'} \\
	{\widetilde{E}\times_X\widetilde{E}'} & {\widetilde{E}\times_X E'} & {E\times_X E'} \\
	{\widetilde{E}} && E
	\arrow["{\alpha'}"', from=1-3, to=1-1]
	\arrow[from=2-1, to=1-1]
	\arrow[from=2-1, to=3-1]
	\arrow["\lrcorner"{anchor=center, pos=0.125, rotate=180}, draw=none, from=2-2, to=1-1]
	\arrow[from=2-2, to=1-3]
	\arrow[from=2-2, to=2-1]
	\arrow[from=2-2, to=2-3]
	\arrow[from=2-2, to=3-1]
	\arrow["\lrcorner"{anchor=center, pos=0.125}, draw=none, from=2-2, to=3-3]
	\arrow[from=2-3, to=1-3]
	\arrow[from=2-3, to=3-3]
	\arrow["\alpha"', from=3-1, to=3-3]
\end{tikzcd}\]\\
(4). A similar argument as in (1) gives $\widetilde{\K}\alpha'_!\CH\simeq\alpha^*\K\CH$ in $\mathrm{Pro}(D_{mon}(\widetilde{E}))$, for $\CH\in D_{mon}(E')$. As both sides are essentially constant, this is actually an isomorphism in $D_{mon}(\widetilde{E})$. Take $\CH=\F\CG$, apply $\widetilde{\F}$ on both sides, use $KF\simeq\id (-d-1)[-2d-1]$ and $\widetilde{F}\widetilde{K}\simeq\id(-\widetilde{d}-1)[-2\widetilde{d}-1]$, we get (4). Note a similar argument gives $\alpha_!\widetilde{\K}\CH\langle\delta\rangle\simeq\K\alpha'^*\CH$ for $\CH\in D_{mon}(\widetilde{E}')$.\\\\
(2). This follows from the following computation: $\forall A\in D_{mon}(E')$, $RHom(A, \F\alpha_*\CF)\simeq RHom(\K A,$ $\alpha_*\CF[-2d-1](-d-1))\simeq RHom(\alpha^*\K A,\CF[-2d-1](-d-1))\simeq RHom(\widetilde{\K}\alpha'_!A,\CF[-2d-1](-d-1))\simeq RHom(\alpha'_!A, \widetilde{\F}\CF\langle\delta\rangle)\simeq RHom(A, \alpha'^!\widetilde{\F}\CF\langle\delta\rangle)$.\\\\
(3). A similar computation as in (2) proves (3). Note one needs to use $\alpha_!\widetilde{\K}\langle\delta\rangle\simeq\K\alpha'^*$ mentioned in the proof of (4).
\end{proof}

\begin{lemma}\label{lem_rank1monbasic_F}
    Let $E=\AAA_X^1$ be the trivial bundle over a quasi-separated rigid analytic variety $X$, and $\CF\in D_{mon}(E)$. Use the same notations as in Proposition \ref{lem_rank1monbasic}. Denote the 1-section of the dual bundle of $\AAA^1_X$ by $1'_X$. Then there is an isomorphism $(\F\CF)_{1'_X}\cong(\Phi_t\CF)(-1)$, under which the equivariant monodromy of $(\F\CF)_{1'_X}$ coincides with the monodromy of $(\Phi_t\CF)(-1)$.
\end{lemma}

\begin{proof}
    Note, by the definition of $\F$ and the proof of the monodromicity of $\F\CF$, we have $(\F\CF)_{1'_X}\simeq p_!(\CF\otimes t^*B)$ with monodromies reversed. The standard recollement and cohomological purity give a distinguished triangle $i_1^!\underline{\Lambda}_{\AAA^1_X}\rightarrow \underline{\Lambda}_{\AAA^1_X}\rightarrow t^*B\rightarrow$. Apply $\CF\otimes(-)$ and $p_!$, we get $p_!(\CF\otimes i_1^!\underline{\Lambda}_{\AAA^1_X})\rightarrow p_!\CF\rightarrow p_!(\CF\otimes t^*B)\rightarrow$. Note $i^!_1\CF\simeq \CF\otimes i_1^!\underline{\Lambda}_{\AAA^1_X}$ and $i_0^!\CF\isoto p_!\CF$, we get $s^*i_1^!\CF\rightarrow i_0^!\CF\rightarrow p_!(\CF\otimes t^*B)\rightarrow$, commuting with monodromies. Tracing through the definitions, one sees that the first arrow coincides with the one in Proposition \ref{lem_rank1monbasic}.2. The conclusion then follows from that proposition.
\end{proof}

\begin{lemma}
    Let $E$ be a vector bundle over a quasi-separated rigid analytic variety $X$. Consider the following diagram 
\[\begin{tikzcd}
	& {E\times_XE'} & {\mathbf{A}^1\times E'} & {\mathbf{A}^1} & {} \\
	E && {E'}
	\arrow["\omega", from=1-2, to=1-3]
	\arrow["p"', from=1-2, to=2-1]
	\arrow["{p'}", from=1-2, to=2-3]
	\arrow["t", from=1-3, to=1-4]
	\arrow[from=1-3, to=2-3]
\end{tikzcd}\]
where $\omega: (v,v')\mapsto (\langle v,v'\rangle, v')$ and $t$ is the projection to $\AAA_K^1$. Then, for $\CF\in D_{mon}(E)$, there are natural isomorphisms $\F\CF\simeq (\F_{\AAA_{E'}^1}\omega_!p^*\CF)|_{1'_{E'}}\simeq (\F_{\AAA_{E'}^1}\omega_*p^*\CF)|_{1'_{E'}}$. Here, $\F$ denotes the Fourier transform on $E$ over $X$, $\F_{\AAA_{E'}^1}$ denotes the Fourier transform on the trivial bundle $\AAA_{E'}^1$ \emph{over $E'$}, and $1_{E'}$ denotes the 1-section on the dual of $\AAA^1_{E'}$, identified with $E'$. Furthermore, there are isomorphisms $(\F_{\AAA_{E'}^1}\omega_!p^*\CF)|_{1'_{E'}}\cong (\Phi_t\omega_!p^*\CF)(-1)$ and $(\F_{\AAA_{E'}^1}\omega_*p^*\CF)|_{1'_{E'}}\cong (\Phi_t\omega_*p^*\CF)(-1)$.
\end{lemma}

\begin{proof}
    We may assume $E$ has constant rank $d$. View $\omega$ as a map of vector bundles over $E'$, apply Lemma \ref{lem_fourierfunctoriality}.1 to $\F_{\AAA_{E'}^1}\omega_!p^*\CF$, we get $(\F_{\AAA_{E'}^1}\omega_!p^*\CF)|_{1'_{E'}}$ $\simeq (\omega'^*\F_{E_{E'}}p^*\CF)|_{1'_{E'}}\simeq(\F_{E_{E'}}p^*\CF)|_\Delta\simeq \F\CF$. Here $E_{E'}$ is $E\X_XE'$ viewed as a bundle over $E'$, and $\Delta$ denotes the diagonal in $E'_{E'}:=E'\X_XE'$, identified with $1_{E'}$ (via $\omega'$) and $E'$. A similar argument, using Lemma \ref{lem_fourierfunctoriality}.2, shows $\F\CF\langle-1\rangle\simeq(\F_{\AAA_{E'}^1}\omega_*p^*\CF)|^!_{1'_{E'}}$, where $(-)|^!_{1'_{E'}}$ denotes !-restriction to ${1'_{E'}}$. Cohomological purity (\cite[\nopp 3.9.1]{huber_etale_1996}, \cite[\nopp XVI 3.8, 3.10.(a)]{SGA4}) then gives $\F\CF\simeq(\F_{\AAA_{E'}^1}\omega_*p^*\CF)|_{1'_{E'}}$. Finally, the “Furthermore” part follows from Lemma \ref{lem_rank1monbasic_F}.
\end{proof}

\begin{corollary}\label{cor_F_perverse_D}
    Let $E$ be a vector bundle of constant rank $d$ over a quasi-separated rigid analytic variety $X$.\\
    (1) Assume $\Lambda=\FF_{\ell^r}$. Then $\F[d]$ is perverse t-exact.\\
    (2) For $\CF\in D_{mon}(E)$, there is an isomorphism $\DD\F\CF\cong(\F\DD\CF)(-1)\langle d\rangle$.
\end{corollary}

Recall the perverse t-structure on $D_{mon}(E)$ is the one induced from $\Dbzc(E)$ (Remark \ref{rmk_defofmon}.3).

\begin{proof}
    (1) First note that the map $\omega: E\X_XE'\rightarrow E'$ is weakly Stein (Definition \ref{def_w_stein}). Indeed, $X$ can be covered by affinoid opens $U$ over which $E$ is trivial, consequently $E_U\X_UE'_U$ and $E'_U$ are both weakly Stein. The perverse t-exactness of $\F\CF[d]$ then follows from $\F\CF[d]\cong\Phi_t\omega_!p^*\CF[d]\cong\Phi_t\omega_*p^*\CF[d]$, the perverse t-exactness of $p^*[d]$ and $\Phi_t$ (Theorem \ref{thm_main_psi_phi}.2), and the perverse right (resp. left) t-exactness of $\omega_!$ (resp. $\omega_*$) (Proposition \ref{prop_artinvanishing}).\\\\
    (2) This follows from $\F\CF\cong(\Phi_t\omega_!p^*\CF)(-1)\cong(\Phi_t\omega_*p^*\CF)(-1)$, the commutativity of $\Phi_t$ and $\DD$ (Theorem \ref{thm_main_psi_phi}.4), $\DD\omega_!\simeq\omega_*\DD$, and $\DD p^*\simeq p^!\DD\simeq p^*\langle d\rangle \DD$.
\end{proof}

\section{Specialisation}\label{sec_special}

In this and the next two sections, we study the specialisation, microlocalisation and micro-hom. The classical references are Chapter IV, §VIII.6, and Chapter X of \cite{kashiwara_sheaves_1990}. The set-up is as in §\ref{sec_intro}.\\

The specialisation is a version of nearby cycles for higher codimensions and without the choice of a function. This is made precise by Proposition \ref{lem_nu&psi}.

\begin{definition-lemma}[{\cite[\nopp §8]{verdier_specialisation_1983}}]\label{def_specialisation}
    Let $Z\hookrightarrow X$ be a closed immersion of quasi-separated rigid analytic varieties. The \underline{specialisation} functor is defined as $\nu_Z$: $\DX\rightarrow D_{mon}(T_ZX)$, $\CF\mapsto \psi_t\bar{p}^*\CF$, where $T_ZX$ is the normal cone of $Z$ in $X$. The maps are as in the following diagram of deformation to the normal cone.
\begin{equation}\label{eqn_deformation_normal_cone}      
\begin{tikzcd}
	{T_ZX} & {\widetilde{T_ZX}} & {X\times\mathbf{G}_m} & X \\
	0 & {\mathbf{A}^1} & {\mathbf{G}_m}
	\arrow["\iota", hook, from=1-1, to=1-2]
	\arrow[from=1-1, to=2-1]
	\arrow["\lrcorner"{anchor=center, pos=0.125}, draw=none, from=1-1, to=2-2]
	\arrow["{\overline{p}}", curve={height=-18pt}, from=1-2, to=1-4]
	\arrow["t", from=1-2, to=2-2]
	\arrow[hook', from=1-3, to=1-2]
	\arrow["p", from=1-3, to=1-4]
	\arrow["\lrcorner"{anchor=center, pos=0.125, rotate=-90}, draw=none, from=1-3, to=2-2]
	\arrow[from=1-3, to=2-3]
	\arrow[hook, from=2-1, to=2-2]
	\arrow[hook', from=2-3, to=2-2]
\end{tikzcd}
\end{equation}
\end{definition-lemma}

Recall $\widetilde{T_ZX}:=\Bl_{\{Z\X0\}}(X\X\AAA^1)-\overline{(X-Z)\X0}=\underline{\Spa}_X(\oplus_{n\in\ZZ}t^nI^{-n})$, and $T_ZX:=\underline{\Spa}_X(\oplus_{n\in\NN}\frac{I^n}{I^{n+1}})$. Here $I$ is the coherent ideal corresponding to $Z$ in $X$, and $I^{-n}:=\CO_X$ for $n\geq 0$.

\begin{proof}
    We need to show $\nu_Z(\CF)$ is indeed monodromic. That $\nu_Z(\CF)$ is Zariski-constructible follows from the Zariski-constructibility of $\bar{p}^*\CF$ and the fact that nearby cycles preserve Zariski-constructible sheaves (Proposition \ref{prop_phipreserveszc}). To show monodromicity, we first reduce to the case when $Z$ is a Cartier divisor (\textit{i.e.} locally cut out by a single function). Consider the following blow-up diagram:
\[\begin{tikzcd}
	{T_H\overline{X}-H} & {T_H\overline{X}} & H & {\overline{X}:=\mathrm{Bl}_ZX} \\
	{T_ZX-Z} & {T_ZX} & Z & X
	\arrow[hook, from=1-1, to=1-2]
	\arrow["{\mathring{\pi}}"', from=1-1, to=2-1]
	\arrow["\simeq", from=1-1, to=2-1]
	\arrow["\lrcorner"{anchor=center, pos=0.125}, draw=none, from=1-1, to=2-2]
	\arrow[from=1-2, to=1-3]
	\arrow["{\widetilde{\pi}}"', from=1-2, to=2-2]
	\arrow["\lrcorner"{anchor=center, pos=0.125}, draw=none, from=1-2, to=2-3]
	\arrow[hook, from=1-3, to=1-4]
	\arrow[from=1-3, to=2-3]
	\arrow["\lrcorner"{anchor=center, pos=0.125}, draw=none, from=1-3, to=2-4]
	\arrow["\pi"', from=1-4, to=2-4]
	\arrow[hook, from=2-1, to=2-2]
	\arrow[from=2-2, to=2-3]
	\arrow[hook, from=2-3, to=2-4]
\end{tikzcd}\]
Note $H$ is a Cartier divisor in $\overline{X}$. The map $\pi$ induces a proper map $\widetilde{T_H\overline{X}}\rightarrow\widetilde{T_ZX}$ over $\AAA^1$ (see the proof of Lemma \ref{lem_nufunctoriality}.1 for the proof of properness in a more general situation). Since nearby cycles commute with proper pushforwards, we get $\widetilde{\pi}_*\nu_H(\pi^*\CF)\simeq \nu_Z(\pi_*\pi^*\CF)$. As $\nu_Z(-)|_{T_ZX-Z}$ only depends on $(-)|_{X-Z}$, and $\pi$ is an isomorphism on over $X-Z$, we get $j_!(\nu_Z(\CF)|_{T_ZX-Z})\simeq\widetilde{\pi}_*\overline{j}_!(\nu_H(\pi^*\CF)|_{T_H\overline{X}-H})$. If $\nu_H(\pi^*\CF)$ satisfies the condition in Corollary \ref{cor_fundamental_of_mon}.2, so does $\overline{j}_!(\nu_H(\pi^*\CF)|_{T_H\overline{X}-H})$, hence $j_!(\nu_Z(\CF)|_{T_ZX-Z})$ (since $\mathring{\pi}$ is an isomorphism commuting with scaling). As $\nu_Z(\CF)|_Z$ is monodromic, this will imply $\nu_Z(\CF)$ is monodromic by Corollary \ref{cor_dualmono}.1.\\\\
We now prove the case when $Z$ is a Cartier divisor. We may work locally, and assume $Z\hookrightarrow X$ is of the form $V(f)\hookrightarrow \Spa(A)$, for some $f\in A$. Consider the following diagram:
\begin{equation}\label{eqn_alpha_f_cross}
\begin{tikzcd}
	{X\times\mathbf{G}_m} & {\widetilde{T_ZX}-(Z\times\mathbf{A}^1)} \\
	& {\mathbf{A}^1}
	\arrow["\alpha", from=1-1, to=1-2]
	\arrow["{f^\times}"', shorten >=7pt, from=1-1, to=2-2]
	\arrow["t", from=1-2, to=2-2]
\end{tikzcd}
\end{equation}
where $\alpha: (x,s)\mapsto (x,s f(x), s)$ (the second and third entries refer to $t$ and $\frac{t}{f}$ coordinates on $\widetilde{T_ZX}=\Bl_{\{t=f=0\}}(X\X\AAA^1)-\overline{(X-Z)\X0}$, respectively), and $f^\X$ is the pullback of $t$ to $X\X\Gm$. The map $\alpha$ is an isomorphism, so $\psi_{f^\X}(\alpha^*\bar{p}^*\CF)\simeq \alpha^*\psi_t(\bar{p}^*\CF)$, where the $\alpha$ on the right hand side is the restriction to $Z\X\Gm\isoto T_ZX-Z$. Denote $\alpha^*\bar{p}^*\CF\simeq \mathrm{pr}_1^*\CF$ by $\CF^\X$, where $\mathrm{pr}_1: X\X\Gm\rightarrow X$ is the projection. It suffices to show $\psi_{f^\X}(\CF^\X)$ is monodromic\footnote{More precisely, to show the !-extension of $\psi_{f^\X}(\CF^\X)$ to $X\X\AAA^1$ is monodromic.}. We verify Corollary \ref{cor_fundamental_of_mon}.2: let $\lambda\in K^\X$, then $\theta_\lambda^*\psi_{f^\X}(\CF^\X)\simeq \psi_{\lambda f^\X}(\theta_\lambda^*\CF^\X)\simeq\psi_{\lambda f^\X}(\CF^\X)\cong\psi_{f^\X}(\CF^\X)$. Here the second step uses the canonical isomorphism $\theta_\lambda^*\CF^\X\simeq\CF^\X$, and the third step uses Lemma \ref{lem_psietalebasechange}.
\end{proof}

\begin{lemma}\label{lem_nu&pervdual}
    Let $Z\hookrightarrow X$ be a closed immersion of quasi-separated rigid analytic varieties.\\
    (1) Assume $\Lambda=\FF_{\ell^r}$. Then $\nu_Z$ is perverse t-exact.\\
    (2) There is a canonical isomorphism $\DD\nu_Z\simeq\nu_Z\DD$.
\end{lemma}

\begin{proof}
    As $\nu_Z\simeq\psi_tp^*$, the claims follow from: (a) $\psi_t[-1]$ is perverse t-exact and commutes with duality, (b) $p$, being smooth, is shifted perverse t-exact and $\DD p^*\simeq p^!\DD\simeq p^*\langle1\rangle\DD$.
\end{proof}

\begin{lemma}\label{lem_sommetdenu}
    Let $i: Z\hookrightarrow X$ be a closed immersion of quasi-separated rigid analytic varieties, $\CF\in \DX$, and $i_0: Z\hookrightarrow T_ZX$ be the 0-section. Then there are canonical isomorphisms $i^*\CF\isoto i_0^*\nu_Z(\CF)$ and $i_0^!\nu_Z(\CF)\isoto i^!\CF$.
\end{lemma}

\begin{proof}
    It suffices to construct $i^*\CF\isoto i_0^*\nu_Z(\CF)$, then $i_0^!\nu_Z(\CF)\isoto i^!\CF$ is obtained by duality (Lemma \ref{lem_nu&pervdual}.2, Corollary \ref{cor_dualmono}.2). The map $i^*\CF\rightarrow i_0^*\nu_Z(\CF)$ is $i^*_0$ applied to $\mathrm{sp}: (\bar{p}^*\CF)|_{t=0}\rightarrow\psi_t(\bar{p}^*\CF)$. We show it is an isomorphism. Using the localisation triangle, we reduce two cases: (i) $\CF$ is supported on $Z$, (ii) $i^*\CF=0$.\\\\
    Case (i). It suffices to show the map $\mathrm{sp}: \CG\rightarrow\psi_tq^*\CG$ is an isomorphism for every $\CG\in \Dbzc(Z)$, where $q$ and $t$ are the projections from $Z\X\AAA^1$ to the $Z$ and $\AAA^1$ factors, respectively. Recall the definition of $\psi_t$ (Definition \ref{def_van}). Consider the following diagram:
\[\begin{tikzcd}
	{Z\X0} & {Z\X0(n)} & {Z\X\AAA^1(n)} & {Z\X\Gm(n)} \\
	& {Z\X0} & {Z\X\AAA^1}
	\arrow[hook, from=1-1, to=1-2]
	\arrow[from=1-1, to=2-2]
	\arrow[hook, from=1-2, to=1-3]
	\arrow[from=1-2, to=2-2]
	\arrow["\lrcorner"{anchor=center, pos=0.125}, draw=none, from=1-2, to=2-3]
	\arrow["{\overline{p}_n}", from=1-3, to=2-3]
	\arrow["{s_n}"', hook', from=1-4, to=1-3]
	\arrow["{j_n}", from=1-4, to=2-3]
	\arrow["i"', hook, from=2-2, to=2-3]
\end{tikzcd}\]
Here $\overline{p}_n$ is $(z,\lambda)\mapsto (z,\lambda^n)$. Denote the composition $Z\X0\hookrightarrow Z\X\AAA^1(n)$ by $i_n$ and the composition $Z\X\Gm(n)\hookrightarrow Z\X\AAA^1(n)\xrightarrow{\pr} Z\X0$ by $r_n$. We identify $Z$ with the two $(Z\X0)$'s in the diagram. Then, $\psi_tq^*\CG= \varinjlim_n i^*j_{n*}j_n^* q^*\CG\simeq \varinjlim_n i^*\overline{p}_{n*}s_{n*}r_n^*\CG\simeq \varinjlim_n i_n^*s_{n*}r_n^*\CG$. As $s_{n*}r_n^*\CG$ is monodromic, by Lemma \ref{lem_hyp_localisation}, $i_n^*s_{n*}r_n^*\CG\simeq r_{n*}r_n^*\CG$. By cohomological purity (\cite[\nopp 3.9.1]{huber_etale_1996}, \cite[\nopp XVI 3.8, 3.10.(a)]{SGA4}), $r_{n*}r_n^*\CG\simeq \CG\otimes (\underline{\Lambda}_Z\oplus \underline{\Lambda}_Z(-1)[-1])$, the transition maps (when $n$ varies) are induced by the identity on the $\underline{\Lambda}_Z$ factor and multiplication by the degree on the $\underline{\Lambda}_Z(-1)[-1]$ factor. So $\psi_tq^*\CG\simeq \varinjlim_n(\CG\otimes (\underline{\Lambda}_Z\oplus \underline{\Lambda}_Z(-1)[-1]))\simeq \CG$. For each $n$, the canonical map $\CG\rightarrow r_{n*}r_n^*\CG$ is identified with the inclusion of the $\underline{\Lambda}_Z$ factor, we conclude that $\mathrm{sp}: \CG\rightarrow\psi_tq^*\CG$ is an isomorphism.\\\\
Case (ii). It suffices to show $i_0^*\nu_Z(\CF)=0$. Since $i_0^*\nu_Z(\CF)$ is Zariski-constructible, it suffices to show its stalk is $0$ at each classical point $x\in Z(K)$. We will use Proposition \ref{prop_nearbyfibinterp}.\\\\
We may assume $X=\Spa(A)$ is affinoid. Let $\{f_1,...,f_r,f_{r+1},...,f_{r+s}\}\subseteq A$ be such that $Z=\Spa (A/(f_1,...,f_r))$ and $x=\Spa(A/(f_1,...,f_r,f_{r+1},...,f_{r+s}))$. Then, on $\widetilde{T_ZX}(=\Bl_{\{Z\X0\}}(X\X\AAA^1)-\overline{(X-Z)\X0})$, we have functions $\{t, \frac{f_1}{t},...,\frac{f_r}{t},f_{r+1},...,f_{r+s}\}$ (we abuse notations and denote the functions and their pullbacks under $\overline{p}: \widetilde{T_ZX}\rightarrow X$ by the same letters), and $x\X0= V(t, \frac{f_1}{t},...,\frac{f_r}{t},f_{r+1},...,f_{r+s})$. Let $\epsilon$ and $\eta$ be as in Proposition \ref{prop_nearbyfibinterp} (applied to $(\psi_t\overline{p}^*\CF)_{x\X0}$). Then $(\psi_t\overline{p}^*\CF)_{x\X0}\cong R\Gamma(U_{\epsilon,\eta},\overline{p}^*\CF)$, where $U_{\epsilon,\eta}:=\{t=\eta, |\frac{f_1}{t}|\leq\epsilon,...,|\frac{f_r}{t}|\leq\epsilon, |f_{r+1}|\leq \epsilon,..., |f_{r+1}|\leq \epsilon\}$. In the following, we identify $U_{\epsilon,\eta}$ with $\{ \frac{|f_1|}{\eta}\leq\epsilon,...,\frac{|f_r|}{\eta}\leq\epsilon,..., |f_{r+1}|\leq \epsilon\}\subseteq X$, and write $R\Gamma(U_{\epsilon,\eta},\overline{p}^*\CF)$ as $R\Gamma(U_{\epsilon,\eta},\CF)$. Apply \cite[\nopp 3.6]{huber_finiteness_1998} to $x=V(f_1,...,f_r,f_{r+1},...,f_{r+s})\in X$, we get an $\epsilon'>0$ such that the restriction map $R\Gamma(U'_{\epsilon'},\CF)\rightarrow\CF_x$ is an isomorphism, where $U'_{\epsilon'}:=\{|f_1|\leq\epsilon',...,|f_{r+s}|\leq \epsilon'\}\subseteq X$. We may assume $\epsilon'<\epsilon$. By Proposition \ref{prop_nearbyfibinterp}, there exists $\eta'>0$ such that $(\psi_t\overline{p}^*\CF)_{x\X0}\cong R\Gamma(U_{\epsilon',\eta'},\CF)$ (the meaning of $U_{\epsilon',\eta'}$ and the notation is as above for $U_{\epsilon,\eta}$). Now consider the restriction maps $R\Gamma(U_{\epsilon,\eta},\CF)\xrightarrow{a} R\Gamma(U'_{\epsilon'},\CF)\xrightarrow{b} R\Gamma(U_{\epsilon',\eta'},\CF)\xrightarrow{c} \CF_x$. By the proof of Proposition \ref{prop_nearbyfibinterp}, the isomorphisms $(\psi_t\overline{p}^*\CF)_{x\X0}\cong R\Gamma(U_{\epsilon,\eta},\CF)$ and $(\psi_t\overline{p}^*\CF)_{x\X0}\cong R\Gamma(U_{\epsilon',\eta'},\CF)$ can be made compatible with the restriction $b\circ a$. So $b\circ a$ is an isomorphism. Since $c\circ b$ is an isomorphism, it follows that $a$, $b$ and $c$ are isomorphisms. This completes the proof.
\end{proof}

\begin{remark}
    One might be able to directly apply the method for Case (ii) to a general $\CF$. The subtlety is that one needs show the isomorphism obtained between $(\psi_t\overline{p}^*\CF)_{x\X0}$ and $\CF_x$ can be made compatible with the canonical map $\mathrm{sp}: \CF\rightarrow\psi_tq^*\CF$.
\end{remark}

\begin{proposition}\label{lem_nu&psi}
    Let $Z=V(f)\hookrightarrow X$ be a Cartier divisor in a quasi-separated rigid analytic variety. $f: X\rightarrow\AAA^1$ induces a map $\widetilde{f}: T_ZX\rightarrow T_0\AAA^1\simeq\AAA^1$, which induces an isomorphism $T_ZX\isoto Z\times\AAA^1$. Denote the section $Z\hookrightarrow Z\times\AAA^1$, $z\mapsto(z,1)$ by $s_f$. Then:\\
    (1) for $\CF\in\DX$, there is a canonical isomorphism $s_f^*\nu_Z(\CF)\isoto\psi_f(\CF)$, under which the equivariant monodromy of $s_f^*\nu_Z(\CF)$ coincides with the monodromy of $\psi_f(\CF)$;\\
    (2) there is a canonical isomorphism $\psi_{\widetilde{f}}(\nu_Z(\CF))\isoto\psi_f(\CF)$ and an isomorphism $\phi_{\widetilde{f}}(\nu_Z(\CF))\cong\phi_f(\CF)$, both commuting with the usual monodromies.
\end{proposition}

\begin{proof}
    (1) We will use Diagram \ref{eqn_alpha_f_cross} and the same notations in that paragraph. Note $s^*_f\nu_Z(\CF)=s^*_f\psi_t(\bar{p}^*\CF)\simeq \psi_{f^\X}(\CF^\X)|_{Z\X1}$ (ignoring monodromies). The same proof as in \cite[350-351]{verdier_specialisation_1983} applies in our case and gives an isomorphism $\psi_{f^\X}(\CF^\X)|_{Z\X1}\isoto\psi_f(\CF)$ commuting with the nearby cycle monodromies. Here the map is induced by $\psi_{f^\X}$ applied to the adjunction: $\CF^\X\rightarrow l_*l^*\CF^\X$, where $l: X\X1\hookrightarrow X\X\Gm$ is the closed immersion. Combine the above, we get $s^*_f\nu_Z(\CF)\isoto\psi_f(\CF)$.\\\\ To see the claim about monodromies: on $\psi_{f^\X}(\CF^\X)$, we have (i) the equivariant monodromy, which is induced by $\theta_\lambda^*\psi_{f^\X}(\CF^\X)\simeq \psi_{\lambda f^\X}(\theta_\lambda^*\CF^\X)\simeq\psi_{\lambda f^\X}(\CF^\X)\isoto\psi_{f^\X}(\CF^\X)$ (the last map depends on the choice of a $\bar{\lambda}\in\varprojlim\mu_n(\lambda)$); and (ii) the usual nearby cycle monodromy. These two monodromies are opposite to each other, essentially because (1) is induced by the $s$ factor in $f^\X(x,s)=sf(x)$ while (2) is induced by $f^\X$ (see Diagram \ref{eqn_alpha_f_cross} for the notations). On the other hand, the isomorphism $s^*_f\nu_Z(\CF)=s^*_f\psi_t(\bar{p}^*\CF)\simeq \psi_{f^\X}(\CF^\X)|_{Z\X1}$ induced by $\alpha$ reverses the equivariant monodromies, because $\alpha|_{Z\X\Gm}: Z\X\Gm\rightarrow (T_ZX-Z)=Z\X\Gm$ reverses the “$\Gm$-direction”. We conclude that the equivariant monodromy of $s_f^*(\nu_Z(\CF))$ is identified with the usual monodromy of $\psi_f(\CF)$.\\\\
    (2) We have $T_ZX\simeq Z\X\AAA^1$ via the map induced by $\widetilde{f}$. Apply Proposition \ref{lem_rank1monbasic}.1 to the monodromic sheaf $\nu_Z(\CF)$, and use $s_f^*\nu_Z(\CF)\isoto\psi_f(\CF)$ from (1), we get $\psi_{\widetilde{f}}(\nu_Z(\CF))\isoto s^*_f\nu_Z(\CF)\isoto\psi_f(\CF)$, with each isomorphism reversing the monodromies. Denote the composition by $b$. The isomorphism $\phi_{\widetilde{f}}(\nu_Z(\CF))\cong\phi_f(\CF)$ is then obtained by taking cones of the left square in the following diagram, to be discussed below: 
    \begin{equation}\label{eqn_i_psi_phi}
\begin{tikzcd}
	{i^*_0\nu_Z(\mathcal{F})} & {\psi_{\widetilde{f}}(\nu_Z(\mathcal{F}))} & {\phi_{\widetilde{f}}(\nu_Z(\mathcal{F}))} \\
	{\mathcal{F}|_Z} & {\psi_f(\mathcal{F})} & {\phi_f(\mathcal{F})}
	\arrow["{\mathrm{sp}}", from=1-1, to=1-2]
	\arrow["\simeq", "a"', from=1-1, to=2-1]
	\arrow[from=1-2, to=1-3]
	\arrow["\simeq", "b"', from=1-2, to=2-2]
	\arrow["\cong", dashed, from=1-3, to=2-3]
	\arrow["{\mathrm{sp}}", from=2-1, to=2-2]
	\arrow[from=2-2, to=2-3]
\end{tikzcd}
    \end{equation}
Here, $i_0$ is the closed immersion of the 0-section in $T_ZX$, $\CF|_Z$ is the restriction of $\CF$ to $Z$, and $a$ is the inverse of the isomorphism in Lemma \ref{lem_sommetdenu}. We need to show the left square is commutative. In addition to the notations in Diagram \ref{eqn_deformation_normal_cone}, we will also use the following notations: $\bar\CF$ denotes $\bar p^*\CF$, $\pi$ denotes the projection $T_ZX\rightarrow Z$, $i_1$ denotes the closed immersion (see Diagram \ref{eqn_alpha_f_cross}) $X\times 1=\{(x,f(x),1)\}\hookrightarrow\widetilde{T_ZX}$ as well as its restriction on $\{t=0\}$: $Z\X1=\{(z,0,1)\}\hookrightarrow T_ZX$.\\\\
First note we have the following two commutative diagrams:
\[\begin{tikzcd}
	{\iota^*{\bar{\CF}}} & {\psi_t\bar{\CF}} & {(\psi_t\bar{\CF})_1} && {i^*_0\iota^*{\bar{\CF}}} & {i^*_0\psi_t\bar{\CF}} \\
	{\iota^*(\bar{\CF}_1)} & {\psi_t(\bar{\CF}_1)} &&& {\pi_*\iota^*{\bar{\CF}}} & {\pi_*\psi_t\bar{\CF}}
	\arrow["{\mathrm{sp}}", from=1-1, to=1-2]
	\arrow["{\iota^*(\mathrm{adj}^*_1)}"', from=1-1, to=2-1]
	\arrow[from=1-2, to=1-3]
	\arrow["{\psi_t(\mathrm{adj}^*_1)}"', from=1-2, to=2-2]
	\arrow[from=1-3, to=2-2]
	\arrow["{i^*_0(\mathrm{sp})}", from=1-5, to=1-6]
	\arrow["{\mathrm{sp}}"', from=2-1, to=2-2]
	\arrow["\simeq", from=2-5, to=1-5]
	\arrow["{\pi_*(\mathrm{sp})}"', from=2-5, to=2-6]
	\arrow["\simeq"', from=2-6, to=1-6]
\end{tikzcd}\]
Here, $\bar\CF_1$ denotes $i_{1*}i_1^*\bar\CF$, similarly for $(\psi_t\bar\CF)_1$, the top right arrow in the left diagram is $\id\rightarrow i_{1*}i_1^*$, the slant arrow is the isomorphism in (1), and the vertical isomorphisms in the right diagram are as in Lemma \ref{lem_hyp_localisation}. The commutativity of the squares follows from the commutativity of $\mathrm{sp}$ with $\mathrm{adj}^*_1$ and with the first isomorphism in Lemma \ref{lem_hyp_localisation}. We leave the details to the reader.\\\\
Apply $\pi_*$ to the left diagram and use the right diagram, we get a commutative diagram:
\[\begin{tikzcd}
	{i^*_0\iota^*{\bar{\CF}}} & {i^*_0\psi_t\bar{\CF}} \\
	{\pi_*\iota^*{\bar{\CF}}} & {\pi_*\psi_t\bar{\CF}} & {\pi_*((\psi_t\bar{\CF})_1)} \\
	{\pi_*\iota^*(\bar{\CF}_1)} & {\pi_*\psi_t(\bar{\CF}_1)}
	\arrow["{i^*_0(\mathrm{sp})}", from=1-1, to=1-2]
	\arrow["\simeq", from=2-1, to=1-1]
	\arrow["{\pi_*(\mathrm{sp})}", from=2-1, to=2-2]
	\arrow["\simeq"', from=2-1, to=3-1]
	\arrow["{i^*_0\psi_t\bar{\CF}}", from=2-2, to=1-2]
	\arrow[from=2-2, to=2-3]
	\arrow[from=2-2, to=3-2]
	\arrow[from=2-3, to=3-2]
	\arrow["{\pi_*(\mathrm{sp})}"', from=3-1, to=3-2]
\end{tikzcd}\]
We have canonical isomorphisms $\pi_*\iota^*(\bar\CF_1)\simeq \CF|_Z$, and $\pi_*\psi_t(\bar\CF_1)\simeq \psi_f\CF$, compatible with $\mathrm{sp}$. Combined with the above diagram, we get a commutative diagram:
\[\begin{tikzcd}
	{i^*_0\iota^*{\bar{\CF}}} & {\pi_*((\psi_t\bar{\CF})_1)} \\
	{\CF|_Z} & {\psi_f\CF}
	\arrow[from=1-1, to=1-2]
	\arrow["\simeq"', from=1-1, to=2-1]
	\arrow["\simeq", from=1-2, to=2-2]
	\arrow["{\mathrm{sp}}"', from=2-1, to=2-2]
\end{tikzcd}\]
Identify $\pi_*((\psi_t\bar{\CF})_1)$ with $\psi_{\widetilde{f}}(\nu_Z(\mathcal{F}))$ via Proposition \ref{lem_rank1monbasic}.1 and use the commutative diagram there, one checks that the above diagram coincides with the left square in Diagram \ref{eqn_i_psi_phi}. This completes the proof.
\end{proof}

Let \begin{equation}\label{eqn_comm_square}
\begin{tikzcd}
	N & Y \\
	M & X
	\arrow[hook, from=1-1, to=1-2]
	\arrow["{f_N}"', from=1-1, to=2-1]
	\arrow["f", from=1-2, to=2-2]
	\arrow[hook, from=2-1, to=2-2]
\end{tikzcd}
\end{equation}
be a commutative diagram of quasi-separated rigid analytic varieties, where the horizontal maps are closed immersions. Recall $(\widetilde{T_MX}\xrightarrow{\bar{p}}X)=(\underline{\Spa}_X(\oplus_{n\in\ZZ}t^nI^{-n})\rightarrow X)$ where $I$ is the coherent $\CO_X$-ideal corresponding to $M$ in $X$. Denote by $J$ (resp. $\widetilde{I}$) the $\CO_Y$-ideal corresponding to $N$ (resp. $M':=M\X_XY$) in $Y$. Note $\widetilde{I}$ is the image of $f^*I$ in $\CO_Y$. We have $\widetilde{T_MX}\times_XY=\underline{\Spa}_X(\oplus t^n f^*I^{-n})\times_X Y$, $\widetilde{T_{M'}Y}=\underline{\Spa}_X(\oplus t^n \widetilde{I}^{-n})$ and $\widetilde{T_NY}=\underline{\Spa}_X(\oplus t^n J^{-n})$. The maps of $\mathcal{O}_Y$-modules $f^*I\twoheadrightarrow \widetilde{I}\hookrightarrow J$ induce the following diagrams:
\[\begin{tikzcd}
	{\widetilde{T_NY}} &&&& {T_NY} & {\widetilde{T_NY}} & {Y\times\mathbf{G}_m} \\
	{\widetilde{T_{M'}Y}} & {\widetilde{T_MX}\times_XY} & Y && {T_MX} & {\widetilde{T_MX}} & {X\times\mathbf{G}_m} \\
	& {\widetilde{T_MX}} & X && 0 & {\mathbf{A}^1} & {\mathbf{G}_m}
	\arrow[from=1-1, to=2-1]
	\arrow["{\bar{f}}"', curve={height=10pt}, dotted, from=1-1, to=3-2]
	\arrow[hook, from=1-5, to=1-6]
	\arrow["{\widetilde{f}}"', from=1-5, to=2-5]
	\arrow["\lrcorner"{anchor=center, pos=0.125}, draw=none, from=1-5, to=2-6]
	\arrow["{\bar{f}}"', from=1-6, to=2-6]
	\arrow[hook', from=1-7, to=1-6]
	\arrow["\lrcorner"{anchor=center, pos=0.125, rotate=-90}, draw=none, from=1-7, to=2-6]
	\arrow["{f\times \id}"', from=1-7, to=2-7]
	\arrow["\alpha", hook, from=2-1, to=2-2]
	\arrow[from=2-2, to=2-3]
	\arrow[from=2-2, to=3-2]
	\arrow["\lrcorner"{anchor=center, pos=0.125}, draw=none, from=2-2, to=3-3]
	\arrow["f", from=2-3, to=3-3]
	\arrow[hook, from=2-5, to=2-6]
	\arrow[from=2-5, to=3-5]
	\arrow["\lrcorner"{anchor=center, pos=0.125}, draw=none, from=2-5, to=3-6]
	\arrow["t"', from=2-6, to=3-6]
	\arrow[hook', from=2-7, to=2-6]
	\arrow[from=2-7, to=3-7]
	\arrow["{\bar{p}}", from=3-2, to=3-3]
	\arrow[hook, from=3-5, to=3-6]
	\arrow[""{name=0, anchor=center, inner sep=0}, hook', from=3-7, to=3-6]
	\arrow["\lrcorner"{anchor=center, pos=0.125, rotate=-90}, draw=none, from=2-7, to=0]
\end{tikzcd}\]
In the left diagram $\bar{f}$ denotes the composition map. Note that $\alpha$ is a closed immersion, and $\bar{f}$ is a map over $\AAA^1$, with the $0$-fibre being $\widetilde{f}$. 


\begin{lemma}\label{lem_nufunctoriality}
    The set-up is as above. (1) If $f$ is proper and Diagram \ref{eqn_comm_square} is Cartesian, let $\CG\in \Dbzc(Y)$, then there is a canonical isomorphism $\nu_M(f_*\CG)\simeq\widetilde{f}_*\nu_N(\CG)$.\\
    (2) If $\bar{f}$ is smooth, let $\CF\in \Dbzc(X)$, then there is a canonical isomorphism $\nu_N(f^*\CF)\simeq\widetilde{f}^*\nu_M(\CF)$. The smoothness of $\bar{f}$ is satisfied in either of the following two cases: (i) $f$ is smooth and Diagram \ref{eqn_comm_square} is Cartesian, (ii) $f$ and $f_N$ are smooth, and $N, M, Y, X$ are smooth.
\end{lemma}

\begin{proof}
(1) As $f$ is proper and $N=M'$, we get $\bar{f}: \widetilde{T_NY}\xhookrightarrow{\alpha}\widetilde{T_MX}\times_XY\rightarrow\widetilde{T_MX}$ is proper. The statement then follows from the commutativity of nearby cycles with proper pushforwards.\\\\
(2) The first statement is clear, by the commutativity of nearby cycles with smooth pullbacks. We now show the smoothness of $\bar{f}$ in the cases (i) and (ii).\\\\
For (i): the flatness of $f$ implies $\widetilde{T_MX}\times_XY\simeq \widetilde{T_NY}$. So $\bar{f}$ is smooth, being a base change of $f$. For (ii): the restriction of $\bar{f}$ to over $\Gm$ is $f\X\id$, hence smooth. Its restriction to the $0$-fibre is $\widetilde{f}$, which is also smooth because of the smoothness of $f$ and $f_N$ (\textit{e.g.}, by the Jacobian criterion of smoothness (\cite[\nopp 1.6.9.(iii)]{huber_etale_1996})). As $\bar{f}$ is a map over $\AAA^1$, $dt$ on $\widetilde{T_MX}$ is pulled-back to $dt$ on $\widetilde{T_NY}$. The Jacobian criterion of smoothness then implies $\bar{f}$ is smooth.
\end{proof}

\section{Microlocalisation}\label{sec_microlocal}

In this section, we study the microlocalisation. It is a version of vanishing cycles for higher codimensions and without the choice of a function. This is made precise by Proposition \ref{lem_microlocalisation_is_vanishingcycle}. The set-up is as in §\ref{sec_intro}.

\begin{definition}[{\cite[\nopp §4.3]{kashiwara_sheaves_1990}}]
    Let $Z\hookrightarrow X$ be a closed immersion of quasi-separated smooth rigid analytic varieties. The \underline{microlocalisation} functor is defined as $\mu_Z$: $\DX\rightarrow D_{mon}(T^*_ZX)$, $\CF\mapsto \F\nu_Z(\CF)$, where $T^*_ZX$ is the conormal bundle of $Z$ in $X$.
\end{definition}

\begin{lemma}
    Let $Z\hookrightarrow X$ be a closed immersion of constant codimension $\delta$ of quasi-separated smooth rigid analytic varieties.\\
    (1) Assume $\Lambda=\FF_{\ell^r}$. Then $\mu_Z[\delta]$ is perverse t-exact.\\
    (2) For $\CF\in\Dbzc(X)$, there is an isomorphism $\DD\mu_Z\CF\cong(\mu_Z\DD\CF)\langle \delta\rangle$.
\end{lemma}

\begin{proof}
    These follow from the perverse t-exactness and commutativity with duality of $\nu_Z$ (Lemma \ref{lem_nu&pervdual}) and $\F[\delta]$ (Corollary \ref{cor_F_perverse_D}). 
\end{proof}

\begin{lemma}\label{}
    Let $i: Z\hookrightarrow X$ be a closed immersion of quasi-separated smooth rigid analytic varieties, $\CF\in \DX$, and $i'_0: Z\hookrightarrow T^*_ZX$ be the 0-section. Then there are canonical isomorphisms $i_0'^*\mu_Z(\CF)\simeq i^!\CF$ and, when $Z$ is of constant codimension $\delta$, $i_0'^!\mu_Z(\CF)\simeq i^*\CF\langle -\delta\rangle$.
\end{lemma}

\begin{proof}
    Denote by $\pi: T_ZX\rightarrow Z$ the projection, and by $i_0: Z\hookrightarrow T_ZX$ the $0$-section. Use Lemma \ref{lem_fourierfunctoriality}, Lemma \ref{lem_sommetdenu}, and Lemma \ref{lem_hyp_localisation}, we compute: $i_0'^*\mu_Z(\CF)\simeq\pi_!\nu_Z(\CF)\isoot i_0^!\nu_Z(\CF)\isoto i^!\CF$, and $i_0'^!\mu_Z(\CF)\simeq \pi_*\nu_Z(\CF)\langle-\delta\rangle\isoto i_0^*\nu_Z(\CF)\langle-\delta\rangle\isoot i^*\CF\langle-\delta\rangle$.
\end{proof}

\begin{proposition}\label{lem_microlocalisation_is_vanishingcycle}
    Let $Z=V(f)\hookrightarrow X$ be a closed immersion of a smooth hypersurface in a quasi-separated smooth rigid analytic variety. $df$ determines a section $s'_f: Z\hookrightarrow T^*_ZX$, $z\mapsto (z,df_z)$. Then, for $\CF\in\DX$, there is an isomorphism $s'^*_f(\mu_Z(\CF))\cong\Phi_f(\CF)(-1)$, under which the equivariant monodromy of $s'^*_f(\mu_Z(\CF))$ coincides with the usual monodromy of $\Phi_f(\CF)(-1)$.
\end{proposition}

\begin{proof}
    Use Lemma \ref{lem_rank1monbasic_F} and Proposition \ref{lem_nu&psi}.2 (and the same notations), we compute: $s'^*_f(\mu_Z(\CF))\cong \Phi_{\widetilde{f}}(\nu_Z(\CF))(-1)\cong\Phi_f(\CF)(-1)$.
\end{proof}

\begin{lemma}\label{lem_mu_functoriality}
   Let 
\[\begin{tikzcd}
	N & Y \\
	M & X
	\arrow[hook, from=1-1, to=1-2]
	\arrow["{f_N}"', from=1-1, to=2-1]
	\arrow["f", from=1-2, to=2-2]
	\arrow[hook, from=2-1, to=2-2]
\end{tikzcd}\]
be a commutative diagram of quasi-separated smooth rigid analytic varieties and 
\[\begin{tikzcd}
	& {N\times_XT_MX} &&&& {N\times_MT^*_MX} & \\
	{T_NY} && {T_MX} && {T^*_NY} && {T^*_MX}
	\arrow["{v'}", from=1-2, to=2-3]
	\arrow["u"', from=1-6, to=2-5]
	\arrow["v", from=1-6, to=2-7]
	\arrow["{u'}", from=2-1, to=1-2]
\end{tikzcd}\]
the induced correspondences of normal and conormal bundles.\\
(1) If $f$ is proper and the diagram is Cartesian, let $\CG\in\Dbzc(Y)$, then there is a canonical isomorphism $\mu_M(f_*\CG)\simeq f_\circ\mu_N(\CG)$. Assume furthermore $N$ and $M$ have constant codimensions, let $\delta=\mathrm{codim}_YN-\mathrm{codim}_XM$, then there is a canonical isomorphism $\mu_M(f_*\CG)\simeq f_{\dot\circ}\mu_N(\CG)\langle\delta\rangle$. Here $f_{\circ}:=v_*u^*$ and $f_{\dot\circ}:=v_*u^!$.\\ 
(2) If $f$ and $f_N$ are smooth, let $\CF\in \DX$, then there is a canonical isomorphism $\mu_N(f^!\CF)\simeq f^{\dot\circ}\mu_M(\CF)$, where $f^{\dot\circ}:=u_*v^!$.
\end{lemma}

In the notations $f_\circ$, $f_{\dot\circ}$ and $f^{\dot\circ}$, we have omitted their dependence on $N$ and $M$. There should be no risk of confusion. Note that if $f$ is proper, then $v$ and $v'$ are proper; if $f$ and $f_N$ are smooth, then $v$ and $v'$ are smooth, $u$ is a closed immersion, and $u'$ is a smooth surjection.

\begin{proof}
    Recall $\widetilde{f}$ denotes the composition $v'u'$.\\
    (1). Working locally, we may assume $f$ has constant relative dimension, and $N$ and $M$ have constant codimensions. Compute: $\mu_M(f_*\CG)=\F\nu_M(f_*\CG)\simeq\F\widetilde{f}_*\nu_N(\CG)\simeq \F\widetilde{f}_!\nu_N(\CG)\simeq\F v'_*u'_!\nu_N(\CG)\simeq v_*\F u'_!\nu_N(\CG)\simeq v_*u^*\mu_N(\CG)$, where in the second step we used Lemma \ref{lem_nufunctoriality}.1, in the fifth step we used Lemma \ref{lem_F_base_change}, in the sixth step we used Lemma \ref{lem_fourierfunctoriality}.2. Similarly, $\mu_M(f_*\CG)=\F\nu_M(f_*\CG)\simeq\F\widetilde{f}_*\nu_N(\CG)\simeq\F v'_*u'_*\nu_N(\CG)\simeq v_*\F u'_*\nu_N(\CG)\simeq v_*u^!\mu_N(\CG)\langle\delta\rangle$.\\\\
    (2). Working locally, we may assume $f$ has constant relative dimension, and $N$ and $M$ have constant codimensions. Compute: $\mu_N(f^!\CF)\simeq \F\nu_N(f^*\CF)\langle \dim f\rangle\simeq \F\widetilde{f}^*\nu_M(\CF)\langle \dim f\rangle\simeq \F\widetilde{f}^!\nu_M(\CF)\langle \dim f-\dim f\rangle\simeq \F u'^!v'^!\nu_M(\CF)\simeq u_*\F v'^!\nu_M(\CF)\simeq u_*v^!\mu_M(\CF)$. Most steps are similar to the above, in the first (resp. third) (resp. last) step we used the smoothness of $f$ (resp. $\widetilde{f}$) (resp. $f$) and Poincaré duality.
\end{proof}

\section{Micro-hom}\label{sec_muhom}

In this section, we study the micro-hom. It can be viewed as a generalisation of the microlocalisation, a version of vanishing cycles over general bases in our context (Lemma \ref{lem_muhom_is_microlocalisation}), and the sheaf-theoretic analogue of $\RHom$ of microdifferential-modules. The set-up is as in §\ref{sec_intro}.\\

Let $f: Y\rightarrow X$ be a map of separated smooth rigid analytic varieties. Denote its graph by $\Gamma$, which is a closed analytic subvariety of $Y\X X$. We make the identification $T_\Gamma^*(Y\X X)\simeq Y\X_X  T^*X$. Denote by $q_Y$ (resp. $q_X$) the projection from $Y\X X$ to $Y$ (resp. $X$).

\begin{definition}[{\cite[\nopp §4.4]{kashiwara_sheaves_1990}}]
    Let $f: Y\rightarrow X$ be a map of separated smooth rigid analytic varieties. Define the following \underline{micro-hom} functors:\\
    (1) $\mu hom((-))\rightarrow(-))$: $\Dbzc(Y)^{op}\times\DX\rightarrow D_{mon}(Y\X_XT^*X)$, $(\CG,\CF)\mapsto \mu hom(\CG\rightarrow\CF):=\mu_\Gamma\RHom_{Y\X X}(q_Y^*\CG,q_X^!\CF)$.\\
    (2) $\mu hom((-))\leftarrow(-))$: $\Dbzc(Y)\times\DX^{op}\rightarrow D_{mon}(Y\X_XT^*X)$, $(\CG,\CF)\mapsto \mu hom(\CG\leftarrow\CF):=a^*\mu_\Gamma\RHom_{Y\X X}(q_X^*\CF,q_Y^!\CG)$, where $a$ is the antipodal map on $Y\X_XT^*X$.\\\\
    For $Y=X$, $f=\id$, define $\mu hom(-,-)$: $\Dbzc(X)^{op}\times\DX\rightarrow D_{mon}(T^*X)$, $(\CF_1,\CF_2)\mapsto \mu hom(\CF_1,\CF_2)$ $:=\mu_\Delta\RHom_{X\X X}(q_1^*\CF_1,q_2^!\CF_2)\simeq a^*\mu_\Delta\RHom_{X\X X}(q_2^*\CF_1,q_1^!\CF_2)$, where $q_1$ (resp. $q_2$) is the first (resp. second) projection from $X\X X$, and $\Delta$ is the diagonal.
\end{definition}

In the analytic context there is no known analogue of vanishing cycles over general bases as in the algebraic context. The micro-hom serves as an alternative: $\mu hom(\CG\rightarrow\CF)$ records information of $\CF$ as “probed” by $\CG$. This can be made precise in the following special case.

\begin{lemma}\label{lem_muhom_is_microlocalisation}
    Let $i: Z\hookrightarrow X$ be a closed immersion of separated smooth rigid analytic varieties and $\CF\in\DX$. Then there is a canonical isomorphism $\mu hom(i_*\underline{\Lambda}_Z,\CF)\simeq\widetilde{i}_*\mu_Z(\CF)$, where $\widetilde{i}$ denotes the closed immersion $T^*_ZX\hookrightarrow T^*X$.
\end{lemma}

\begin{proof}
    We will use the following diagrams and compute: 
\[\begin{tikzcd}
	\Gamma & {Z\times X} && \Gamma & {Z\X X} \\
	{\Delta_X} & {X\times X} && Z & X
	\arrow[hook, from=1-1, to=1-2]
	\arrow[hook, from=1-1, to=2-1]
	\arrow["\lrcorner"{anchor=center, pos=0.125}, draw=none, from=1-1, to=2-2]
	\arrow["{i\times \id}", hook, from=1-2, to=2-2]
	\arrow[hook, from=1-4, to=1-5]
	\arrow["\simeq"', from=1-4, to=2-4]
	\arrow["{q_X}", from=1-5, to=2-5]
	\arrow[hook, from=2-1, to=2-2]
	\arrow["i"', hook, from=2-4, to=2-5]
\end{tikzcd}\]
\begin{align*}
\mu hom(i_*\underline{\Lambda}_Z,\CF) &\simeq\mu_{\Delta_X}\RHom((i\X\id)_*\underline{\Lambda}_{Z\X X},q_2^!\CF)\\
&\simeq\mu_{\Delta_X}(i\X\id)_*q_X^!\CF  \,\,(\text{used general fact: $f_*\RHom(A,f^!B)\simeq\RHom(f_!A,B)$})\\
&\simeq (i\X\id)_\circ\mu_\Gamma q_X^!\CF \,\,(\text{Lemma \ref{lem_mu_functoriality} applied to the left diagram})\\
&\simeq (i\X\id)_\circ q_X^{\dot\circ}\mu_Z\CF \,\,(\text{Lemma \ref{lem_mu_functoriality} applied to the right diagram})\\
&\simeq \widetilde{i}_*\mu_Z(\CF)
\end{align*}
In the last step, we used that, for the left diagram, $u$ (resp. $v$) is an isomorphism (resp. closed immersion) (notations as in Lemma \ref{lem_mu_functoriality}), and for the right diagram, $v$ (resp. $u$) is an isomorphism (resp. closed immersion).
\end{proof}

Combined with Proposition \ref{lem_microlocalisation_is_vanishingcycle}, we see that, for $Z=V(f)\hookrightarrow X$ a smooth hypersurface, $\phi_f(\CF)$ can be computed from $\mu hom(i_*\underline{\Lambda}_Z,\CF)$. Thus, $\mu hom$ can be viewed as a “vanishing cycle with general sources”. Note it only records information of “\emph{linear} specialisations”. In the context of real and complex manifolds, linear specialisations are sufficient; while in the positive characteristic algebraic context, they are not, due to wild ramifications. It is our expectation that, in our context, which does not have local wild ramifications, linear specialisations are also sufficient for microlocal purposes.

\begin{lemma}\label{lem_muhom_functoriality}
    Let $f: Y\rightarrow X$ be a map of separated smooth rigid analytic varieties, $\CF\in\DX$ and $\CG\in \Dbzc(Y)$. We have the following diagrams: 
\[\begin{tikzcd}
	{T^*Y\simeq T^*_{\Delta_Y}(Y\times Y)} & {\Delta_Y} & {Y\times Y} \\
	{Y\times_XT^*X\simeq T^*_{\Gamma}(Y\times X)} & \Gamma & {Y\times X} \\
	{T^*X\simeq T^*_{\Delta_X}(X\times X)} & {\Delta_X} & {X\times X}
	\arrow[hook, from=1-2, to=1-3]
	\arrow["\simeq", from=1-2, to=2-2]
	\arrow["{\circled{A}}"{description}, draw=none, from=1-2, to=2-3]
	\arrow["{\id\times f}", from=1-3, to=2-3]
	\arrow["{f\X f}"{description}, curve={height=-40pt}, dotted, from=1-3, to=3-3]
	\arrow["u"', from=2-1, to=1-1]
	\arrow["v", from=2-1, to=3-1]
	\arrow[hook, from=2-2, to=2-3]
	\arrow[from=2-2, to=3-2]
	\arrow["\lrcorner"{anchor=center, pos=0.125}, draw=none, from=2-2, to=3-3]
	\arrow["{\circled{B}}"{description}, draw=none, from=2-2, to=3-3]
	\arrow["{f\times \id}", from=2-3, to=3-3]
	\arrow[hook, from=3-2, to=3-3]
\end{tikzcd}\]
    (1) If $f$ is smooth, then we have canonical isomorphisms:\\
    (1.1) $u_*\mu hom(\CG\rightarrow\CF)\simeq \mu hom(\CG,f^!\CF)$,\\
    (1.2) $u_*\mu hom(\CG\leftarrow\CF)\simeq \mu hom(f^*\CF,\CG)$,\\
    (1.3) $f^{\dot\circ}\mu hom(\CF_1,\CF_2)\simeq \mu hom(f^*\CF_1,f^!\CF_2)$.\\\\
    (2) If $f$ is proper, then we have canonical isomorphisms:\\
    (2.1) $v_*\mu hom(\CG\rightarrow\CF)\simeq \mu hom(f_*\CG,\CF)$,\\
    (2.2) $v_*\mu hom(\CG\leftarrow\CF)\simeq \mu hom(\CF,f_*\CG)$,\\
    (2.3) if $f$ is a closed immersion, then $f_\circ\mu hom(\CG_1,\CG_2)\simeq \mu hom(f_*\CG_1, f_*\CG_2)$.\\\\
    (3) If $f$ is smooth and proper, then we have canonical isomorphisms:\\
    (3.1) $f_\circ\mu hom(\CG,f^!\CF)\simeq \mu hom(f_*\CG, \CF)$,\\
    (3.2) $f_\circ\mu hom(f^*\CF,\CG)\simeq \mu hom(\CF, f_*\CG)$.
\end{lemma}

\begin{proof}
    The notations are as in Lemma \ref{lem_mu_functoriality}. In the following when we say the $u$ or $v$ for \circled{A} or \circled{B}, we mean the maps in Lemma \ref{lem_mu_functoriality}, to be distinguished from the $u$ and $v$ in the diagram above. We show (1) and (3.1). The others are similar.
    \begin{align*}
(1.1).\,&u_*\mu hom(\CG\rightarrow\CF)\\ 
&\simeq (\id\X f)^{\dot\circ}\mu_\Gamma\RHom(q_Y^*\CG, q_X^!\CF)\,\,(\text{the $v$ for \circled{A} is an isomorphism})\\
&\simeq \mu_{\Delta_Y}(\id\X f)^!\RHom(q_Y^*\CG, q_X^!\CF)  \,\,(\text{Lemma \ref{lem_mu_functoriality} applied to \circled{A}})\\
&\simeq \mu_{\Delta_Y}\RHom((\id\X f)^*q_Y^*\CG, (\id\X f)^!q_X^!\CF) \,\,(\text{general fact $f^!\RHom(A,B)\simeq\RHom(f^*A,f^!B)$})\\
&\simeq \mu_{\Delta_Y}\RHom(q_1^*\CG, q_2^!f^!\CF)=\mu hom(\CG,f^!\CF).
   \end{align*}
   \begin{align*}
(1.2).\,&u_*\mu hom(\CG\leftarrow\CF)\qquad\qquad\qquad\qquad\qquad\qquad\qquad\qquad\qquad\qquad\qquad\qquad\qquad\qquad\qquad\qquad\,\,\,\,\,\,\,\,\,\,\,\\
&\simeq a^*\mu_{\Delta_Y}\RHom((\id\X f)^*q_X^*\CF, (\id\X f)^!q_Y^!\CG) \,\,(\text{as above})\\
&\simeq a^*\mu_{\Delta_Y}\RHom(q_2^*f^*\CF, q_1^!\CG)\simeq\mu hom(f^*\CF,\CG).
    \end{align*}
    \begin{align*}
(1.3).\,&f^{\dot\circ}\mu hom(\CF_1,\CF_2)\\ 
&\simeq (\id\X f)^{\dot\circ}(f\X\id)^{\dot\circ}\mu_{\Delta_X}\RHom(q_1^*\CF_1, q_2^!\CF_2)\\
&\simeq \mu_{\Delta_Y}(\id\X f)^!(f\X\id)^!\RHom(q_1^*\CF_1, q_2^!\CF_2)\,\,(\text{Lemma \ref{lem_mu_functoriality} applied to \circled{B} and \circled{A}})\\
&\simeq \mu_{\Delta_Y}\RHom((f\X f)^*q_1^*\CF_1, (f\X f)^!q_2^!\CF_2)\,\,(\text{general fact $f^!\RHom(A,B)\simeq\RHom(f^*A,f^!B)$})\\
&\simeq \mu_{\Delta_Y}\RHom(q_1^*f^*\CF_1, q_2^!f^!\CF_2)=\mu hom(f^*\CF_1,f^!\CF_2).
   \end{align*}
   \begin{align*}
(3.1).\,&f_\circ\mu hom(\CG,f^!\CF)\qquad\qquad\qquad\qquad\qquad\qquad\qquad\qquad\qquad\qquad\qquad\qquad\qquad\qquad\qquad\qquad\,\,\,\,\,\,\,\,\,\,\,\\ 
&= v_*u^*\mu hom(\CG,f^!\CF)\\
&\simeq v_*u^*u_*\mu hom(\CG\rightarrow\CF)\,\,(\text{(1.1)})\\
&\simeq v_*\mu hom(\CG\rightarrow\CF)\,\,(\text{$u$ is a closed immersion})\\
&\simeq \mu hom(f_*\CG, \CF)\,\,(\text{(2.1)}).
   \end{align*}
\end{proof}

\section{The singular support}\label{sec_SS}
In this section, we define and study the singular support. This is a closed conical subset of the cotangent bundle recording the directions in which the input sheaf is not locally constant. The classical reference is \cite[Chapter V]{kashiwara_sheaves_1990}, see \cite{beilinson_constructible_2016} for the algebraic setting. The set-up is as in §\ref{sec_intro}.

\begin{definition}\label{def_SS}
    Let $X$ be a smooth rigid analytic variety. For $\CF\in\DX$, the \underline{singular support} $SS(\CF)$ is defined as $\{(x,\xi)\subseteq (T^*X)(K)\,|\,\text{there exists a test function $f$ at $(x,\xi)$ such that}\,\, \phi_f(\CF)_x\neq 0\}^{cl}$. Here, “a test function $f$ at $(x,\xi)$” means a function $f$ on some open neighbourhood of $x$ in $X$ such that $f(x)=0$ and $df(x)=\xi$, and “$cl$” means taking closure in $T^*X$. We will abbreviate the set in “$\{...\}$” by $\{\dddot\CF\}$ (so $SS(\CF)=\{\dddot\CF\}^{cl}$).
\end{definition}

\begin{remark}\label{rmk_SS}
    (1) We emphasise that $\dotF$ is a subset of classical points in $T^*X$, and the closure is taken in the analytic topology rather than the Zariski topology. It is thus \textit{a priori} not clear if $SS(\CF)$ is an analytic subset. See Question \ref{que_ana}.\\
    (2) By Lemma \ref{lem_psietalebasechange}, $SS(\CF)$ is a conical subset of $T^*X$.\\
    (3) The same definition makes sense for all $\CF\in D^{}(X)$, but we have restricted to $\CF\in \DX$ because $\dotF$ does not “see” points other than classical points, so the above definition is not the “correct” one for general $\CF$. For example, let $\CF=j_!\underline{\Lambda}_D$, where $j: D\hookrightarrow \AAA^1$ is the open immersion. Then $SS(\CF)=D^{cl}$ with the above definition. It is desirable to extend the notion of singular support to all sheaves without constructibility restrictions, as has been done in complex analytic as well as algebraic settings (cf. the footnote to Remark \ref{rmk_to_thm_1.2}.2).
\end{remark}

Here is an equivalent definition of the singular support (\textit{cf.} \cite[\nopp 1.5]{beilinson_constructible_2016}).

\begin{definition}\label{def_Ctransversal}
    Let $X$ be a smooth rigid analytic variety, $C\subseteq T^*X$ be a conical closed subset, and $f: X\rightarrow \AAA^1$ a map. Denote by $\Gamma_{df}$ the graph of the differential of $f$. We say \underline{$f$ is $C$-transversal} if $((\Gamma_{df}-T^*_XX)\cap C)(K)=\varnothing$, \textit{i.e.}, for every classical point $(x,\xi)\in C$, we have $df(x)\neq \xi$.
\end{definition}

\begin{definition}\label{def_phiLA}
    Let $f: X\rightarrow\AAA^1$ be a map of rigid analytic varieties, and $\CF\in\DX$. We say \underline{$(f,\CF)$ is $\phi$-locally acyclic} ($\phi$-LA) if for every $a\in K$, we have $\phi_{f-a}(\CF)=0$.  
\end{definition}

\begin{remark}
    We choose the terminology “$\phi$-LA” instead of “LA” to distinguish from the notion of local acyclicity introduced in \cite[\nopp IV.2]{fargues_geometrization_2021}. It would be interesting to compare these two notions.
\end{remark}

\begin{definition-lemma}\label{def_microsupport}
    Let $X$ be a smooth rigid analytic variety, and $\CF\in\DX$. Let $\CalC(\CF)$ be the set of conical closed subsets $C$ of $T^*X$ satisfying the following condition: for every open subset $U\subseteq X$ and every function $f: U\rightarrow \AAA^1$ which is $C|_U$-transversal, we have $(f,\CF|_U)$ is $\phi$-LA. $\CalC(\CF)$ is closed under intersections. Denote the intersection of all $C\in\CCC(\CF)$ by $C_{min}(\CF)$.
\end{definition-lemma}

\begin{proof}
    We need to show $\CalC(\CF)$ is closed under intersections. Let $\{C_i\}_{i\in I}\subseteq \CCC(\CF)$ and $C=\cap_{i\in I} C_i$. It suffices to show that for $f: U\rightarrow\AAA^1$ a function which is $C$-transversal, we have $\phi_f(\CF)= 0$. Suppose not, then, by the Zariski-constructibility of $\phi_f(\CF)$, there exists $x\in U(K)$ such that $\phi_f(\CF)_x\neq 0$. As $f$ is $C$-transversal, $(x,df(x))\notin C$, so $(x,df(x))\notin C_i$ for some $i$. Then, there exists an open neighbourhood $V$ of $x$ in $U$ such that $\Gamma_{df}\cap C_i=\varnothing$ over $V$. So $f$ is $C_i$-transversal over $V$, hence $\phi_f(\CF)_x=0$, a contradiction.
\end{proof}

\begin{lemma}\label{lem_alt_def_of_SS}
    Let $X$ be a smooth rigid analytic variety, and $\CF\in\DX$. Then $SS(\CF)=C_{min}(\CF)$.
\end{lemma}

\begin{proof}
    “$\subseteq$”: it suffices to show $\{\dddot{\CF}\}\subseteq C_{min}(\CF)$. Suppose there exists $(x,\xi)\in\{\dddot{\CF}\}-C_{min}(\CF)$. Let $f$ be a test function at $(x,\xi)$ such that $\phi_f(\CF)_x\neq 0$. But $f$ is $C_{min}(\CF)$-transversal over some open neighbourhood of $x$ in $X$, so $\phi_f(\CF)_x=0$, a contradiction.\\\\
    “$\supseteq$”: it suffices to show that for all $SS(\CF)$-transversal $f$, we have $\phi_f(\CF)=0$. Suppose not, then, by the Zariski-constructibility of $\phi_f(\CF)$, there exists $x\in X(K)$ such that $\phi_f(\CF)_x\neq 0$. So $(x,df(x))\in\{\dddot{\CF}\}$, contradicting the $SS(\CF)$-transversality of $f$.
\end{proof}

We now prove some basic properties of the singular support. First recall the following operations on closed conical subsets in the cotangent bundle (\textit{cf.} \cite[\nopp 1.2]{beilinson_constructible_2016}). Let $f: X\rightarrow Y$ be a map of smooth rigid analytic varieties. It induces a correspondence:
\[\begin{tikzcd}
	& {X\times_YT^*Y} \\
	{T^*X} && {T^*Y}
	\arrow["u", from=1-2, to=2-1]
	\arrow["v"', from=1-2, to=2-3]
\end{tikzcd}\]

If $f$ is proper, then $v$ is proper. For a closed conical subset $C$ in $T^*X$, define $f_\circ C=vu^{-1}(C)$, which is a closed conical subset of $T^*Y$. If $f$ is smooth, then $v$ is smooth and $u$ is a closed immersion. For a closed conical subset $C$ in $T^*Y$, define $f^\circ C=uv^{-1}(C)$, which is a closed conical subset of $T^*X$.\\

Both operations have their corresponding versions when restricted to classical points, we will denote them by the same symbols. For example, for $f$ proper and $C$ in $T^*X(K)$ a closed conical subset, $f_\circ C:=vu^{-1}(C)$, where $(u,v)$ is the correspondence $T^*X(K)\xleftarrow{u} (X\X_YT^*Y)(K)\xrightarrow{v} T^*Y(K)$. 

\begin{lemma}\label{lem_SSpullpush}
    (1) Let $p: Y\rightarrow X$ be a smooth map of smooth rigid analytic varieties, and $\CF\in \DX$. Then $p^\circ SS(\CF)\subseteq SS(p^*\CF)$. If $p$ is \etale, then equality holds.\\
    (2) Let $g: X\rightarrow Y$ be a proper map of smooth rigid analytic varieties, and $\CF\in \DX$. Then $SS(g_*\CF)\subseteq g_\circ SS(\CF)$. If $g$ is a closed immersion, then equality holds.
\end{lemma}

One expects equality holds in (1). See Question \ref{que_smoothpull}.

\begin{proof}
    (1) Let $(y,\zeta)\in p^\circ\{\dddot{\CF}\}$ and $(x,\xi)\in\{\dddot{\CF}\}$ be the unique element such that $x=p(y)$ and $\zeta=dp_y(\xi)$. Then, $\phi_f(\CF)_x\neq 0$ for some test function $g$ at $(x,\xi)$. As vanishing cycles commute with smooth pullbacks (Lemma \ref{lem_smoothqcqs}.1), we get $\phi_{fp}(p^*\CF)_y\simeq\phi_p(\CF)_x\neq 0$, so $(y,\zeta)\in \{\dddot{p^*\CF}\}$. So $p^\circ SS(\CF)=(p^\circ\{\dddot{\CF}\})^{cl}\subseteq \{\dddot{p^*\CF}\}^{cl}=SS(p^*\CF)$. The second statement follows from the fact that \etale maps are locally isomorphisms at classical points.\\\\
    (2) Let $(y,\zeta)\in\{\dddot{g_*\CF}\}$, and $f$ be a test function at $(y,\zeta)$ such that $\phi_f(g_*\CF)_y\neq 0$. Apply proper base change and the fact that vanishing cycles commute with proper pushforwards (Lemma \ref{lem_smoothqcqs}.2), we get $\phi_f(g_*\CF)_y\simeq R\Gamma(X_y,\phi_{fg}(\CF))$. As $\phi_{fg}(\CF)$ is Zariski-constructible, there exists $x\in X_y(K)$ such that $\phi_{fg}(\CF)_x\neq 0$. So $(x,dg_x(\zeta))\in\{\dddot{\CF}\}$. This proves $\{\dddot{g_*\CF}\}\subseteq g_\circ\{\dddot{\CF}\}$. Consequently $SS(g_*\CF)\subseteq (g_\circ\{\dddot{\CF}\})^{cl}\subseteq g_\circ SS(\CF)$.\\\\
    For the last statement, let $g=i$ be a closed immersion. First note that $(i_\circ\{\dddot{\CF}\})^{cl}= i_\circ SS(\CF)$. So, to show $SS(g_*\CF)\supseteq g_\circ SS(\CF)$, it suffices to show $i_\circ\{\dddot{\CF}\}\subseteq\{\dddot{i_*\CF}\}$. Let $(x,\xi)\in\{\dddot{\CF}\}$, and $(x,\zeta)\in T^*Y(K)$ with $\zeta|_X=\xi$. Let $f$ be a test function on $X$ at $(x,\xi)$ such that $\phi_f(\CF)_x\neq 0$. It extends to a test function $\tilde{f}$ on $Y$ at $(x,\zeta)$. Since $\phi_{\tilde{f}}(i_*\CF)_x\simeq\phi_f(\CF)_x$, we get $(x,\zeta)\in\{\dddot{i_*\CF}\}$. 
\end{proof}

\begin{lemma}\label{lem_basesupp}
    Let $X$ be a smooth rigid analytic variety, and $\CF\in\DX$. Then the base of $SS(\CF)$ equals the support of $\CF$.
\end{lemma}

By “the base of $SS(\CF)$” we mean $\pi(SS(\CF))$, where $\pi: T^*X\rightarrow X$ is the projection. Note, for a Zariski-constructible $\CF$, its support satisfies $\supp(\CF)=\{x\in X(K)\,|\,\CF_x\neq0\}^{cl}$, which is a consequence of the fact that classical points are dense on a rigid analytic variety (\cite[\nopp 4.2]{huber_continuous_1993}).

\begin{proof}
    For $x_0\in\{x\in X(K)\,|\,\CF_x\neq0\}$, consider the test function $0$ at $(x_0,0)$, one sees $(x_0,0)\in\{\dddot{\CF}\}$. So $\supp(\CF)\subseteq \pi (SS(\CF))$. On the other hand, clearly $\pi\{\dddot{\CF}\}\subseteq\supp(\CF)$, so $\pi (SS(\CF))=\pi(\{\dddot{\CF}\}^{cl})\subseteq (\pi\{\dddot{\CF}\})^{cl}\subseteq\supp(\CF)$.
\end{proof}

\begin{lemma}\label{lem_comparisonSS}
    Let $\CX$ be a smooth finite type scheme over $K$, $\CF\in D^b_c(X)$. Let $X$ and $\CF^{an}$ be their analytifications. Then $(SS(\CF))^{an}\subseteq SS(\CF^{an})$.
\end{lemma}

One expects equality holds. See Question \ref{que_comparisonSS}.

\begin{proof}
    Let $\{\dddot\CF\}^{alg}=\{(x,\xi)\subseteq T^*\CX(K)\,|$ there exists a test function $f$ at $(x,\xi)$ in an open neighbourhood of $x$ in $\CX$ such that $\phi_f(\CF)_x\neq 0\}$. In the following we view $\{\dddot\CF\}^{alg}$ as a subset of $T^*X(K)$($=T^*\CX(K)$). Consider the following three closed subsets of $T^*X$: $(\{\dddot\CF\}^{alg})^{zcl}$, where “$zcl$” denotes closure in the Zariski topology on $T^*X$; $(\{\dddot\CF\}^{alg})^{cl}$, where “$cl$” denotes closure in the analytic topology; and $\{\dddot{\CF^{an}}\}^{cl}(=SS(\CF^{an}))$. We claim $(SS(\CF))^{an}=(\{\dddot\CF\}^{alg})^{zcl}=(\{\dddot\CF\}^{alg})^{cl}\subseteq\{\dddot{\CF^{an}}\}^{cl}$. Indeed, the first two equalities follow from the following facts: (i) vanishing cycles commute with the analytification, (ii) $(SS(\CF))^{an}$ equals the closure $\{\dddot\CF\}^{alg}$ in $T^*\CX$, and (iii) classical points are dense in a rigid analytic variety. The last inclusion is clear.
\end{proof}

\begin{proposition}\label{prop_perverseaddSS}
    Let $X$ be a smooth rigid analytic variety, and $\CF\in\DX$. Then:\\
    (1) $SS(\CF)=SS(\DD\CF)$;\\
    (2) assume $\Lambda=\FF_{\ell^r}$ and $\CF$ has finitely many irreducible perverse constituents $\{\CF_\alpha\}$ \footnote{This means each $\CF_\alpha$ is an irreducible subquotient of $^p\CH^i(\CF)$ for some $i\in\ZZ$.} (this holds, for example, if $X$ is quasi-compact). Then $SS(\CF)=\cup SS(\CF_\alpha)$. 
\end{proposition}

\begin{proof}
    (1) follows from the fact that vanishing cycles commute with duality. For (2), we use the alternative definition of the singular support (Definition-Lemma \ref{def_microsupport}, Lemma \ref{lem_alt_def_of_SS}). If $\CF$ is perverse, by induction on the length, we reduce to showing $SS(\CF)=SS(A)\cup SS(B)$ for an exact sequence $0\rightarrow A\rightarrow\CF\rightarrow B\rightarrow0$ of perverse sheaves. It is clear that $SS(\CF)\subseteq SS(A)\cup SS(B)$. We claim $SS(A)\subseteq SS(\CF)$ (then $SS(B)\subseteq SS(\CF)$ follows because $SS(B)\subseteq SS(\CF)\cup SS(A)$). Indeed, it suffices to show $A$ is micro-supported on $SS(\CF)$, which is a direct consequence of the shifted perverse t-exactness of vanishing cycles. For general $\CF$, we do induction on the perverse amplitude. The argument is essentially the same.
\end{proof}

\begin{proposition}\label{prop_0seclisse}
    Assume $\Lambda=\FF_{\ell^r}$. Let $X$ be a connected smooth rigid analytic variety, and $\CF\in\DX$. Then $SS(\CF)$ equals the $0$-section if and only if $\CF$ is a non-zero local system.
\end{proposition}

\begin{proof}
    The “if” part is a consequence of Example \ref{ex_phicomp}.1 and Lemma \ref{lem_basesupp}. For the “only if” part, first note we may assume $X$ is quasi-compact. Then, by Proposition \ref{prop_perverseaddSS}.2, it suffices to show the following: if $\CF$ is an irreducible (hence non-zero) perverse sheaf with $SS(\CF)=T^*_XX$, then $\CF$ is a local system. $\CF$ is of the form $j_{!*}\CL$, for $j: U\hookrightarrow X$ a Zariski-locally closed immersion, and $\CL$ an irreducible perverse local system on $U$. By Lemma \ref{lem_SSpullpush}.2, $U$ must be a Zariski-open subset of $X$. By Lemma \ref{lem_max_liss_open_is_zariski_open}, we may assume that $U$ is the maximal open on which $\CF$ is locally constant. If $Z:=X-U$ is everywhere of codimension $\geq2$, then $j_{!*}\CL$ is a local system by purity (\cite[\nopp 2.15]{hansen_vanishing_2020}). Now suppose $Z$ is of codimension $1$ in some open subset of $X$. The regular locus $Z_{reg}$ is dense in $Z$. Let $x\in Z_{reg}(K)$, and $f$ be a function in an open neighbourhood $U$ of $x$ in $X$ such that $Z\cap U=V(f)$. We will show $\phi_f(\CF)_x\neq0$, which contradicts $SS(\CF)=T^*_XX$ and completes the proof.\\\\  By completing $f$ to a coordinate system at $x$ and use the fact that \etale maps are local isomorphisms at classical points, we may identify $Z=V(f)\xhookrightarrow{i} U\xhookleftarrow{j}(U-Z)$ with $D^{n-1}=V(T_1)\xhookrightarrow{i} D^n=\Spa(K\langle T_1,...,T_n\rangle)\xhookleftarrow{j}D^{n-1}\X D^\X$, $x$ with the origin, and $f$ with the projection to the $T_1$-coordinate. Let $Y\rightarrow D^{n-1}\X D^\X$ be some finite \etale covering trivialising $\CL$. Abhyankar’s Lemma in the form of \cite[\nopp 3.2]{lutkebohmert_riemanns_1993} implies that this covering becomes a disjoint union of Kummer coverings over $D^{n-1}(\epsilon)\X D^\X(\eta)$, for small enough $\epsilon,\,\eta\in|K^\X|$.\footnote{Two remarks on details: (i) in \cite[\nopp 3.2]{lutkebohmert_riemanns_1993}, one takes $r\rightarrow 0$ to get the analogous statement where $A(r,R)$ is replaced by a punctured disc (which is what we are actually using). This is legitimate as $d$ does not depend on $r$ and $R$. (ii) Since $D^{n-1}(\epsilon)$ is cofinal in \etale open neighbourhoods of $0$ in $D^{n-1}$, the $\sigma_i$'s in \textit{loc. cit.} always exist over $D^{n-1}(\epsilon)\times D^\X$ for small enough $\epsilon$.} So, on $D^{n-1}(\epsilon)\X D^\X(\eta)$, we have $j_{!*}\CL\cong f^*j'_{!*}\CL'$, where $j'_{!*}: D^\X(\eta)\hookrightarrow D(\eta)$ and $\CL'$ is some local system on $D^\X(\eta)$ such that $j'_{!*}\CL'$ is not locally constant. As vanishing cycles commute with smooth pullbacks, we get $\phi_f(\CF)_x\cong\phi_{\id}(j'_{!*}\CL')_0\neq0$.
\end{proof}

\begin{lemma}\label{lem_max_liss_open_is_zariski_open}
     Let $X$ be a rigid analytic variety, and $\CF\in\DX$. Then the maximal open on which $\CF$ is locally constant is a Zariski-open.
\end{lemma}

We learned the following argument from Bhargav Bhatt.

\begin{proof}
    Denote by $X_{lis}(\CF)$ (resp. $X_{Zlis}(\CF)$) the maximal open (resp. Zariski-open) in $X$ on which $\CF$ is locally constant. We need to show $X_{lis}(\CF)\subseteq X_{Zlis}(\CF)$. The formation of $(-)_{lis}$ clearly commutes with restricting to open subsets, \textit{i.e.}, $(X_{lis}(\CF)\cap U)=U_{lis}(\CF_U)$ for each open $U$ in $X$. The proof of Theorem 3.5 in \cite{bhatt_six_2022} shows that the same holds for $(-)_{Zlis}$. The conclusion follows.
\end{proof}

\begin{example}\label{ex_SS_compute}
    (1) Let $X$ be a connected smooth rigid analytic variety of dimension $1$. A Zariski-constructible sheaf $\CF$ on $X$ is locally constant except at a discrete subset $S=\{s_i\}$ of classical points. $SS(\CF)$ equals $T_X^*X\cup (\cup_i T^*_{s_i}X)$ if $\CF$ is non-zero on $X-S$, and $(\cup_i T^*_{s_i}X)$ otherwise.\\
    (2) Let $X$ be a smooth rigid analytic variety, and $H$ be a simple normal crossings divisor, \textit{i.e.}, $H=H_1\cup H_2\cup...\cup H_r$ where $H_i$ are smooth divisors such that for each subset $I\subseteq \{1,2,...,r\}$, we have $H_I:=\cap_{i\in I}H_i$ is smooth (possibly empty). Let $\CL$ be a non-zero local system on $X-H$, and $j: X-H\hookrightarrow X$ be the open immersion. Then $SS(j_!\CL)=SS(j_*\CL)=\cup T^*_{H_I}X$, where $I$ ranges through subsets of $\{1,2,...,r\}$ (for $I=\varnothing$, $T^*_{H_I}X:=T_X^*X$). Furthermore, assume $\Lambda=\FF_{\ell^r}$, then $SS(j_{!*}\CL)\subseteq SS(j_!\CL)=SS(j_*\CL)=\cup T^*_{H_I}X$. 
\end{example}

\begin{proof}[Proof of (2)]
    It suffices to show $SS(j_!\CL)=\cup T^*_{H_I}X$, then $SS(j_*\CL)=\cup T^*_{H_I}X$ follows from Proposition \ref{prop_perverseaddSS}.1, and the “Furthermore” part follows from Proposition \ref{prop_perverseaddSS}.2. Let $x\in X(K)$, suppose it is contained in $r'$ components of $H$, $r'\in\{0,1,...,r\}$. We may choose coordinates $\{T_1,...,T_n\}$ in an open neighbourhood $U$ of $x$ in $X$ such that $U$ is identified with $D^n=\Spa(K\langle T_1,...,T_n\rangle)$ and $H\cap U$ is identified with $V(T_1...T_{r'})$. By Abhyankar’s Lemma in the form of \cite[Proposition 2.1]{li_logarithmic_2019}, after restricting to $D^n(\epsilon)$ for a small enough $\epsilon\in|K^\X|$, $\CL$ is trivialised by some Kummer covering in the $\{T_1,...,T_{r'}\}$ coordinates. We have shown that every $x\in X(K)$ has an open neighbourhood $W$ on which we are in the following situation: $W\cong D^n=\Spa(K\langle T_1,...,T_n\rangle)$, $H\cap W=V(T_1...T_{r'})$, and $\CL_W$ is a local system on $U_W\cong D^{n-r}\X(D^r-V(T_1...T_{r'}))$ trivialised by a Kummer covering in the $\{T_1,...,T_{r'}\}$ coordinates. The same covering and descent data defines a local system $\CL'$ on $\AAA^n-V(T_1...T_{r'})$ whose restriction to $D^n$ is isomorphic to $\CL_W$. $(\AAA^n,\CL')$ is algebraisable, so Lemma \ref{lem_comparisonSS} and \cite[\nopp 4.11]{saito_characteristic_2017} imply $SS(j_{W!}\CL_W)\supseteq (\cup T^*_{H_I}X)|_W$. As this is the case at every classical point in $X$ and $\cup T^*_{H_I}X$ is the closure of $(\cup T^*_{H_I}X)(K)$ in $T^*X$, we get $SS(j_!\CL)\supseteq \cup T^*_{H_I}X$.\\\\
    It remains to show $SS(j_!\CL)\subseteq\cup T^*_{H_I}X$. It suffices to show $\{\dddot{j_!\CL}\}\subseteq\cup T^*_{H_I}X$. Let $(x,\xi)\in (T^*X-\cup T^*_{H_I}X)(K)$, and $f$ be a test function at $(x,\xi)$. Necessarily $\xi$ is non-zero, so $f$ is smooth in an open neighbourhood $U$ of $x$ in $X$, and $H\cap U$ is a simple normal crossings divisor relative to $f: U\rightarrow\AAA^1$. Let $j: (U-H)\hookrightarrow U$ be the inclusion, and denote the restriction of $\CL$ to $U-H$ by the same letter. We want to show $\phi_f(j_!\CL)_x=0$. This follows from the same method as in \cite[\nopp XIII, 2.1.11]{SGA7}: by induction on the amplitude we may assume $\CL$ is in degree $0$. By Abhyankar as above, after possibly shrinking $U$ towards $x$, we may assume $\CL$ is trivialised by some Kummer covering $g: Y\rightarrow (U-H)$. This induces a resolution $\CL\isoto (\CL_1\rightarrow\CL_2\rightarrow\cdots)$, where $\CL_1=g_*g^*\CL$, $\CL_2=g_*g^*\mathrm{coker}(\CL\rightarrow\CL_1)$, and so on. It suffices to show $\phi_f(j_!\CL_i)_x=0$ for all $i$. Note each $\CL_i$ becomes constant after $g^*$-pullback. As $g$ is Kummer, it extends to a ramified covering $\bar{g}: \overline{Y}\rightarrow U$ such that $\bar{g}^{-1}(U\cap H)$ is a simple normal crossings divisor relative to $f\bar{g}$. Apply the commutativity of vanishing cycles with proper pushforwards, we reduce to showing $\phi_f(j_!\CL)_x=0$ for $\CL$ constant. By standard dévissage, we further reduce to showing $\phi_f(\text{a constant sheaf on } U\cap H_I)_x=0$, which is clear (Example \ref{ex_phicomp}.1). 
\end{proof}

\begin{remark}\label{rmk_failure_Morse}
    We record a distinction between the real or complex analytic and the rigid analytic settings. Namely, the Microlocal Morse Lemma as in \cite[\nopp 5.4.19]{kashiwara_sheaves_1990} does not hold in the rigid world. Let $\CL$ be a local system on $D$, and consider $R\Gamma(D(r),\CL)$ for $r\in|K^\X|$ increasing to $1$. By \cite[\nopp 10.6.(iv)]{huber_swan_2001}, the ramification of $\CL$ at the “out-pointing” rank-$2$ point of $D(r)$ contributes to $R\Gamma(D(r),\CL)$ (see \cite[\nopp 10.5]{huber_swan_2001} for an explicit computation), while $SS(\CL)=T^*_{D}D$ does not “see” these rank-$2$ points. This distinction roots in the fact that the finite-\etale fundamental group of a rigid annulus is much more complicated than $\hat{\ZZ}(1)$.
\end{remark}

\section{Questions and conjectures}\label{sec_open}

\begin{question}[Smooth pullback of $SS$]\label{que_smoothpull}
    Let $p: Y\rightarrow X$ be a smooth map of smooth rigid analytic varieties, and $\CF\in \DX$. Is $p^\circ SS(\CF)=SS(p^*\CF)$? “$\subseteq$” is known from Lemma \ref{lem_SSpullpush}.(1). Further, one can show $SS(p^*\CF)$ does not contain any covectors with non-zero “vertical” components (reduce to the case $p: X\times D \rightarrow X$, then for any test function $f: X\times D\rightarrow \AAA^1$ at $(x,\xi)\in T^*(X\times D)(K)$ with $\xi(\partial_T)\neq 0$, consider the projection $X\times D\times \AAA^1\rightarrow X\times\AAA^1$ and the function $f-z$ on $X\times D\times \AAA^1$, where $z$ is the coordinate on $\AAA^1$. Use the fact that \etale maps are local isomorphisms at classical points). In particular, the answer is yes when $\mathrm{dim}(X)=1$.
\end{question}

\begin{question}[Comparison of $SS$]\label{que_comparisonSS}
    Let $\CX$ be a smooth finite type scheme over $K$, and $\CF\in D^b_c(X)$. Let $X$ and $\CF^{an}$ be their analytifications. Is $(SS(\CF))^{an}= SS(\CF^{an})$? “$\subseteq$” is known from Lemma \ref{lem_comparisonSS}. The converse can be proved if one can prove the following statement:
    \begin{statement}
        Let $D^n=\Spa(K\langle T_1,...,T_n\rangle)$, $\CF\in D^{b}_{zc}(D^n)$, and $f\in K\langle T_1,...,T_n\rangle$ be such that $f(0)=0$ and $df(0)\neq 0$. If $\phi_f(\CF)_0\neq 0$, then there exists an $N\in\ZZ_{\geq 1}$ such that $\phi_{f_N}(\CF)_0\neq 0$, where $f_N$ denotes $f$ with terms of order $>N$ removed.
    \end{statement}
\end{question}

\begin{question}[Analyticity of $SS$]\label{que_ana}
    Let $X$ be a smooth rigid analytic variety, and $\CF\in \DX$. Is $SS(\CF)$ an analytic closed subset of $T^*X$? We know its base equals $supp(\CF)$ by Lemma \ref{lem_basesupp} so is analytic. We can also show that the maximal open subset of $X$ on which $\CF$ is lisse is Zariski open. Together with Proposition \ref{prop_0seclisse} (assuming $\Lambda=\FF_{l^r}$), this implies that the base of $SS(\CF)-X$ is analytic.
\end{question}

\begin{conjecture}
    Let $X$ be a smooth rigid analytic variety, and $\CF\in\DX$. Then $SS(\CF)=\supp\,\mu hom(\CF,\CF)$. Note this would give positive answers to Questions \ref{que_smoothpull} and \ref{que_ana}.
\end{conjecture}

The following conjecture makes precise a conjecture of Bhatt and Hansen (\cite[last sentence]{bhatt_six_2022}). Note that Proposition \ref{prop_perverseaddSS} and Example \ref{ex_SS_compute} imply the conjecture is true (assuming $\Lambda=\FF_{\ell^r}$) on smooth surfaces after possibly removing a discrete set of classical points.

\begin{conjecture}
    Let $X$ be a smooth rigid analytic variety, and $\CF\in\DX$. Then $SS(\CF)$ is an analytic Lagrangian subspace of $T^*X$, \textit{i.e.}, it is an analytic closed subset and is Lagrangian on the regular locus. In particular, if $X$ is connected, then each irreducible component of $SS(\CF)$ is of dimension $\dim(X)$. 
\end{conjecture}

\section{Appendix: $\infty$-categorical characterisations of monodromic sheaves}
We give an $\infty$-categorical characterisation of monodromic sheaves. We are grateful to Bhargav Bhatt for suggesting this to us and for discussion on its proof. A similar characterisation is known in the algebraic setting using representation theoretic methods, see \cite{eteve_free_2025}. The set-up is as in §\ref{sec_intro}, we further assume $p>0$, this is because the formalism of v-stacks in \cite{scholze_etale_2022} that we will use is developed under this assumption.\\

Let $E\rightarrow X$ be a cone over a rigid analytic variety over $K$ (see §\ref{sec_monsheav}), with no assumptions on $X$. Consider the pro-system of $n$-th power coverings $\Gm(n)(=\Gm)\rightarrow\Gm, z\mapsto z^n, n\in\ZZ_{\geq 1}$. The limit $\widetilde{\Gm}$ exists as an adic space, and is in fact a perfectoid group. Each $\Gm(n)$ acts on $E$ via $\Gm(n)\rightarrow\Gm$ and the scaling action of $\Gm$ on $E$. Denote by $E/\Gm(n)$ the quotient stack. More precisely, we take the diamond $E^\diamondsuit$ (resp. $\Gm(n)^\diamondsuit$) associated to $E$ (resp. $\Gm$) (\cite[\nopp 15.5]{scholze_etale_2022}) and form the small v-stack $E^\diamondsuit/\Gm(n)^\diamondsuit$. We will drop $^\diamondsuit$ in the notations for convenience. Similarly, we have the small v-stack $E/\widetilde{\Gm}$.\\

For a small v-stack $Y$, the stable $\infty$-category $\mathcal{D}_{\et}(Y)$ is defined in \cite[\nopp 14.13, 17.1]{scholze_etale_2022}. By \cite[\nopp 14.14]{scholze_etale_2022}, $\mathcal{D}_{\et}(-)$ is a v-sheaf (valued in $\infty$-categories) on the category of small v-stacks. For a locally spatial diamond $Y$, $\mathcal{D}_{\et}(Y)$ is the left completion of $\mathcal{D}(Y_{\et})$, the latter denotes the stable $\infty$-category of sheaves on $Y$ with the \etale topology (\cite[\nopp 14.15]{scholze_etale_2022}). If $Y$ is (the diamond associated to) a rigid analytic variety, then $\mathcal{D}_{\et}(Y)\simeq\mathcal{D}(Y_{\et})$ as $\mathcal{D}(Y_{\et})$ is left complete (\textit{cf.} \cite[\nopp 3.29]{bhatt_six_2022}). We denote by $\cDb(E/\Gm(n))$ (resp. $\mathcal{D}^b(E/\Gm(n))$) (resp. $\cDbzc(E/\Gm(n))$) the $\infty$-fullsubcategory of $\mathcal{D}_{\et}(E/\Gm(n))$ consisting of those $\CF$ such that $g^*\CF$ is in $\cDb(E)$ (resp. $\mathcal{D}^b(E)$) (resp. $\cDbzc(E)$), where $g: E\rightarrow E/\Gm(n)$ is the quotient map, and $\cDb(E)$ (resp. $\mathcal{D}^b(E)$) (resp. $\cDbzc(E)$) is the usual $\infty$-category of locally bounded \etale sheaves (resp. bounded \etale sheaves) (resp. locally bounded Zariski-constructible \etale sheaves) on the rigid analytic variety $E$ (we are implicitly using \cite[\nopp 15.6]{scholze_etale_2022}). Similarly we have $\cDb(E/\widetilde\Gm)$ (resp. $\mathcal{D}^b(E/\widetilde\Gm)$) (resp. $\cDbzc(E/\widetilde\Gm)$).\\

Consider the ind-system $\{\mathcal{D}^b(E/\Gm(n))\}_{n\in\ZZ_{\geq1}}$ with transition maps induced by pullbacks along quotient maps $E/\Gm(m)\rightarrow E/\Gm(n)$ for $n$ dividing $m$. Let $\varinjlim_n\mathcal{D}^b(E/\Gm(n))$ be its colimit in $\mathrm{Cat}_{\infty}$. Note then $RHom_{\varinjlim\mathcal{D}^b_{zc}(E/\Gm(n))}\simeq\varinjlim RHom_{\mathcal{D}^b_{zc}(E/\Gm(n))}$.\\

We have the functors $\varinjlim_n\mathcal{D}^b(E/\Gm(n))\rightarrow\cDb(E/\widetilde{\Gm})\rightarrow\cDb(E)$ induced by pullbacks along quotient maps.

\begin{proposition}
    Let $E\rightarrow X$ be a cone over a rigid analytic variety. Then:\\
    (1) the above functors are fully faithful embeddings: $\varinjlim_n\mathcal{D}^b(E/\Gm(n))\hookrightarrow\cDb(E/\widetilde{\Gm})\hookrightarrow\cDb(E)$. These restrict to fully faithful embeddings: $\varinjlim_n\mathcal{D}^b_{zc}(E/\Gm(n))\hookrightarrow\cDbzc(E/\widetilde{\Gm})\hookrightarrow\cDbzc(E)$.\\
    (2) The essential image of $\cDbzc(E/\widetilde{\Gm})\hookrightarrow\cDbzc(E)$ is $\mathcal{D}_{mon}(E)$.\\
    (3) For $X$ quasi-compact, $\varinjlim_n\mathcal{D}^b_{zc}(E/\Gm(n))\hookrightarrow\mathcal{D}^b_{zc}(E/\widetilde{\Gm})$ is an equivalence.
\end{proposition}

Here $\mathcal{D}_{mon}(E)$ denotes the $\infty$-fullsubcategory of $\cDbzc(E)$ whose objects are $\CF\in \cDbzc(E)$ such that there exists an isomorphism $\theta_\lambda^*\CF\isoto \CF$ for all $\lambda\in K^\X$. It follows from this Proposition that $\mathcal{D}_{mon}(E_U)$ is the stackification of $\varinjlim_n\mathcal{D}^b_{zc}(E_U/\Gm(n))$ for $U$ ranging over open subsets in $X$ with the usual topology.

\begin{proof}
    We make two preliminary observations. Denote by $\mathcal{D}_{MON}(E)$ the $\infty$-fullsubcategory in $\cDb(E)$ consisting of $\CF\in \cDb(E)$ such that there exists an isomorphism $\theta_\lambda^*\CF\isoto \CF$ for all $\lambda\in K^\X$. Then, (i) $\cDb(E/\widetilde{\Gm})$ lands in $\mathcal{D}_{MON}(E)$: for each $\lambda\in K^\X$, any choice of a pro-system of $n$-th roots of $\lambda$ gives an isomorphism $\theta_\lambda^*\CF\isoto \CF$ induced by the $\widetilde{\Gm}$-equivariance structure; (ii) for $X$ quasi-compact, $\mathcal{D}_{MON}(E)\subseteq\mathcal{D}^b(E)$: this is because every $\CF\in\mathcal{D}_{MON}(E)$ is bounded in some open neighbourhood $U$ of the $0$-section in $E$, and every point in $E-U$ can be moved into $U$ by a scaling. We now proceed with the proof.\\\\
    \underline{Step 1}. We show $\cDb(E/\widetilde{\Gm})\rightarrow\cDb(E)$ is fully faithful.\\\\ 
    Consider the following diagram of small v-stacks, where $f$ and $g$ are the quotient maps:
\[\begin{tikzcd}
	{E\X_K\Gm} & E \\
	{E/\mu} & {E/\widetilde{\Gm}}
	\arrow["{g'}", from=1-1, to=1-2]
	\arrow["{f'}", from=1-1, to=2-1]
	\arrow["\lrcorner"{anchor=center, pos=0.125}, draw=none, from=1-1, to=2-2]
	\arrow["f", from=1-2, to=2-2]
	\arrow["g"', from=2-1, to=2-2]
\end{tikzcd}\]  
    Let $\CF, \CG\in \cDb(E/\widetilde{\Gm})$. We need to show $RHom(\CF,\CG)\isoto RHom(f^*\CF,f^*\CG)$. As $RHom(f^*\CF,f^*\CG)\simeq RHom(\CF,f_*f^*\CG)$, it suffices to show $\CG\isoto f_*f^*\CG$. The question being local, we may assume $X$ is quasi-compact and so $\CG\in\mathcal{D}^b(E/\widetilde{\Gm})$. Induction on the amplitude then reduces to the case $\CG\in\Sh(E/\widetilde{\Gm})$ The isomorphism can be checked on a v-cover (\cite[\nopp 14.10]{scholze_etale_2022}), in particular $g: E/\mu\rightarrow E/\widetilde{\Gm}$. Using \cite[\nopp 23.15, 24.4]{scholze_etale_2022}, one checks that $g$ is $\ell$-cohomologically smooth. So $g^*f_*f^*\CG\isoto f'_*g'^*f^*\CG$ by \cite[\nopp 23.16.(ii)]{scholze_etale_2022}, so it suffices to show: for every $A\in\Sh(E/\mu)$, we have $A\isoto f'_*f'^*A$. Consider the following resolutions:
\[\begin{tikzcd}
	{\cdots E\X_K\widetilde{\Gm}\X_K\mu} & {E\X_K\widetilde{\Gm}} & {E\X_K\Gm} \\
	{\cdots E\X_K\mu} & E & {E/\mu}
	\arrow[shift right, from=1-1, to=1-2]
	\arrow[shift left, from=1-1, to=1-2]
	\arrow["{a_1}", curve={height=18pt}, dotted, from=1-1, to=1-3]
	\arrow[from=1-1, to=2-1]
	\arrow["{a_0}", from=1-2, to=1-3]
	\arrow[from=1-2, to=2-2]
	\arrow["\lrcorner"{anchor=center, pos=0.125}, draw=none, from=1-2, to=2-3]
	\arrow["{f'}", from=1-3, to=2-3]
	\arrow[shift right, from=2-1, to=2-2]
	\arrow[shift left, from=2-1, to=2-2]
	\arrow["{a_1}", curve={height=18pt}, dotted, from=2-1, to=2-3]
	\arrow["{a_0}", from=2-2, to=2-3]
\end{tikzcd}\]
Denote $(E\X_K\widetilde{\Gm}\X_K\mu\X_K\cdots\X_K\mu)\rightarrow (E\X_K\mu\X_K\cdots\X_K\mu)$ ($n$-copies of $\mu$) by $f'_n: E_n\X_K\widetilde{\Gm}\rightarrow E_n$. First note that it suffices to show the \underline{Claim}: $B\isoto f'_{n*}f'^*_nB$ for $B\in\Sh(E_n)$ and each $n$. Indeed, accepting this claim, we have, for every $C\in\mathcal{D}^+(E/\mu)$:
\begin{align*}
RHom(C,f'_*f'^*A)&\simeq RHom(\varprojlim C_n,f'_*f'^*\varprojlim A_n)\\
&\simeq RHom(f'^*\varprojlim C_n,f'^* \varprojlim A_n)\\
&\simeq RHom(\varprojlim f'^*_nC_n, \varprojlim f'^*_nA_n)\\
&\simeq \varprojlim RHom(f'^*_nC_n, f'^*_nA_n)\\
&\simeq \varprojlim RHom(C_n, f'_{n*}f'^*_nA_n)\\
&\simeq \varprojlim RHom(C_n, A_n)\\
&\simeq RHom(\varprojlim C_n, \varprojlim A_n)\simeq RHom(C, A).
   \end{align*}
Here, in the first step, we used $\mathcal{D}_{\et}(E/\mu)\simeq \varprojlim \mathcal{D}_{\et}(E_n)$ (which follow from the fact that $\mathcal{D}_{\et}(-)$ is a v-stack) to write $A$ as $\varprojlim A_n$, where $A_n=a_n^*A$, similarly for $C$; in the third step we used that the pullback on the simplicial resolution is computed termwise; in the fourth and seventh steps we used that $RHom$ of a limit is the limit of $RHom$'s.\\\\
It remains to show the Claim. This follows from a more general statement.
\begin{lemma*}
    Let $Y$ be a locally spatial diamond which admits a cover by open subdiamonds of the form $U=\varprojlim_i U_i$, where $i$ is indexed by positive integers ordered by divisibility, $\{U_i\}$ are (diamonds associated to) quasi-compact quasi-separated (qcqs) rigid analytic varieties over $K$ with qcqs transition maps. Let $\pi: Y\X_K\widetilde{\Gm}\rightarrow Y$ be the projection. Then, for every $A\in\mathcal{D}^{(b)}(Y)$, we have $A\isoto \pi_*\pi^*A$. Here $\mathcal{D}^{(b)}(Y)$ denotes the $\infty$-fullsubcategory of $\mathcal{D}_{\et}(Y)$ consisting of those objects that are bounded locally on $Y$.
\end{lemma*}
\begin{proof}
    The question being local, we may assume $Y=\varprojlim_i Y_i$ is of the same form as $U$. Note $Y$ is then qcqs and the maps $Y\rightarrow Y_i$ are also qcqs, and the topos $Y_{\et}\!\!\tilde{}$ is the limit of the fibred topoi $Y_{i,\et}\!\!\!\!\!\tilde{}$ (\cite[\nopp 11.22, 11.23]{scholze_etale_2022}). We have $A\in\mathcal{D}^{b}(Y)$ and induction on the amplitude further reduces to the case $A\in\Sh(Y_\et)$. Choose a decreasing sequence $r_j\in\sqrt{|K^\X|}$, $j\in\ZZ_{\geq 1}$ converging to $0$, consider the cover of $\widetilde{\Gm}$ by qcqs open subsets $\widetilde{D}_j$, where $\widetilde{D}_j:=\varprojlim_m D(^m\!\!\sqrt{r_j},^m\!\!\!\!\sqrt{1/r_j})$, $m\in\ZZ_{\geq 1}$ is the limit of the annuli centred at $0$ of radius $(^m\!\!\sqrt{r_j},^m\!\!\!\!\sqrt{1/r_j})$. (In other words, $\widetilde{D}_j$ is the pre-image of $D(r_j,1/r_j)$ under $\widetilde{\Gm}\rightarrow\Gm$.) We have $\pi_*\pi^*A\simeq R\varprojlim_j \pi_{j*}\pi^*_jA$, where $\pi_j$ is the restriction of $\pi$ to $\widetilde{D}_j$. It suffices to show $A\isoto \pi_{j*}\pi^*_jA$ for each $j$.\\\\ 
    For a fixed $j$, we abbreviated $\widetilde{D}_j=\varprojlim_m D(^m\!\!\sqrt{r_j},^m\!\!\!\!\sqrt{1/r_j})$ as $\widetilde{D}=\varprojlim D_m$, and $\pi_j$ as $\pi$. Consider the following Cartesian diagram:
\[\begin{tikzcd}
	{Y\X_K\widetilde{D}} & {Y_{n}\X_K \widetilde{D}} \\
	Y & {Y_{n}}
	\arrow[from=1-1, to=1-2]
	\arrow["{f_{n}}"', curve={height=12pt}, dotted, from=1-1, to=1-2]
	\arrow["\pi"', from=1-1, to=2-1]
	\arrow["{\pi_{n}}", from=1-2, to=2-2]
	\arrow[from=2-1, to=2-2]
	\arrow["{f_{n}}"', curve={height=12pt}, dotted, from=2-1, to=2-2]
\end{tikzcd}\]
By \cite[\nopp VI.8.7.4]{SGA4}, $R^i\pi_*\pi^*A\simeq \varinjlim_nf_n^*R^i\pi_{n*}$$^o\!f_{n*}\pi^*A\simeq \varinjlim_nf_n^*R^i\pi_{n*}\pi_n^*$$^o\!f_{n*}A$. Here we have explicated the derived functors (and will continue until the end of the proof of this Lemma), and in the second step we used qcqs base change (\cite[\nopp 17.6]{scholze_etale_2022}). It suffices to show ($\ast$) $R^i\pi_{n*}\pi_n^* B$ is $0$ unless $i=0$ and $B\isoto R^0\pi_{n*}\pi_n^* B$, for $B\in\Sh(Y_{n,\et})$. Indeed, then $\varinjlim_nf_n^*R^i\pi_{n*}\pi_n^*$$^o\!f_{n*}A\simeq \varinjlim_nf_n^*$$^o\!f_{n*}A\simeq A$ by \cite[\nopp VI.8.5.2]{SGA4}.\\\\
It remains to show ($\ast$). Again by \cite[\nopp VI.8.7.5]{SGA4}, $R^i\pi_{n*}\pi_n^* B\simeq \varinjlim_m R^i\pi^m_{n*}\pi_n^{m*} B$, where $\pi^m_n$ is the map $Y_n\X_KD_m\rightarrow Y_n$ induced by $D_m\rightarrow\pt$. We have finally reduced to the rigid world. Using \cite[\nopp 3.2.4, 3.9.1.(II)]{huber_etale_1996}, one can show that $R^i\pi^m_{n*}\pi_n^{m*} B$ is $B$ for $i=0$ with identify transition maps, $B(-1)$ for $i=1$ with transition maps being multiplications, and $0$ otherwise. It follows that $\varinjlim_m R^i\pi^m_{n*}\pi_n^{m*} B$ is $B$ for $i=0$ and $0$ otherwise\footnote{Alternatively, one can prove this using the same method as in \cite[176, 177, 178]{huber_etale_1996}.}.
\end{proof}
\noindent\underline{Step 2}. We show $\varinjlim_n\mathcal{D}^b(E/\Gm(n))\rightarrow\cDb(E)$ is fully faithful. This completes the proof of (1).\\\\
    Let $\CF,\CG\in\mathcal{D}^b(E/\Gm(n))$. Denote the quotient maps $E\rightarrow E/\Gm(n)$ and $E/\Gm(m)\rightarrow E/\Gm(n)$ (for $m$ divisible by $n$) by $g$ and $g_{m}$ respectively. We need to show $\varinjlim_{n|m}RHom(g_m^*\CF,g_m^*\CG)\isoto RHom(g^*\CF,g^*\CG)$. One checks that $g$ is $\ell$-cohomologically smooth, so $\RHom(g^*\CF,g^*\CG)\isoto g^*\RHom(\CF,\CG)$ (\cite[\nopp 23.17]{scholze_etale_2022}). We claim that we also have $\RHom(g_m^*\CF,g_m^*\CG)\isoto g_m^*\RHom(\CF,\CG)$. To see this, one can either note that $g_m$ is $\ell$-cohomologically smooth in the more general sense of \cite[\nopp 1.3.(ii)]{gulotta_enhanced_2022} and apply \cite[\nopp 4.14]{gulotta_enhanced_2022}, or one can argue directly, as in the displayed formula in \underline{Step 1}, by reducing to the corresponding statement for each term in the following simplicial resolution. We omit the details.   
\[\begin{tikzcd}
	{\cdots E\X_K\Gm(m)} & E & {E/\Gm(m)} \\
	{\cdots E\X_K\Gm(n)} & E & {E/\Gm(n)}
	\arrow[shift right, from=1-1, to=1-2]
	\arrow[shift left, from=1-1, to=1-2]
	\arrow["{a_1}", curve={height=18pt}, dotted, from=1-1, to=1-3]
	\arrow[from=1-1, to=2-1]
	\arrow["{a_0=g}", from=1-2, to=1-3]
	\arrow[equals, from=1-2, to=2-2]
	\arrow["{g_m}", from=1-3, to=2-3]
	\arrow[shift right, from=2-1, to=2-2]
	\arrow[shift left, from=2-1, to=2-2]
	\arrow["{a_1}"', curve={height=18pt}, dotted, from=2-1, to=2-3]
	\arrow["{a_0=g}", from=2-2, to=2-3]
\end{tikzcd}\]
    We need to show $\varinjlim R\Gamma(E/\Gm(m),g_m^*\RHom(\CF,\CG))\simeq R\Gamma(E,g^*\RHom(\CF,\CG))$. As $\RHom(\CF,\CG)\in \mathcal{D}_{\et}^+(E/\Gm(n))$, we have:
    \begin{align*}
&\varinjlim R\Gamma(E/\Gm(m),g_m^*\RHom(\CF,\CG))\\
&\simeq \varinjlim_m\varprojlim_{\Delta} R\Gamma(E\X_K\Gm(m)\X_K\cdots\X_K\Gm(m),a_i^*g_m^*\RHom(\CF,\CG))\\
&\simeq \varprojlim_{\Delta}\varinjlim_m R\Gamma(E\X_K\Gm(m)\X_K\cdots\X_K\Gm(m),a_i^*g_m^*\RHom(\CF,\CG)).
   \end{align*}
    Denote $a_0^*\RHom(\CF,\CG)$ by $A$. Then $a_i^*g_m^*\RHom(\CF,\CG)\simeq p^*A$ for $p: E\X_K\Gm(m)\X_K\cdots\X_K\Gm(m)\rightarrow E$ the projection. By smooth base change, $R\Gamma(E\X_K\Gm(m)\X_K\cdots\X_K\Gm(m),p^*A)\simeq R\Gamma(\Gm(m)\X_K\cdots\X_K\Gm(m),\underline{R\Gamma(E,A)})$. One may then use Lemma \ref{lem_coh_purity_consequence} and induction to see that the last term is isomorphic to $R\Gamma(\Gm(m)\X_K\cdots\X_K\Gm(m),\Lambda)\otimes_\Lambda \underline{R\Gamma(E,A)}$. Consequently:
    \begin{align*}
&\varprojlim_{\Delta}\varinjlim_m R\Gamma(E\X_K\Gm(m)\X_K\cdots\X_K\Gm(m),a_i^*g_m^*\RHom(\CF,\CG))\\
&\simeq\varprojlim_{\Delta}\varinjlim_m (R\Gamma(\Gm(m)\X_K\cdots\X_K\Gm(m),\Lambda)\otimes_\Lambda \underline{R\Gamma(E,A)})\\
&\simeq\varprojlim_{\Delta}\underline{R\Gamma(E,A)}\simeq R\Gamma(E,A).
   \end{align*}
    As $R\Gamma(E,A)\simeq R\Gamma(E,g^*\RHom(\CF,\CG))\simeq R\Gamma(E,\RHom(g^*\CF,g^*\CG))=RHom(g^*\CF,g^*\CG)$, this finishes \underline{Step 2}.\\\\  
    \underline{Step 3}. We show (2) and (3).\\\\  
    Denote the essential image of $\cDbzc(E/\widetilde{\Gm})\hookrightarrow\cDbzc(E)$ (resp. $\varinjlim_n\mathcal{D}^b_{zc}(E/\Gm(n))\hookrightarrow\cDbzc(E)$) by $\CCC$ (resp. $\CCC'$). It is clear that $\CCC\subseteq\mathcal{D}_{mon}(E)$. Let $\CF\in\mathcal{D}_{mon}(E)$, to show $\CF$ lies in $\CCC$, it suffices to show this locally on $X$ because $\cDbzc(E_U/\widetilde{\Gm})$ is a stack for $U$ ranging over open subsets in $X$ with the usual topology. In other words, we have reduced to showing (2) and (3) for $X$ quasi-compact. Note $\CCC'$, $\CCC$ and $\mathcal{D}_{mon}(E)$ are stable under truncations and taking cones: this is clear for $\CCC'$ (using $RHom_{\varinjlim\mathcal{D}^b_{zc}(E/\Gm(n))}\simeq\varinjlim RHom_{\mathcal{D}^b_{zc}(E/\Gm(n))}$) and $\CCC$, for $\mathcal{D}_{mon}(E)$ this follows from Proposition \ref{prop_fundamental_of_mon_qc}. So, by induction on the amplitude, we may assume $\CF\in\Sh_{mon}(E)$, thus reducing (2) and (3) to the following statement: for $X$ quasi-compact, every $\CF\in \Sh_{mon}(E)$ comes from $\Sh(E/\Gm(n))$ for some $n$ depending on $\CF$.\\\\  
    It is well-known that to endow an object $\CG\in\Sh(Y)$ with a $G$-equivariant structure with respect to a connected algebraic group $G$ acting on a space $Y$, it suffices to provide an isomorphism $a^*\CG\isoto \pr^*\CG$, where $a$ (resp. $\pr$) is the action (resp. projection) map $G\X Y\rightarrow Y$ (see, for example, \cite[\nopp 12.6]{jantzen_nilpotent_2004}). In our situation, let $n$ be such that $\theta(n)^*\CF$ is constant on $\Gm\X v$ for every $v\in E$ (notations as in §\ref{sec_monsheav}), then the composition $\theta(n)^*\CF\rightarrow \pr^!\pr_!(\theta(n)^*\CF)\rightarrow \CH^0(\pr^!\pr_!(\theta(n)^*\CF))$ gives a desired  isomorphism $\theta(n)^*\CF\isoto \pr^*\CF$. This is the same argument as the second paragraph in the proof of Proposition \ref{prop_fundamental_of_mon_qc}. We refer to that paragraph for details.
\end{proof}

We can now generalise the definitions of monodromic sheaves and their monodromy (Definitions \ref{def_monsheav}, \ref{def_equimonodromy}) to the case over a general rigid analytic variety without quasi-separatedness assumptions.

\begin{definition}\label{def_infinity_monodromic}
    Let $E\rightarrow X$ be a cone over a rigid analytic variety. A sheaf $\CF\in\cDbzc(E)$ is called \underline{monodromic} if there exists an isomorphism $\theta_\lambda^*\CF\isoto \CF$ for all $\lambda\in K^\X$. The $\infty$-fullsubcategory in $\cDb(E)$ of monodromic sheaves is denoted by  $\mathcal{D}_{mon}(E)$. Equivalently, $\mathcal{D}_{mon}(E)$ is the essential image of $\cDbzc(E/\widetilde{\Gm})\hookrightarrow\cDbzc(E)$. For $\CF\in\mathcal{D}_{mon}(E)$, its \underline{monodromy} is the $\mu$-action induced by $\mu\hookrightarrow \widetilde{\Gm}(K)$ and the $\widetilde{\Gm}$-equivariance structure on $\CF$.
\end{definition}


\printbibliography

\Addresses

\end{document}